\documentclass[final, reqno]{amsart}

\usepackage{amsfonts,amssymb,amsmath,amsthm}
\usepackage{mathtools,mathrsfs,mathabx,bbm,ulem}
\usepackage[utf8]{inputenc}
\usepackage{enumerate}

\usepackage[all]{xy}

\newcommand{\al}{\alpha}
\newcommand{\bt}{\beta}
\newcommand{\g}{\gamma}
\newcommand{\de}{\delta}

\newcommand{\ka}{\kappa}

\newcommand{\la}{{\alpha_0}}

\newcommand{\s}{\sigma}

\newcommand{\G}{\Gamma}
\newcommand{\D}{\Delta}

\newcommand{\Om}{\Omega}

\newcommand{\bbR}{\mathbb R} 

\newcommand{\bbP}{\mathbb P}
\newcommand{\bbE}{\mathbb E}

\renewcommand{\1}{\mathbbm{1}}

\newcommand{\cal}{\mathcal}
\newcommand{\cA}{\cal A}
\newcommand{\cB}{\cal B}

\newcommand{\cD}{\cal D}
\newcommand{\cE}{\cal E}
\newcommand{\cF}{\cal F}
\newcommand{\cG}{\cal G}
\newcommand{\cH}{\cal H}
\newcommand{\cI}{\cal I}

\newcommand{\cL}{\cal L}
\newcommand{\cM}{\cal M}

\newcommand{\cP}{\cal P}

\newcommand{\cS}{\cal S}

\newcommand{\cV}{\cal V}
\newcommand{\cW}{\cal W}

\newcommand{\wt}{\widetilde}
\newcommand{\wh}{\widehat}

\newtheorem{theorem}{Theorem}[section]
\newtheorem{lemma}[theorem]{Lemma}
\newtheorem{proposition}[theorem]{Proposition}

\newtheorem{example}{Example}

\theoremstyle{definition}
\newtheorem{definition}[theorem]{Definition}
\newtheorem{assumption}[theorem]{Assumptions}

\theoremstyle{remark}
\newtheorem{remark}{Remark}

\newcommand{\be}{\begin{equation}}
\newcommand{\ee}{\end{equation}}

\newcommand{\bea}{\be\begin{aligned}}
\newcommand{\eea}{\end{aligned}\ee}

\newcommand{\bal}{\begin{aligned}}
\newcommand{\eal}{\end{aligned}}

\newcommand \ba {\begin{array}}
\newcommand \ea {\end{array}}

\newcommand{\abs}[1]{\left\lvert{#1}\right\rvert}

\newcommand{\norm}[1]{\left\lVert{#1}\right\rVert}

\newcommand{\bra}[1]{\langle{#1}\rangle}

\newcommand{\conv}[2]{\xrightarrow[#1\to#2]{}}

\newcommand\grad{\nabla}
\newcommand\Div{\mathrm{div}}
\newcommand\del{\partial}

\newcommand\Ent{\mathrm{Ent}}

\numberwithin{equation}{section}

\title[averaging of stochastic flows and convergence of Dirichlet forms]{An averaging principle for stochastic flows and convergence of non-symmetric Dirichlet forms}



\author{Florent Barret}
\address{Laboratoire MODAL'X, Université Paris Nanterre, 200, avenue de la République - 92001 Nanterre cedex - France}
\email{fbarret@parisnanterre.fr}
\urladdr{http://barret.perso.math.cnrs.fr/}

\author{Olivier Raimond}
\address{Laboratoire MODAL'X, Université Paris Nanterre, 200, avenue de la République - 92001 Nanterre cedex - France}
\email{oraimond@parisnanterre.fr}
\urladdr{http://raimond.perso.math.cnrs.fr/}

\begin{document}
\begin{abstract}
We study diffusion processes and stochastic flows which are time-changed random perturbations of a deterministic flow on a manifold. Using non-symmetric Dirichlet forms and their convergence in a sense close to the Mosco-convergence, we prove that, as the deterministic flow is accelerated, the diffusion process converges in law to a diffusion defined on a different space. This averaging principle also holds at the level of the flows.

Our contributions in this article include: 
\begin{itemize}
\item a proof of an original averaging principle for stochastic flows of kernels;
\item the definition and study of a convergence of sequences of non-symmetric bilinear forms defined on different spaces;
\item the study of weighted Sobolev spaces on metric graphs or ``books''.
\end{itemize}
\end{abstract}

\maketitle

\section*{Introduction}
Let $X^\kappa$ be a diffusion process on a Riemannian manifold $M$ with generator $A^{\kappa}=A+\kappa V$, where $A$ is a second order differential operator, $V$ is a vector field on $M$ and $\kappa$ is a large positive parameter.
The diffusion $X^\kappa$, after an appropriate time change, is a random perturbation of the dynamical system $\frac{dx_t}{dt}=V(x_t)$.

Averaging principles for such diffusions $X^\kappa$ in $\mathbb{R}^d$ have been studied by Freidlin and Wentzell in \cite{freidlin.wentzell12}. They were mostly interested to the particular case where $d=2$ and where the vector field is given by $V=(-\partial_2 H,\partial_1 H)$, with $H:\mathbb{R}^2\to\mathbb{R}$ (sometimes called an Hamiltonian) a sufficiently regular function. More precisely, they first construct a mapping $\pi:\mathbb{R}\to G$, with $G$ a metric graph. Each point of $G$ corresponds to an orbit of $V$ (or equivalently to a connected component of a level set of $H$) and for a generic point $x\in \mathbb{R}^2$, $\pi(x)=(H(x),i)$ where the index $i$ labels the orbits of $V$.
Then they show that as $\kappa\to \infty$, the process $\pi(X^\kappa)$ converges in distribution towards a diffusion on the metric graph $G$.
The scheme of their proof was, after having checked that the family of the distributions of $\pi(X^\kappa)$ is tight, to prove that every limit point for this family solves a well-posed martingale problem, whose unique solution is a diffusion on $G$. 

An alternative proof for this result was proposed by Barret and von Renesse in \cite{BvR}. 
The main tool used in their proof is the convergence of non-symmetric Dirichlet forms.
The assumptions (on the Hamiltonian $H$ and on the generator $A$) of the theorem stated in \cite{BvR}  are weaker than the ones in \cite{freidlin.wentzell12}.
But there is an inaccuracy in their proof (more precisely the proof of Lemma 3.1 in \cite{BvR} is not correct).

\medskip
In this paper, we will follow (and correct) the approach initiated by Barret and von Renesse. 
Our results will permit to state averaging principles in higher dimension and at the level of flows.   
Let us explain this in more details. 

We will suppose that the vector field $V$ is such that 
\begin{itemize}
\item $V$ is complete and generates a flow $\phi$;
\item the flow $\phi$ has an invariant measure $m$ having a $C^1$ positive density with respect to the volume measure on $M$;
\item there is an increasing sequence $(M_n)$ of compact sets, invariant for $\phi$, and such that $M=\bigcup_n M_n$.
\end{itemize}
The flow $\phi$ generated by $V$ can be viewed as a shear flow.
We denote by $\mathcal{P}(M)$ the set of Borel probability measures on $M$, equipped with the topology of narrow convergence. 
Then (see Proposition \ref{prop:P}) there is a measurable map $P:M\to\mathcal{P}(M)$ such that for $m$-almost all $x\in M$, $P(x)$ is a probability measure on $M$, ergodic for $\phi$ and with $x\in \hbox{Supp}(P(x))$. 

In this introduction, for sake of simplicity, we will suppose that $M=\mathbb{R}^d$, $A=\frac{1}{2}\Delta$ is the generator of a Brownian motion on $M$ and $m$ is the Lebesgue measure on $M$.
Then, the process $X^{\kappa}$ can be constructed as a solution of a stochastic differential equation (SDE)
\begin{equation}\label{eq:1}
dX^{\kappa}_t=\kappa V(X^{\kappa}_t)dt+dB_t
\end{equation}
with $(B_t)_t$ a Brownian motion in $\mathbb{R}^d$. For large $\kappa$, the process $Y^{\kappa}_t:=X^{\kappa}_{t/\kappa}$ is a random perturbation of the flow $\phi$. Indeed, $Y^{\kappa}$ satisfies the SDE:
\begin{equation}\label{eq:1bis}
d Y^{\kappa}_t=V(Y^{\kappa}_t)dt+\kappa^{-1/2}dB^\kappa_t,
\end{equation}
where $B^\kappa_t=\sqrt{\kappa}B_{t/\kappa}$ is also a Brownian motion in $\mathbb{R}^d$.

Suppose also that one is able to construct a metric space $\wt M$, a continuous mapping $\pi:M\to\wt M$ and a measurable mapping $p:\wt M\to\mathcal{P}(M)$
such that for $m$-almost all $x\in M$, $P(x)=p\circ \pi(x)$.
Then the drift term $\kappa V$ (which explodes as $\kappa\to\infty$ in \eqref{eq:1}) disappears in the SDE satisfied by the $\wt M$-valued process  $\wt X^{\kappa}_t:=\pi(X^{\kappa}_t)$.
Note that if $P$ is continuous, one can simply take $\pi=P$ and for $\wt M$, the set of ergodic probability measures.
Our first main result is Theorem \ref{th:main} in Section \ref{sec:avg}, whose statement is the convergence in law of $\wt X^{\kappa}$ as $\kappa\to\infty$ towards a diffusion process $\wt X$ in $\wt M$.

Theorem \ref{th:main} extends Theorem 2.2 in \cite{freidlin.wentzell12}, Chap. 8, where Freidlin and Wentzell consider the case $M=\mathbb{R}^2$ and $V=(-\partial_2H,\partial_1H)$, with $H:\mathbb{R}^2\to\mathbb{R}$ a regular function. Theorem \ref{th:main} also extends, corrects and gives a correct formulation of Theorem 3.11 in \cite{BvR}.


In \cite{funaki1993}, Funaki and Nagai also studied a diffusion $X^\kappa$  solution to SDE \eqref{eq:1}, but defined on a manifold. They were also concerned with the asymptotics of $X^\kappa$ as $\kappa\to\infty$. Assuming that the vector field $V$ has a regular, compact with no boundary, submanifold $M$ of asymptotically stable fixed points, 
they prove the convergence of $X^\kappa$, as $\kappa\to\infty$, towards a diffusion on $M$ via a martingale problem. 
Their framework is however different to the one of this article: 
in \cite{funaki1993}, the vector fields are typically gradient vector fields and in our setting,  they are typically divergent free.

\medskip
Let us now consider $n$ particles $X^{(n),\kappa}=(X^{1,\kappa},\dots, X^{n,\kappa})$ in a turbulent fluid with a shear flow generated by $V$. 
For some $\eta >0$ (a viscosity parameter),  $X^{(n),\kappa}$  solves an SDE on $(\mathbb{R}^d)^n$: for $1\leq k\leq n$
\begin{align}\label{eq:2}
d X^{k,\kappa}_t=\kappa V(X^{k,\kappa}_t)dt+\eta d B^k_t+\sum_{i\in I} U_i(X^{k,\kappa}_t)\circ dW^i_t,
\end{align}
where ($I$ being a countable set) $\{W^i,\; i\in I\}$ is a family of independent Brownian motions,
$\{U^i,\; i\in I\}$ is a family of vector fields on $M$ such that $\sum_{i\in I} U_i(U_if)=\Delta f$ for all $f\in C^{\infty}(\mathbb{R}^d)$ and $\{B^1,B^2,...,B^n\}$ is a family of $n$ independent Brownian motions in $\mathbb{R}^d$ independent of $\{W^i,i\in I\}$.
The $n$ particles $X^{1,\kappa},\dots,X^{n,\kappa}$ are $n$ diffusions with generator $\frac12(1+\eta^2)\Delta+\kappa V$ and are correlated through the common noise $W=\sum_{i\in I} U_iW^i$. 

The process $X^{(n),\kappa}$ is also a time-changed perturbation of the flow  $\phi^{\otimes n}$, which is generated by $V^{\otimes n}$. 
In order to apply Theorem \ref{th:main} to $X^{(n),\kappa}$, we will have to assume in addition that for $m^{\otimes n}$-almost all $x\in M^n$, $\otimes_{i=1}^n P(x_i)$ is ergodic for $\phi^{\otimes n}$. 
This allows to prove Theorem \ref{th:main2} in Section \ref{sec:avgstoflows}, whose statement is the convergence in law as $\kappa\to\infty$ of $(\pi(X^{1,\kappa}),\dots,\pi(X^{n,\kappa}))$ to a diffusion $\wt X^{(n)}$ on $\wt M^n$.

By construction, the family of processes $(X^{(n),\kappa},\,n\ge 1)$ is a \textit{consistent and exchangeable} family of Feller diffusions (see Section \ref{sec:consistency}). Using Theorem 2.1 from \cite{FCN}, $X^{(n),\kappa}$ is the $n$-point motion of a stochastic flow of kernels (SFK), $K^{\kappa}=(K^\ka_{s,t})_{s\le t}$. 
This SFK solves the following SDE: for all $f\in C^2_c(M)$, $x\in M$ and $s\le t$, 
\be 
\label{eq:3}K^\ka_{s,t}f(x)=f(x)+ \int_s^t K^\ka_{s,u}((1+\eta^2)A+\ka V)f(x) \,du +\int_s^t K^\ka_{s,u}(Wf(du))(x) ,
 \ee
with $W=\sum_{i\in I} U_i W^i$ a vector field-valued white noise, $\{W^i,\, i\in I\}$ being a family of independent white noises (we use  the notation $\int_s^t K^\ka_{s,u}(Wf(du))(x)=\sum_{i\in I} \int_s^t K^\ka_{s,u}(U_if)(x)W^i(du)$). 
The covariance of $W$ is given by  $C=\sum_{i\in I}U_i\otimes U_i$, and (following \cite{FCN}) the SDE \eqref{eq:3} is called the $((1+\eta^2)A+\ka V,C)$-SDE. 
Under the condition that for all $n$, the diffusion $\wt X^{(n)}$ is Feller,
the family of processes $(\wt X^{(n)},\, n\ge 1)$ is 
consistent and exchangeable and is associated to a SFK $\wt K$ on $\wt M$. 
We then prove Theorem \ref{th:mainflows} that states that the SFK $K^{\kappa}$ converges in the sense of finite dimensional distributions to $\wt K$ as $\kappa\to\infty$.

\bigskip
In order to prove Theorem \ref{th:main}, we had to revisit the framework of convergence of non-symmetric closed forms of \cite{hino98}, in which Hino only considers convergence of forms all defined on the same space. 
This is done in Section \ref{sec:conv}, where Theorems \ref{thm:mosco} and \ref{th:correspondance} are proved. 
These two theorems extend Theorem 2.4 in \cite{kuwae.shioya03} to non-symmetric closed forms.

\medskip

In Sections \ref{sec:bilinear} and \ref{sec:DF} the standard objects of the theory of non-symmetric Dirichlet forms are introduced and the correspondence with Markov processes is given. 
In Section \ref{sec:conv}, we recall and extend the notion of convergence (called Mosco-convergence by analogy to the symmetric case) of non-symmetric closed forms and show that this convergence entails the convergence of their associated semigroups. The main result of these sections is Theorem \ref{th:finite} where the convergence of the finite dimensional distributions of the Markov process associated to the regular Dirichlet form $\cE^{\kappa}$ is proven.

In Section \ref{sec:avg}, after recalling some notions of ergodic theory, Theorem \ref{th:main} is proved, i.e. we prove the convergence of $\pi(X^{\kappa})$ to a diffusion process $\wt X$ in $\wt M$. 
To prove Theorem \ref{th:main}, we prove the Mosco-convergence of the Dirichlet  forms $\cE^{\kappa}$ on $L^2(m)$ with domain $H^1(\mathbb{R}^d)$ towards a Dirichlet form on $L^2(\wt m)$, where $\wt m=\pi_* m$.
The Dirichlet form $\cE^{\kappa}:=\cE+\kappa\cE^V$ is such that for $f,g\in H^1(\mathbb{R}^d)$,
\begin{align}
\cE(f,g)&:=\frac{1}{2}\int_M \langle \nabla f,\nabla g\rangle \, dm,&\cE^V(f,g)&:=-\int_M (Vf) \, g\, dm.
\end{align}
Since $m$ is invariant for $\phi$, the form $\cE^{V}$ is anti-symmetric: $\cE^{V}(f,g)=-\cE^V(g,f)$.
Using Theorem \ref{thm:mosco} of Section \ref{sec:conv}, Theorem \ref{thm:mosco2} shows that $\cE^{\kappa}$ converges to a bilinear form $\wt\cE$ with a domain $\wt\cH$ as $\kappa\to \infty$. 
This limiting form is defined on $\wt M$ by:
\begin{align}
\wt\cH&=\{f:\wt M\to \mathbb{R}, f\circ\pi\in \cH\},&
\wt\cE(f,g)&=\cE(f\circ\pi,g\circ\pi).
\end{align}
$(\wt\cE,\wt\cH)$ can be viewed as a \textit{contraction} of the Dirichlet form $(\cE,\cH)$ by $\pi$.
Proposition \ref{prop:regular} states that the contracted bilinear form $(\wt\cE,\wt\cH)$ is itself a regular Dirichlet form and thus associated to a process $\wt X$ in $\wt M$. 
In particular, the \textit{regularity} $(\wt\cE,\wt\cH)$ of the form is non-trivial (this is the difficulty overlooked in \cite{BvR}).

In Section \ref{sec:hamilt} and Section \ref{sec:R3}, two examples inspired by the literature are studied. 
In Section \ref{sec:hamilt}, we consider the classic example of \cite{freidlin.wentzell12} of a complete vector field in $\bbR^2$ given by an Hamiltonian. 
In Section \ref{sec:R3}, we analyze a specific example on $\bbR^3$, based on a previous analysis by Mattingly and Pardoux in \cite{MP} and describe the limiting diffusion process and its generator. 
In Section \ref{sec:hamilt}, $\wt M$ is a metric graph (i.e. a gluing of segments) and in Section \ref{sec:R3}, $\wt M$ is a gluing of four cones in $\mathbb{R}^2$.
In each case, proving the regularity of the form $\wt \cE$ is a difficult task.
This has required a careful analysis of the space $\wt \cH$, which can be viewed as a weighted Sobolev space on $\wt M$.
This study is conducted in the Appendix and we believe it is of independent interest.

In Section \ref{sec:avgstoflows}, an averaging principle for SFKs is proved.
Finally, the two examples of Section \ref{sec:hamilt} and Section \ref{sec:R3} are studied at the flow level. 

\section*{Notation} For $M$ a topological space equipped with a positive Radon measure $m$,
\begin{itemize}
\item $C(M)$, $C_0(M)$, $C_b(M)$ and $C_c(M)$ denote respectively the spaces of continuous functions, continuous functions vanishing at $\infty$, bounded continuous functions and compactly supported continuous functions on $M$. 
These spaces will be equipped with the uniform norm $\norm{\cdot}_\infty$.
\item For $p\geqslant1$, $L^p(m)$ is the usual $L^p$-space associated to $m$ and its norm is denoted $\norm{\cdot}_{L^p(m)}$.
\item The inner product on $L^2(m)$ will be denoted by $\langle \cdot,\cdot\rangle_{L^2(m)}$.
\end{itemize}
And when $M$ is a smooth manifold,
\begin{itemize}
\item for $r\le \infty$, $C^r(M)$ denotes the spaces of $r$-times continuously differentiable functions on $M$, and $C^r_c(M)$ denotes the subspace of $C^r(M)$ of compactly supported functions.
\end{itemize}
For an operator $S$ on a Banach space $\cL$ with norm denoted by $\norm{\cdot}_{\cL}$, set $\norm{S}_\cL:=\sup_{\{f\in \cL:\norm{f}_\cL=1\}} \norm{Sf}_\cL$.

\section{Bilinear closed forms, semigroups and resolvents}\label{sec:bilinear}

\subsection{Resolvents and semigroups} Let $\cL$ be a Banach space, with norm denoted by $\norm{\cdot}_{\cL}$.
A {\it strongly continuous contraction resolvent} $(G_\al)_{\al>\la}$, with $\la\in \bbR$, is a family of operators on $\cL$ such that:
\begin{itemize}
\item $G_\alpha-G_\beta+(\alpha-\beta)G_\alpha G_\beta=0$ for all $\al,\bt >\la$;
\item For all $f\in \cL$, $\norm{(\al-\la)G_\al f}_{\cL}\leq \norm{f}_{\cL}$;
\item For all $f\in \cL$, $\lim_{\al\to\infty} \al G_\al f=f$ in $\cL$.
\end{itemize}

A {\it strongly continuous semigroup} $(T_t)_{t\geq 0}$ on $\cL$ is a family of operators on $\cL$ such that,
\begin{itemize}
\item $T_0f=f$ and $T_{s+t}f=T_sT_tf$ for all $f\in \cL$ and $s,t>0$;
\item  For all $f\in \cL$, $\lim_{t\to 0}T_tf=f$ in $\cL$.
\end{itemize}
From \cite{ma.rockner92}, Propositions I.1.5, I.1.10 and Theorem I.1.12, we have
\begin{proposition}\label{prop:association}
If $(T_t)$ is a strongly continuous semigroup such that for some $\la$, $\norm{e^{-\la t}T_t }_\cL\leq 1$  for all $t>0$, then $(G_\al)_{\al>\la}$, defined by the Bochner integral
\begin{align}\label{eq:Laplace}
G_\al f&=\int_0^{+\infty}e^{-\al t}T_t f \,dt,\quad \forall f\in \cL,
\end{align}
is a strongly continuous contraction resolvent. Reciprocally, if $(G_\al)_{\al>\la}$ is a strong\-ly continuous contraction resolvent then there exists a unique strongly continuous semigroup $(T_t)$ satisfying $\norm{e^{-\la t}T_t}_\cL\leq 1$  for all $t>0$ and such that Equation \eqref{eq:Laplace} holds.
\end{proposition}

Let $(T_t)$ be a strongly continuous semigroup such that for some $\la$, $\norm{e^{-\la t}T_t }_\cL\leq 1$  for all $t>0$.
Let $A$ be the infinitesimal generator of $(T_t)$ with domain $\cD(A)$, i.e. the set of all $u\in \cL$ such that $Au:=\lim_{t\downarrow 0}\frac{1}{t}(T_tu-u)$ exists in $\cL$. 
\begin{lemma}\label{lem:generator} 
\begin{itemize}
\item[(i)] For all $\alpha>\la$, $\cD(A)=G_\alpha(\cL)$.
\item[(ii)] For $f\in \cL$, $AG_\al f=\al G_\al f-f$.
\item[(iii)] $\cD(A)$ is dense in $\cL$.
\end{itemize}
\end{lemma}
\begin{proof}
The resolvent equation implies that for $\alpha,\beta>\la$, $G_\alpha(\cL)=G_\beta(\cL)$.
Using \eqref{eq:Laplace}, we prove that $G_\alpha(\cL)\subset \cD(A)$ and (ii). And we prove that $\cD(A)\subset G_\alpha(\cL)$ by taking for $u\in \cD(A)$, $f=\alpha u -Au$. Then it holds that $u=G_\alpha f$.
\end{proof}

\subsection{Bilinear closed forms}\label{sec:Conv_DF}
Let $\cL$ be a real separable Hilbert space, with inner product denoted by $\langle\cdot,\cdot\rangle_{\cL}$ and norm $\norm{\cdot}_{\cL}$. 
Let $\cE$ be a bilinear form on $\cL$ with domain $\cH$, a dense linear subset of $\cL$.
We denote by $\cE^s$ and $\cE^a$ respectively the symmetric and antisymmetric parts of $\cE$.
For $\alpha\in \mathbb{R}$, set $\cE_\alpha(\cdot,\cdot)=\cE(\cdot,\cdot)+\alpha\langle\cdot,\cdot\rangle_{\cL}$.
Note that $\cE^a_\alpha=\cE^a$.
The bilinear form $\cE$ is said closed if there is some $\alpha_0\in\mathbb{R}$ such that 
\begin{itemize}
\item[($\cE.1$)] $(\cE^s_{\alpha_0},\cH)$ is a positive definite symmetric bilinear form and, equipped with the inner product $\langle \cdot,\cdot\rangle_\cH:=\cE^s_{{\alpha_0}+1}$, $\cH$ is a Hilbert space. The associated norm will be denoted by $\norm{\cdot}_\cH$. 
\item[($\cE.2$)] $(\cE,\cH)$ satisfies the following weak sector condition\,: there exists $K\ge 1$ such that 
\be\label{eq:sector}|\cE_{{\alpha_0}+1}(u,v)|\le K \cE_{{\alpha_0}+1}(u,u)^{1/2}\cE_{{\alpha_0}+1}(v,v)^{1/2}, \hbox{ for all } u,v\in \cH.\ee
\end{itemize} 
Note that for $u\in \cH$, $\norm{u}_\cL\le \norm{u}_\cH$. 

Following \cite{oshima13} (Theorem 1.1.2) or \cite{ma.rockner92} (Theorem I.2.8), we have
\begin{proposition} \label{prop:resolvent}
There exist two unique strongly continuous contraction resolvents on $\cL$, $(G_\alpha)_{\alpha>\la}$ and $(\wh G_\alpha)_{\alpha>\la}$, such that for all $\alpha>\la$, $G_\al(\cL)\subset \cH$, $\wh G_\al(\cL)\subset \cH$ and 
\be\label{eq:resolvent}\cE_\alpha(G_\alpha f,u)=\cE_\alpha(u,\wh{G}_\alpha f)=\langle f,u\rangle_\cL \qquad \text{ for all }f\in \cL,\; u\in \cH.\ee
Moreover, for $\al>\la$, $\wh G_\alpha$ is the adjoint of $G_\alpha$, i.e. $\langle G_\alpha f,g\rangle_\cL=\langle f,\wh G_\alpha g\rangle_\cL$ for all $f,g\in \cL$. 
\end{proposition}

Let $\Theta(u)=\sup_{\{v\in \cH;\,\norm{v}_\cH=1\}}\cE_{\la+1}(v,u)$ for $u\in \cH$.

\begin{lemma}\label{lem:theta} For $u\in \cH$, $\norm{u}_\cH\le \Theta(u)\le K\norm{u}_\cH$. \end{lemma}
\begin{proof}
The sector condition ($\cE.2$) implies that $\Theta(u)\le  K\norm{u}_\cH$ and since  
$$ \norm{u}^2_\cH=\cE_{\la+1}(u,u)=\norm{u}_\cH\cE_{\la+1}(v,u),$$ with $v=u/ \norm{u}_\cH$, we have $\norm{u}_\cH\le \Theta(u)$.
\end{proof}
We thus have that $\Theta(\cdot)$ is a norm on $\cH$ equivalent to $ \norm{\cdot}_\cH$.

\begin{lemma} \label{lem:thetaG*}
For $f\in \cL$ and $\alpha>\la$, $\Theta(\wh{G}_\alpha f)\le C \norm{f}_\cL$, with $C=1+\big|\frac{\alpha-\la-1}{\alpha-\la}\big|$.
\end{lemma}
\begin{proof}
It holds that, for $v\in \cH$ with $\norm{v}_\cH=1$, 
$\cE_{\la+1}(v,\wh{G}_\alpha f)=\langle v,f\rangle_\cL +(\la+1-\alpha)\langle v,\wh{G}_\alpha f\rangle_\cL$ and
$|\cE_{\la+1}(v,\wh{G}_\alpha f)|\le\norm{f}_\cL +\big|\frac{\alpha-\la-1}{\alpha-\la}\big|\norm{f}_\cL.$
\end{proof}

For $\beta>\la$, define the bilinear form $\cE^{(\beta)}$ on $\cL$ by
\be\label{eq:approxE} \cE^{(\beta)}(u,v)=\beta\big\langle u-\beta G_\beta u,v\big\rangle_\cL, \qquad u,v\in \cL.\ee
The form $\cE^{(\beta)}$ approximates the form $\cE$ as shows the following proposition.
\begin{proposition}\label{prop:approxE}
\begin{itemize}
\item[(i)] $\cE^{(\beta)}(u,v)=\cE\big(\beta G_\beta u,v\big)$, for $u\in \cL$, $v\in \cH$.
\item[(ii)] $\cE(\bt G_\bt u,\bt G_\bt u)=\cE^{(\beta)}(u,u)-\beta\norm{u-\bt G_\bt u}^2_\cL$, for $u\in \cL$.
\item[(iii)] If $\limsup_{\beta\to\infty}\cE^{(\beta)}(u,u)<\infty$, then $u\in \cH$.
\item[(iv)] $\lim_{\beta\to\infty} \cE^{(\beta)}(u,v)=\cE(u,v)$, for $u,v\in \cH$. 
\end{itemize}
\end{proposition}
\begin{proof}
Assertion (i) follows from the fact that
\begin{align}
\cE (G_\beta u,v)
&= \cE_\beta (G_\beta u,v)-\beta\langle G_\beta u,v\rangle_\cL
= \big\langle u-\beta G_\beta u,v\big\rangle_\cL.
\end{align}
Assertion (ii) is a consequence of (i): 
\begin{align}
\cE(\bt G_\bt u,\bt G_\bt u)=\cE^{(\beta)}(u,\bt G_\bt u)=\bt\bra{u-\bt G_\bt u,u}_\cL-\beta\norm{u-\bt G_\bt u}^2_\cL.
\end{align}

{\it Proof of (iii)}\;: Let $u\in \cL$. 
Using (ii) and that $(\beta-\la)G_\beta$ is a contraction, we obtain  that $(\bt G_\bt u)_{\bt>\la+1}$ is bounded in $\cH$. 
Therefore, using Banach-Saks Theorem, there is a sequence $(\bt_n G_{\bt_n} u)$ with $\beta_n\to\infty$ whose Cesaro mean converges in $\cH$ to a $v\in \cH$.
We also have that $(\bt_n G_{\bt_n} u)$ converges to $u$ in $\cL$. Thus $u=v\in \cH$.

{\it Proof of (iv)}\;: Let $\alpha>\la$ and $f\in \cL$. 
For $u=G_\alpha f$, $v\in \cH$ and $\beta>\la$, using the definition of $G_\alpha$ and the resolvent equation,
\begin{align*}
\cE^{(\beta)}(G_\alpha f,v)-\cE(G_\alpha f,v)
&= \beta\bra{G_\alpha f-\beta G_\beta G_\alpha f,v}_\cL - \bra{f-\alpha G_\alpha f,v}_\cL \\
&= \frac{\beta}{\beta-\alpha}\bra{-\alpha G_\alpha f+\beta G_\beta f,v}_\cL - \bra{f-\alpha G_\alpha f,v}_\cL \\
&= \frac{\alpha}{\beta-\alpha} \bra{f-\alpha G_\alpha f,v}_\cL -\frac{\beta}{\beta-\alpha} \bra{f - \beta G_\beta f,v}_\cL
\end{align*}
which converges to $0$ as $\beta\to\infty$. 
To conclude, we prove that $G_\alpha (\cL)$ is dense in $\cH$. 
Let $u\in \cH$. As in the proof of (iii), using Banach-Saks Theorem, there is a sequence $(\bt_n G_{\bt_n} u)$ with $\beta_n\to\infty$ whose Cesaro mean converges in $\cH$ to $u\in \cH$. Since $G_\al (\cL)=G_{\bt_n}(\cL)$ for all $n$ (by the resolvent equation) this implies the result.
\end{proof}

By Proposition \ref{prop:association}, there exists a unique strongly continuous semigroup $(T_t)$ such that $\norm{e^{-\la t} T_t}_\cL\leq 1$ associated to $(G_\al)_{\al>\la}$. 
One then says that $(T_t)$ is associated to $(\cE,\cH)$.
\begin{lemma}\label{lem:cEA}
Let $A$ be the infinitesimal generator of $(T_t)$ the semigroup associated to $(\cE,\cH)$, with domain $\cD(A)$. Then for $u\in \cD(A)$ and $v\in \cH$,
\be \cE(u,v) = -\bra{Au,v}_{\cL}. \label{eq:cEA}\ee
Moreover, $u$ belongs to $\cD(A)$ if and only if, as a functional with domain $\cH$, the mapping $v\mapsto \cE(u,v)$ is continuous with respect to $\|\cdot\|_\cL$.
\end{lemma}
\begin{proof} Fix $\alpha>{\alpha_0}$, $u\in \cD(A)$ and $v\in \cH$. Using Lemma \ref{lem:generator}, there is an $f\in \cL$ such that $u=G_\alpha f$. 
Since $Au=\alpha G_\alpha f-f$, we easily obtain \eqref{eq:cEA} by using the definition of $G_\alpha f$. 
The second part is Proposition I.2.16 p.23 in \cite{ma.rockner92}.
\end{proof}

\section{Dirichlet Forms}\label{sec:DF}
\subsection{Dirichlet forms and Markov processes}
\subsubsection{Definitions} Let $M$ be a locally compact separable metric space. 
Let $m$ be a positive Radon measure on $M$ such that $\hbox{Supp}[m]=M$.  
Let $\cE$ be a bilinear closed form on $\cL:=L^2(m)$, with domain $\cH$. In particular, there is some ${\alpha_0}>0$ for which $(\cE.1)$ and $(\cE.2)$ are satisfied.
If the bilinear form $\cE$ also satisfies
\medskip

\noindent ($\cE.3$) for all $u\in \cH$ and $a\ge 0$, $u\wedge a\in \cH$ and $\cE(u\wedge a,u-u\wedge a)\ge 0$
\medskip

%

\noindent then $(\cE,\cH)$ is called a {\it Dirichlet form} on $\cL$.

Note that, if $(T_t)$ is the semigroup associated to $\cE$, then (Theorem 1.1.5 in \cite{oshima13}) ($\cE.3$) 
is equivalent to
\medskip

\noindent($\cE.3a$) $(T_t)$ is sub-Markov\,: if $f\in L^2(m)$ satisfies $0\le f\le 1$ $m$-a.e., then $0\le T_t f\le 1$ $m$-a.e.
\medskip

\noindent
So, when ($\cE.3$) is satisfied, the restriction of $T_t$ to $L^2(m)\cap L^{\infty}(m)$ can be extended to an operator $T^{\infty}_t$  on $L^{\infty}(m)$ which satisfies $\norm{T^{\infty}_t}_{ L^{\infty}(m)}\leq 1$. 
Then $(T^\infty_t)$ is a strongly continuous contraction semigroup and if we set $G^{\infty}_\alpha f=\int_0^\infty e^{-\alpha t} T^{\infty}_tf dt$ for $f\in L^\infty(m)$ and $\al>0$,  
$(G^{\infty}_\alpha)_{\al>0}$ is a strongly continuous contraction resolvent on $L^\infty(m)$, i.e. such that $\norm{\al G^{\infty}_\al }_{ L^{\infty}(m)}\leq1$ for $\al >0$.
Moreover, for $f\in L^2(m)\cap L^{\infty}(m)$ and $\alpha>{\alpha_0}$, $G_\alpha f=G_\alpha^\infty f$ $m$-a.e.

The Dirichlet form $\cE$ is said {\it regular} if $\cH\cap C_c(M)$ is dense in $H$ with respect to $\norm{\cdot}_\cH$ and is dense in $C_c(M)$ with respect to $\norm{\cdot}_\infty$.

\begin{theorem}[Theorem 3.3.4. in \cite{oshima13}]
For any given regular Dirichlet form $\cE$ on $L^2(m)$ with domain $\cH$, there exists a Hunt process whose resolvent $R_\alpha f$ is a quasi-continuous modification of $G^{\infty}_\alpha f$ for any $f\in L^\infty(m)$ and $\alpha>0$.
\end{theorem}

We say that a Dirichlet form $(\cE,\cH)$ possesses the {\it  local property} if $\cE(u,v)=0$ whenever $u,v\in \cH$ have disjoint supports.
\begin{theorem}[Theorem 3.5.12 in \cite{oshima13}]\label{th:diff}
Let $(\cE,\cH)$ be a regular Dirichlet form, then the following conditions are equivalent to each other:
\begin{itemize}
\item[(i)] $(\cE,\cH)$ possesses the local property;
\item[(ii)] the Hunt process associated to $(\cE,\cH)$ is a diffusion process.
\end{itemize}
\end{theorem}

\subsubsection{Dirichlet form associated to a diffusion on a manifold}
\label{sec:DFM}
Let $M$ be a smooth connected oriented Riemannian manifold.
Let $m$ be a positive Borel measure $m$ on $M$ equivalent to the volume form on $M$, with a $C^1$ density we also denote by $m$.  
For $V$ a $C^1$-vector field, set $\hbox{div}_m V = m^{-1}\hbox{div}(mV) \in C(M)$, where $\hbox{div}$ is the divergence operator on $M$. Then we have the following integration by part formula
\be \langle Vf,g\rangle_{L^2(m)}= - \langle f,Vg\rangle_{L^2(m)} - \int_M fg \,\hbox{div}_m V dm, \hbox{ for all } f,g \in C^{1}_c(M) \label{eq:IBP1}\ee

Let $S$ be a second order differential operator on $M$, with $S1=0$. 
Suppose also that 
\begin{itemize}
\item[(i)] $Sf\in L^2(m)$ for all $f\in C_c^\infty(M)$.
\item[(ii)] $S$ is symmetric on $L^2(m)$, i.e. $\langle Sf,g\rangle_{L^2(m)}=\langle f,Sg\rangle_{L^2(m)}$, for all $f,g\in C^\infty_c(M)$.
\item[(iii)] $S$ is non positive, i.e. $\langle Sf,f\rangle_{L^2(m)}\le 0$ for all $f\in C_c^\infty(M)$.
\end{itemize}
Let $\Gamma$ be the {\it carr\'e du champ operator} of $S$ defined by 
\be\Gamma(f,g)=\frac{1}{2}[S(fg)-fSg-gSf].\ee
Then, one has that $-\langle Sf,g\rangle_{L^2(m)}=\int_M \Gamma(f,g) dm$.
       
Let $V$ be a $C^1$-vector field on $M$ and define
\be 
\cE(f,g)=\int \Gamma(f,g) dm - \langle Vf, g \rangle_{L^2(m)},\qquad f,g\in C^{\infty}_c(M). \label{eq:DFGV0}
\ee
\begin{remark}
The symmetric and antisymmetric parts of $\cE$ are given by\,: for $u,v\in C^{\infty}_c(M),$
\begin{align}
\cE^s(f,g)=&\int \G(f,g)dm+\frac12\int  fg\, \hbox{div}_m V\, dm\\
\cE^a(f,g)=&\frac12\int (fVg-gVf)dm.
\end{align}
It may be convenient to take for $m$ the volume form on $M$, in which case $\Div_m=\Div$.
\end{remark}

Proposition I.3.3 in \cite{ma.rockner92} and Theorem 1.2.1 in \cite{oshima13} imply the following proposition
\begin{proposition}\label{prop:DF0}
Assume that for some constant $c_1\in\mathbb{R}$,
\begin{align}\label{eq:hyp(3)} \Div_m V \ge c_1
\end{align}
and that for some constant $K<\infty$ and for $\alpha_0:=-\frac{c_1}{2}$, the weak sector condition 
\be\label{eq:sector1}|\langle Vf,g\rangle_{L^2(m)}|\le K \cE_{{\alpha_0}+1}(f,f)^{1/2}\cE_{{\alpha_0}+1}(g,g)^{1/2}, \hbox{ for all } f,g\in C^\infty_c(M),\ee
holds.
Then 
\begin{itemize}
\item[(i)] $(\mathcal{E},C^\infty_c(M))$ is closable and its smallest closed extension $(\mathcal{E},\mathcal{H})$ is a closed form satisfying $(\mathcal{E}.1)$ and $(\mathcal{E}.2)$ with $\alpha_0=-\frac{c_1}{2}$.
\item[(ii)] $(\mathcal{E},\mathcal{H})$ is a regular Dirichlet form that possesses the local property.
\end{itemize}
\end{proposition}
\begin{proof} It is straightforward to check that $\cE^s_{\alpha_0}$ is positive definite. 
Then the fact that $(\cE_{\alpha_0},C^\infty_c(M))$ is closable follows from Proposition I.3.3 in \cite{ma.rockner92}, which is applied with $S$ replaced by $ S+V-{\alpha_0}$.
And we can conclude the item (i) by using Theorem 1.2.1 in \cite{oshima13}. By construction, (ii) is satisfied.
\end{proof}

The Dirichlet form $(\mathcal{E},\mathcal{H})$ will be called the Dirichlet form on $L^2(m)$ associated to $\Gamma$ (or $S$) and $V$.

\begin{remark} Denote by $H^1(m)$ the Hilbert space obtained as the completion of $C_c^\infty(M)$ with respect to the norm $\displaystyle \|f\|_{H^1(m)}=\left(\int (f^2+|\nabla f|^2) dm\right)^{\frac{1}{2}}$. Then
\begin{itemize}
\item $\cH\supset H^1(m)$ as soon as there is a positive constant $c_2$ such that
\begin{align} \label{eq:hyp(1)} & \G(f,f) \le c_2 \abs{\grad f}^2, \hbox{ for all } f\in C_c^\infty(M);
\end{align}
\item $\cH\subset H^1(m)$ as soon as $\G$ is uniformly elliptic, i.e. if there is a positive constant  $c_3$ such that
\be
c_3 \abs{\grad f}^2\leq \G(f,f),  \hbox{ for all } f\in C^\infty_c(M). \label{Guelliptic}
\ee
\end{itemize}
\end{remark}

We now give sufficient conditions ensuring that the weak sector condition \eqref{eq:sector1} holds.

\begin{proposition}\label{prop:WScond} Assume that \eqref{eq:hyp(3)} is satisfied and that there is a measurable function $v:M\to\mathbb{R}$ such that  for all $f\in C^\infty_c(M)$, $|Vf|^2 \le v^2 \Gamma(f,f)$.
Then the weak sector condition \eqref{eq:sector1} holds for all $f,g \in C^\infty_c(M)$ as soon as one of the following conditions is satisfied:
\begin{enumerate}
\item $v\in L^\infty(m)$.
\item For some $q>2$, $v\in L^q(m)$ and $(\Gamma,m)$ satisfies a Sobolev inequality with dimension $q>2$  (or of exponent $p:=\frac{2q}{q-2}>2$) and constants $A\in\mathbb{R}$ and $C>0$, i.e. for all $f\in C^\infty_c(M)$
$$\|f\|^2_{L^p(m)}\le A \|f\|^2_{L^2(m)} + C\int \Gamma(f,f) dm.$$
\item $\int \exp(v^{2}) dm\,<\infty$, $m$ is a probability measure and $(\Gamma,m)$ satisfies a logarithmic Sobolev inequality with  constants $C>0$ and $D\ge 0$, i.e. for all $f\in C^\infty_c(M)$
$$\Ent_m(f^2)\le 2C \int \Gamma(f,f) dm + D \|f\|^2_{L^2(m)}$$
where $\displaystyle \Ent_{m}(f):=\int f\ln f dm-\int f dm\ln\left(\int f dm\right)$, with $f\ge 0$.
\end{enumerate}
\end{proposition}

\begin{proof}
It is straightforward to check that \eqref{eq:sector1} holds under condition (1).

To prove the weak sector condition \eqref{eq:sector1} for all $f,g\in C^\infty_c(M)$ under condition (2) or (3), we use
\begin{align*}
\left|\int f Vg dm\right|
& \le \int |v| |f| \sqrt{\Gamma(g,g)} dm\\
& \le \left(\int |v|^2 |f|^2 dm\right)^{\frac{1}{2}} \left(\int \Gamma(g,g) dm\right)^{\frac{1}{2}}\\
& \le \left(\int v^2 |f|^2 dm\right)^{\frac{1}{2}}\cE^s_{\alpha_0+1}(g,g)^{\frac{1}{2}}.
\end{align*}

When condition (2) is satisfied, Hölder inequality and Sobolev inequality imply that, for some constant $K<\infty$,
\begin{align*}
\int |v|^2 |f|^2 dm
&\le \left(\int |v|^{q} dm\right)^{\frac{2}{q}} \left(\int |f|^{p} dm\right)^{\frac{2}{p}}
\le K \cE^s_{\alpha_0+1}(f,f).
\end{align*}

When condition (3) is satisfied, we use the entropic inequality (5.1.2) p 236 in \cite{BGL} and the logarithmic Sobolev inequality :
\begin{align*}
\int |v|^2 |f|^2 dm
&\le \Ent_m(f^2)+\int f^2dm\ln\left(\int \exp(|v|^{2}) dm\right)
\le K \cE^s_{\alpha_0+1}(f,f)
\end{align*}
for some constant $K<\infty$.
\end{proof}

Note that when $\G$ is uniformly elliptic, then condition (1) is satisfied as soon as $V$ is bounded.
\begin{example} Let $\{V_\ell: 0\le \ell\le L\}$ be a finite family of $C^1$-vector fields.
Define
\be Su=\sum_{\ell=1}^L V_\ell(V_\ell u) + \hat V u\ee
with $\hat V=\sum_{\ell= 1}^L (\Div_m V_\ell) V_\ell$.
Then $S$ is symmetric on $L^2(m)$ and
\be \Gamma(u,v)=\sum_{\ell=1}^L (V_\ell u) (V_\ell v) .\ee
Suppose that $V_0$ satisfies \eqref{eq:hyp(3)} and condition (1). Then there is a (unique) Dirichlet form on $L^2(m)$ associated to $\Gamma$ and $V_0$.
\end{example}

In the two following examples, $M=\mathbb{R}^d$, $\Gamma(f,g)=\frac12\langle\nabla f,\nabla g\rangle$ and $m=e^{-W}\lambda_d$, where $W\in C^2(M)$ and $\lambda_d$ is the Lebesgue measure on $\mathbb{R}^d$.

\begin{example} Suppose that $d\ge 3$, $W=0$ and that $V$ is such that $\int \|V\|^{d} dm\,<\infty.$ Then $(\Gamma,m)$ satisfies a Sobolev inequality with exponent $p=\frac{2d}{d-2}$ and condition (2) is satisfied.
\end{example}

\begin{example}
Let us assume that $(\Gamma,m)$ satisfies a curvature-dimension condition $CD(\rho,\infty)$ with $\rho>0$ (for instance, $CD(\rho,\infty)$ holds as soon as $\hbox{Hess}\, W \ge \rho$). 
Suppose also that $\displaystyle \int e^{-W} d\lambda_d=1$ and $\displaystyle\int \exp(\|V\|^{2}) e^{-W} d\lambda_d\,<\infty.$
Then $(\Gamma,m)$ satisfies a logarithmic Sobolev inequality (see \cite{BGL} p.268)
and condition (3) is satisfied.
\end{example}

Example 2 can be found in \cite{ma.rockner92} or in \cite{oshima13}. But Example 3 is new up to our knowledge.

\section{Convergence of bilinear closed forms} \label{sec:conv}
Let us given for all integer $n\geq 0$, a bilinear closed form $(\cE^n,\cH_n)$ on a separable Hilbert space $\cL_n$ and let $(\cE,\cH)$ be a bilinear closed form on a separable Hilbert space $\cL$. 
We assume that there is $\la\in\mathbb{R}$ such that the bilinear forms $(\cE,\cH)$ and $(\cE^n,\cH_n)_{n\ge 0}$ satisfy $(\cE.1)$ and $(\cE.2)$ with an eventually different constant in $(\cE.2)$ for each $n$, $K_n\ge 1$. 

For each $n$, let $(G^n_\al)_{\al>\la}$ be the resolvent associated to $(\cE^n,\cH_n)$ (by Proposition \ref{prop:resolvent}) and let $T^n_t$ be its associated semigroup on $\cL_n$  (by Proposition \ref{prop:association}). 
Let also $(G_\al)_{\al>\la}$ be the resolvent associated to $(\cE,\cH)$ and $T_t$ be its associated semigroup on $\cL$.

\subsection{Convergence of Hilbert spaces}
Assume that $\cL_n$ converges to $\cL$ in the sense that there are linear operators $\Phi_n:\cL\to \cL_n$, such that
\be \lim_{n\to\infty} \norm{\Phi_n u}_{\cL_n}=\norm{u}_\cL,\qquad u\in \cL.\ee

Let $(u_n)$ be a sequence with $u_n\in \cL_n$, and let $u\in \cL$.

\begin{definition}We say that $(u_n)$ {\bf strongly converges} to $u$ if 
\be \lim_{n\to\infty}\norm{\Phi_n u-u_n}_{\cL_n}=0.\ee
We say that $(u_n)$ {\bf weakly converges} to $u$ if 
\be \lim_{n\to\infty} \langle u_n,v_n\rangle_{\cL_n}=\langle u,v\rangle_\cL\ee
for any sequence $(v_n)$ with $v_n\in \cL_n$  strongly converging to a $v\in \cL$.
\end{definition}

Note that the strong convergence do imply the weak convergence, and that these convergences correspond to the usual strong and weak convergence in $\cL$, when for all $n$, $\cL_n=\cL$ (and $\Phi_n$ is the identity on $\cL$).  

\begin{lemma}
If $\sup_n\norm{u_n}_{\cL_n}<\infty$ and if for all $v\in \cL$, 
$$\lim_{n\to\infty}\langle u_n,\Phi_n v\rangle_{\cL_n}=\langle u,v\rangle_{\cL},$$ 
then $(u_n)$ weakly converges to $u$.
\end{lemma}
\begin{proof} 
Let $(v_n)$ be strongly converging to $v$. The lemma follows from the fact that
\be \big|\langle u_n,v_n\rangle_{\cL_n}-\langle u_n,\Phi_n v\rangle_{\cL_n}\big| \le \norm{u_n}_{\cL_n} 
 \norm{\Phi_n v-v_n}_{\cL_n}\ee
converges to $0$, by definition of the strong convergence of $(v_n)$ towards $v$. \end{proof}

\begin{lemma}\label{lem:WCVrel_comp}
If $\sup_n\norm{u_n}_{\cL_n}<\infty$, then there exists a weakly converging subsequence $(u_{n_k})$.
\end{lemma}
\begin{proof} See Lemma 2.2 in \cite{kuwae.shioya03}.
\end{proof}

\begin{lemma} \label{lem:W-SCV} 
\begin{itemize}
\item[(i)] If $u_n$ weakly converges to $u$, then $\sup_n\norm{u_n}_{\cL_n}<\infty$ and $\norm{u}_\cL\le \liminf_{n\to\infty}\norm{u_n}_{\cL_n}$.
\item[(ii)] $u_n$ strongly converges to $u$ if and only if $u_n$ weakly converges to $u$ and $\norm{u}_\cL= \lim _{n\to\infty}\norm{u_n}_{\cL_n}$.
\item[(iii)] $u_n$ strongly converges to $u$ if and only if 
$$\lim_{n\to\infty} \langle u_n,v_n\rangle_{\cL_n}=\langle u,v\rangle_\cL,$$ 
for any sequence $v_n$ weakly converging to $v\in \cL$.
\end{itemize}
\end{lemma}
\begin{proof} It is lemma 2.3 (for (i) and (ii)) and Lemma 2.4 (for (iii)) in  \cite{kuwae.shioya03}. \end{proof}

\subsection{Convergence of bounded operators}
For each $n$, let $A_n$ be a bounded operator on $\cL_n$, and let $A$ be a bounded operator on $\cL$.
Denote by $A_n^*$ and by $A^*$ respectively the dual operators of $A_n$ in $\cL_n$ and of $A$ in $\cL$. 

\begin{definition}
We say that $(A_n)$ {\bf strongly (resp. weakly) converges} to $A$ if $A_nu_n$ strongly (resp. weakly) converges to $Au$, for any sequence $(u_n)$ with $u_n\in \cL_n$ strongly (resp. weakly) converging to a $u\in \cL$.
\end{definition}

It should be noted that the strong convergence do not imply the weak convergence in general.

\begin{proposition}\label{prop:CVws}
$A_n$ strongly converges to $A$ if and only if $A_n^*$ weakly converges to $A^*$.
\end{proposition}
\begin{proof}
Suppose $A_n$ strongly converges to $A$. Let $(u_n)$ be weakly converging to a $u$ and $(v_n)$ be strongly converging to a $v$. Then, as $n\to\infty$
$$\langle A^*_n u_n,v_n\rangle_{\cL_n}=\langle u_n,A_n v_n\rangle_{\cL_n}\to \langle u,A v\rangle_{\cL}=\langle A^*u,v\rangle_\cL$$
which proves that $A^*_n$ weakly converges to $A$.

On the converse, suppose that $A^*_n$ weakly converges to $A^*$.
Let $(u_n)$ be strongly converging to a $u$ and $(v_n)$ be weakly converging to a $v$. 
Then, as $n\to\infty$
$$\langle A_n u_n,v_n\rangle_{\cL_n}=\langle u_n,A^*_nv_n\rangle_{\cL_n}\to \langle u,A^*v\rangle_{\cL}=\langle Au,v\rangle_\cL.$$
We then conclude using Lemma \ref{lem:W-SCV} (iii).
\end{proof}

A consequence of this proposition is that the strong convergence and the weak convergence are equivalent for symmetric bounded operators.

\subsection{Mosco convergence}
In this section, a convergence of bilinear closed forms is defined. Since this convergence and the Mosco convergence are equivalent for symmetric Dirichlet forms, it will be also called the Mosco convergence.
We follow and adapt \cite{hino98} to our framework (in \cite{hino98}, we would have $\cL_n=\cL$ for all $n$). Following the notation before Lemma \ref{lem:theta}, set $\Theta^n(u)=\sup_{\{v\in \cH_n;\,\norm{v}_{\cH_n}=1\}}\cE^n_{\la+1}(v,u)$ for $u\in \cH_n$.

\begin{definition}
We say that $\cE^n$ Mosco-converges to $\cE$ if (F1) and (F2) hold, with
\begin{itemize}
\item[(F1)] If $(u_n)$ with $u_n\in \cH_n$ weakly converges to $u$ and if $\liminf_{n\to\infty} \Theta^n(u_n)<\infty$, then $u\in \cH$.
\item[(F2)] For any sequence $(u_n)$ with $u_n\in \cH_n$ weakly converging to a $u\in \cH$, and any $v\in \cH$, there exists a sequence $(v_n)$ with $v_n\in \cH_n$ strongly converging to $v$ such that \be\label{eq:F2}\lim_{n\to\infty}\cE^n(v_n,u_n)=\cE(v,u).\ee
\end{itemize}
\end{definition}

We introduce also the two following conditions
\begin{itemize}
\item[(F2')] For $n_k\uparrow \infty$, $u_k\in \cH_{n_k}$ such that $u_k$ weakly converges to a $u\in \cH$ which satisfies $\sup_k\Theta^{n_k}(u_k)<\infty$, there exists a dense subset $C$ in $\cH$ such that for all $v\in C$, there exists a sequence $(v_k)$ with $v_k\in \cH_{n_k}$ strongly converging to $v$ with \be\label{eq:F2'}\liminf_{k\to\infty}\cE^{n_k}(v_k,u_k)\le\cE(v,u).\ee
\item[(R)] $G_\alpha^n$ strongly converges to $G_\alpha$, for $\alpha>\la$.
\end{itemize}


\begin{theorem} \label{thm:mosco}We have
\begin{itemize}
\item[(i)] $\cE^n$ Mosco-converges to $\cE$ if and only if (F1) and (F2') hold.
\item[(ii)] $\cE^n$ Mosco-converges to $\cE$ if and only if (R) holds.
\end{itemize}
\end{theorem}
\begin{proof} This theorem corresponds to Theorem 3.1 in \cite{hino98}, whose statement is\,: $(F2)\Rightarrow (F2')$ , $(F1)(F2')\Leftrightarrow (R)\Leftrightarrow (F1)(F2).$ We follow the proof of this theorem.

\smallskip {\it Proof of $(F2)\Rightarrow (F2')$}\,:  
Note first that (F2) implies that for all $v\in \cH$, there exists $(v_n)$ with $v_n\in \cH_n$ strongly converging to $v$ (it suffices to take $u_n=u=0$ in (F2)).\\
Let us now prove (F2') with $C=\cH$:
Let $n_k\uparrow \infty$ and a sequence $(u_k)$ with $u_k\in \cH_{n_k}$ such that $(u_k)$  weakly converges to a $u\in \cH$ and such that $\sup_k\Theta^{n_k}(u_k)<\infty$.
There exists  $(w_n)$ with $w_n\in \cH_n$ strongly converging to $u$.
By taking $u'_n=w_n$ for $n\not\in\{n_k:k\ge 0\}$ and $u'_{n_k}=u_k$ for $k\ge 0$, we have that $(u'_n)$ weakly converges to $u$.
Condition (F2) implies that for $v\in \cH$, there exists $(v_n)$ with $v_n\in \cH_n$ strongly converging to $v$ satisfying \eqref{eq:F2}, which easily implies \eqref{eq:F2'}. 

\smallskip {\it Proof of $(F1)(F2')\Rightarrow (R)$}\,: Fix $\alpha>\la$. 
Let $(f_n)$ with $f_n\in \cL_n$ weakly converging to an $f\in \cL$. Set $u_n=G_\alpha^{n,*} f_n$. 
Then if $u_n$ weakly converges to $G_\alpha^*f$, then $G_\alpha^{n,*}$ weakly converges to $G_\alpha^*$ which implies (R) (using Proposition \ref{prop:CVws}).

First, since $\norm{(\alpha-\la)u_n}_{\cL_n}\le \norm{f_n}_{\cL_n}$, $\sup_n \norm{u_n}_{\cL_n}\le (\alpha-\la)^{-1}\sup_n \norm{f_n}_{\cL_n}<\infty$ (using Lemma \ref{lem:W-SCV}-(i)). 
Thus, Lemma \ref{lem:WCVrel_comp} implies that there is a subsequence $(u_{n_k})$ weakly converging to a $u\in \cL$. Using Lemma \ref{lem:thetaG*}, $\sup_n\Theta^n(u_n)<\infty$. 
Therefore, (F1) implies that $u\in \cH$.

From (F2'), there exists $C$ dense in $\cH$ such that for all $v\in C$, there exists a sequence $(v_k)$ with $v_k\in \cH_{n_k}$ strongly converging to $v$ with \be\liminf_{k\to\infty}\cE^{n_k}(v_k,u_{n_k})\le\cE(v,u).\ee
Now, $\cE^{n_k}_\alpha (v_k,u_{n_k})=\langle v_k,f_{n_k}\rangle_{\cL_{n_k}}$ which converges to $\langle v,f\rangle_\cL$. 
As a consequence we have that for all $v\in C$, 
\be \cE_\alpha(v,u)\ge \langle v,f\rangle_\cL.\ee
Using the fact that $C$ is dense in $\cH$, this inequality also holds for all $v\in \cH$. Then, using this inequality with $-v$ instead of $v$, we get $\cE_\alpha(v,u)\le \langle v,f\rangle_\cL$ for $v\in \cH$. 
Therefore, $\cE_\alpha(v,u)= \langle v,f\rangle_\cL$ for all $v\in \cH$, and thus $u=G_\alpha^*f$. 
The sequence $(u_n)$ has a unique weak accumulation point therefore it converges weakly to $u=G_\alpha^*f$.

\smallskip {\it Proof of $(R)\Rightarrow (F1)$}\,: Let $(u_n)$ with $u_n\in \cH_n$ be weakly converging to $u\in \cL$ such that $M:=\liminf_{n\to\infty} \Theta^n(u_n)<\infty$. 
From Proposition \ref{prop:approxE}-(iii), to prove that $u\in \cH$, it suffices to prove that $\limsup_{\bt\to\infty}\cE^{(\bt)}(u,u)<\infty$.
Set $v_n=\Phi_nu$, then $(v_n)$ strongly converges to $u$.
Fix $\beta>\la\vee 0$. On one hand, since $G_\beta^nv_n$ strongly converges towards $G_\beta u$,
\begin{align}\label{eq:convform}
\cE^{n,(\bt)}\big(v_n,u_n)
&= \bt\bra{v_n-\bt G^n_{\bt}v_n, u_n}_{\cL_n}
\conv{n}{\infty} \bt\bra{u-\bt G_{\bt}u, u}_{\cL}=\cE^{(\bt)}(u,u).
\end{align}
On the other hand, using Proposition \ref{prop:approxE}-(i),
\begin{align}\nonumber
\cE^{n,(\bt)}(v_n,u_n)
&=\cE^n(\bt G^n_\bt v_n,u_n)
=\cE_{\la+1}^{n}(\bt G^n_\bt v_n,u_n)-(\la+1)\bra{\bt G^n_\bt v_n,u_n}_{\cL_n}\\
&\leq \Theta^n(u_n)\norm{\bt G^n_\bt v_n}_{\cH_n}-(\la+1)\bra{\bt G^n_\bt v_n,u_n}_{\cL_n}
\end{align}
Then, let us proceed as in Equation \eqref{eq:convform} to get 
\begin{align}\nonumber
\cE^n(\bt G_\bt^nv_n,\bt G^n_\bt v_n)
\conv{n}{\infty}\cE(\bt G_\bt u,\bt G_\bt u).
\end{align}
This entails that $\lim_{n\to\infty}\norm{\bt G^n_\bt v_n}_{\cH_n}=\norm{\bt G_\bt u}_{\cH}$ which implies that (using that $(\beta-\la)G_\beta$ is a contraction)
\begin{align}
\cE^{(\bt)}(u,u)=\lim_{n\to\infty}\cE^{n,(\bt)}(v_n,u_n)&\leq M \norm{\bt G_\bt u}_{\cH}-(\la+1)\bra{\bt G_\bt u,u}_\cL\nonumber\\
&\leq M \norm{\bt G_\bt u}_{\cH}+\frac{|\la+1|\beta}{(\beta-\la)}\norm{u}^2_\cL. \label{eq:eb1}
\end{align}

Proposition \ref{prop:approxE}-(ii) and that $(\beta-\la)G_\beta$ is a contraction imply that 
\be\norm{\beta G_\beta u}_\cH^2\le \cE^{(\beta)}(u,u)+\frac{|\la+1|\beta}{(\beta-\la)}\norm{u}^2_\cL.\label{eq:eb2}\ee
Equations \eqref{eq:eb1} and \eqref{eq:eb2} easily imply that $\limsup_{\beta\to\infty}\cE^{(\beta)}(u,u)<\infty$.

\smallskip {\it Proof of $(R)\Rightarrow (F2)$}\,:
Let $(u_n)$ with $u_n\in \cH_n$ be weakly converging to a $u\in \cH$, and let $v\in \cH$.
Set $w_n=\Phi_nv$, then $(w_n)$ strongly converges to $v$.
Thus (the convergences below being strong convergences), 
\be\lim_{\beta\to\infty}\lim_{n\to\infty} \beta G_\beta^nw_n=\lim_{\beta\to\infty}\beta G_\beta v=v\ee
Since (by definition) $\cE^{n,(\beta)}(w_n,u_n)=\beta\big\langle w_n-\beta G_\beta^nw_n,u_n\rangle_{\cL_n}$, we have that 
\begin{align}
 \lim_{\beta\to\infty}\lim_{n\to\infty}\cE^{n,(\beta)}(w_n,u_n)
&=\lim_{\beta\to\infty}\cE^{(\beta)}(v,u)= \cE(v,u).
\end{align}
This implies  the existence of a sequence $\beta_n\uparrow\infty$ such that $v_n:=\beta_n G_{\beta_n}^nw_n$ strongly converges to $v$ and
\be \lim_{n\to\infty} \cE^{n,(\beta_n)}(w_n,u_n) =\cE(v,u).\ee
Since $\cE^n(v_n,u_n)=\cE^{n,(\beta_n)}(w_n,u_n)$ (Proposition \ref{prop:approxE}-(i)), we have that
\be \lim_{n\to\infty} \cE^{n}(v_n,u_n) =\cE(v,u).\ee

\end{proof}  

\subsection{Convergence of semigroups} 
The next theorem proves the equivalence between the convergence of the resolvents $(G^n_\al)$ (condition (R)) and the convergence of the semigroups $(T^n_t)$. Using Theorem \ref{thm:mosco} $(ii)$, this proves that the Mosco-convergence is also equivalent to the convergence of the semigroups $(T^n_t)$.

\begin{theorem}\label{th:correspondance} Denote $\norm{\Phi_n}=\sup_{f\in \cL:\norm{f}_\cL=1}\norm{\Phi_nf}_{\cL_n}$ and let us assume that $\sup_{n\ge 0} \norm{\Phi_n}<\infty$. Then, the following statements are equivalent:
\begin{enumerate}
\item For some $\al>\la$, $G^n_{\al}$ strongly converges to $G_{\al}$;
\item $T^n_t$ strongly converges to $T_t$ for all $t\geq 0$;
\item $T^n_t$ strongly converges to $T_t$ uniformly on bounded intervals i.e. for all $T>0$ and all sequences $(f_n)$ strongly converging to $f$, 
\be
\lim_{n\to\infty}\sup_{t\in[0,T]}\norm{T^n_tf_n-\Phi_n T_tf}_{\cL_n}=0.
\ee
\end{enumerate}
\end{theorem}
\begin{proof}[Proof of Theorem \ref{th:correspondance}] (3)$\Rightarrow$(2) is immediate. We prove (2)$\Rightarrow$(1) and (1)$\Rightarrow$(3).

\textit{Proof of (2)$\Rightarrow$(1)}:
Fix some $\alpha>\la$. Let $(f_n)$ with $f_n\in \cL_n$ be strongly converging to a $f\in \cL$ and let $(g_n)$ with $g_n\in \cL_n$ be weakly converging to a $g\in \cL$. 
Then, for all $t\geq 0$, $\lim_{n\to\infty}\bra{T^n_t f_n,g_n}_{\cL_n}=\bra{T_t f,g}_{\cL}$.
Using Lemma \ref{lem:W-SCV} (i), there exists $M>0$ such that $\abs{\bra{T^n_t f_n,g_n}_{\cL_n}}\leq e^{\la t}\norm{f_n}_{\cL_n}\norm{g_n}_{\cL_n}\leq Me^{\la t}$ for all $n$.
We also have: $\abs{\bra{T_t f,g}_{\cL}}\leq e^{\al_0 t}\norm{f}_{\cL}\norm{g}_{\cL}$. 
By the dominated convergence theorem, 
\begin{align}\nonumber
\lim_{n\to\infty}\bra{G^n_\al f_n,g_n}_{\cL_n}&=\int_0^{+\infty}e^{-\al t}\lim_{n\to\infty}\bra{T^n_t f_n,g_n}_{\cL_n}dt\\
&=\int_0^{+\infty}e^{-\al t}\bra{T_t f,g}_{\cL}dt=\bra{G_\al f,g}_{\cL}.
\end{align}
Lemma \ref{lem:W-SCV}-(iii) implies that $G^n_\alpha f_n$ strongly converges to $G_\alpha f$. This proves (1).

\textit{Proof of (1)$\Rightarrow$(3)}: We follow  Kato (p.504, Theorem IX.2.16).
Set for each $n$, $K^n=G^n_\al\Phi_n -\Phi_nG_\al$. Then for $f\in \cL$, $\lim_{n\to\infty}\norm{K^nf}_{\cL_n}=0$  and, using that $(\alpha-\la)G_\alpha^n$ and $(\alpha-\la)G_\alpha$ are contractions respectively on $\cL_n$ and $\cL$,
\be K:=\sup_n\norm{K^n}\le \sup_n\norm{\Phi_n} \frac{2}{\alpha-\la}<\infty,\ee
with  $\norm{K^n}:=\sup_{f\in \cL;\norm{f}_\cL=1}\norm{K^nf}_{\cL_n}$.

Let $(f_n)$ with $f_n\in \cL_n$ be a sequence strongly converging to a $f\in \cL$. For $T>0$ and $t\in [0,T]$,
\be
 \norm{T^n_t f_n-\Phi_nT_t f}_{\cL_n}
\le  \norm{T^n_t f_n-T^n_t \Phi_n f}_{\cL_n}+ \norm{T^n_t \Phi_n f-\Phi_nT_t f}_{\cL_n}. \ee
Since
\be \sup_{t\in [0,T]} \norm{T^n_t f_n- T^n_t\Phi_n f}_{\cL_n}\le e^{{\alpha_0} T}\norm{f_n-\Phi_n f}_{\cL_n}\conv{n}{\infty} 0\ee
To prove (3), it remains to prove that for all $f\in \cL$,
\be\label{eq:CVunifnT}
\lim_{n\to\infty} \sup_{t\in [0,T]} \norm{T^n_t \Phi_n f-\Phi_nT_t f}_{\cL_n}=0.\ee 
Suppose first that $f=G_\alpha g$ for some $g\in \cL$ and $\alpha>\la$. For $t\in[0,T]$
\be
T^n_t\Phi_n G_\al -\Phi_nT_tG_\al 
= -T_t^n K^n  +J^n_t+K^nT_t \ee
with $J^n_t=G^n_\al T^n_t\Phi_n -G^n_\al\Phi_nT_t$.
\be \sup_{t\in [0,T]} \norm{-T^n_t K^ng}_{\cL_n} \le e^{\la T}\norm{K^ng}_{\cL_n} \conv{n}{\infty} 0.\ee
We also have, for $t\in [0,T]$, $\lim_{n\to\infty}\norm{K^nT_tg}_{\cL_n}= 0$.
For $s,t\in [0,T]$,
\be\norm{K^nT_tg-K^nT_sg}_{\cL_n}\le K\norm{T_tg-T_sg}_{\cL}.\ee
Therefore, since $t\mapsto T_t g$ is uniformly continuous on $[0,T]$, the family of functions $t\mapsto K^nT_tg$ is uniformly equicontinuous on $[0,T]$. This permits to prove that $\lim_{n\to\infty} \sup_{t\in [0,T]}\norm{K^nT_tg}_{\cL_n}=0$.

We now prove that $\norm{J^n_tg}_{\cL_n}$ converges to $0$ uniformly on $[0,T]$.  
Suppose first that $g=G_\al h$ for some $h\in \cL$. 
\begin{lemma}\label{lem:technique}
For all $h\in \cL$, $n\geq 0$ and $t\geq 0$,
\begin{align}
G^n_\al (T^n_t\Phi_n -\Phi_nT_t)G_\al h=\int_0^tT_{t-s}^n(G^n_\al\Phi_n-\Phi_n G_\al)T_s h\,ds
\end{align}
\end{lemma}

\begin{proof}
Let us denote $F(s)=T^n_{t-s}G^n_\al \Phi_n G_\al T_sh$. We denote the generators $(A_n)$ (resp. $A$) of $T^n$ (resp. $T$).
We have
\begin{align}
\frac{d}{ds} F(s)=T^n_{t-s}(-A_nG^n_\al \Phi_n + \Phi_n AG_\al) T_sh
\end{align}
Then using the fact that $AG_\al=-I+\al G_\al$, we have
\begin{align}
\frac{d}{ds} F(s)=T^n_{t-s}(\Phi_nG_\al  -G^n_\al \Phi_n) T_s h
\end{align}
which implies the result when one integrates from $0$ to $t$.
\end{proof}

With this Lemma, we have that
\be
\norm{J^n_tG_\alpha h}_{\cL_n}\leq \int_0^t e^{\la (t-s)}\norm{K^n T_sh}_{\cL_n}\,ds
\leq Te^{\la T} \sup_{s\in [0,T]}\norm{K^nT_sh}_{\cL_n},
\ee
which converges to $0$. We have thus obtain \eqref{eq:CVunifnT} for $f\in G_\alpha G_\alpha(\cL)$.
Lemma \ref{lem:generator} implies that $G_\alpha G_\alpha(\cL)$ is dense in $\cL$. Let $f\in \cL$. 
For all $\epsilon>0$  there is $g\in G_\alpha G_\alpha(\cL)$ such that $\norm{g-f}_\cL\le \epsilon.$
Now, for $t\in [0,T]$
\begin{align}
 \norm{T^n_t \Phi_n f-\Phi_nT_t f}_{\cL_n}\le& 2\epsilon e^{{\alpha_0} T}\sup_n\norm{\Phi_n}+\norm{T^n_t \Phi_n f-\Phi_nT_t f}_{\cL_n}
\end{align}
and thus for all $\epsilon >0$,
\be \limsup_{n\to\infty}\sup_{t\in [0,T]}\norm{T^n_t \Phi_n f-\Phi_nT_t f}_{\cL_n} \le 2\epsilon e^{{\alpha_0} T}\sup_n\norm{\Phi_n} \ee
and therefore, we obtain \eqref{eq:CVunifnT} for all $f\in \cL$.
\end{proof}

\subsection{Main Application}\label{sec:contract}
Let $M$ be a locally compact separable metric space and $m$ a positive Borel measure on $M$ with full support.
Let $(\cE,\cH)$ be a Dirichlet form on $L^2(m)$ with domain $\cH$. In particular, $\cE$ satisfies $(\cE.1)$ for some constant ${\alpha_0}\in \bbR$ and, equipped with the inner product $\langle\cdot,\cdot\rangle_\cH:=\cE^s_{{\alpha_0}+1}$, $\cH$ is an Hilbert space.
Let $\cE^a$ be an antisymmetric bilinear form on $\cH$. Suppose that there is $K<\infty$ such that for all $u,v\in \cH$
\be\label{eq:asym} \abs{\cE^a(u,v)}\le K\cE_{{\alpha_0}+1}(u,u)^{1/2}\cE_{{\alpha_0}+1}(v,v)^{1/2}.\ee
Then, for all $\ka\in \bbR$, $(\cE^\ka:=\cE+\ka\cE^a,\cH)$ is a Dirichlet form on $L^2(m)$.

Let $\wt M$ be also a locally compact separable metric space and $\pi:M\to \wt M$ be a measurable mapping. Set $\wt m:=\pi_*m$ and suppose that $\hbox{Supp}[\wt m]=\wt M$.

Set
\be\label{eq:dom}
\wt \cH=\{\wt u\in L^2(\wt m)\,:\; \wt u\circ \pi\in \cH\}
\ee
and equip $\wt \cH$ with the inner product $\bra{\wt u,\wt v}_{\wt \cH}=\bra{\wt u\circ\pi,\wt v\circ\pi}_{\cH}$. Note that $\wt f\mapsto \wt f\circ \pi$ defines an isometry from $L^2(\wt m)$ onto $L^2(m)$.

\begin{lemma}
$\wt \cH$ is an Hilbert space.
\end{lemma}
\begin{proof}
The fact that $\wt \cH$ is prehilbertian is quite obvious. Let us now prove it is complete. Let $(\wt u_n)$ be a Cauchy sequence in $\wt \cH$. Then it is also a Cauchy sequence in $L^2(\wt m)$ which thus converges in $L^2(\wt m)$ to some $\wt u\in L^2(\wt m)$.
We also have that  $(\wt u_n\circ \pi)$ is a Cauchy sequence in $\cH$, hence it converges to some $u\in \cH$. 
Then it must hold that $u=\wt u\circ \pi$. Thus $\wt u\in \wt \cH$ and $\wt u_n$ converges in $\wt \cH$ to $\wt u$.
\end{proof}

Let $\wt{\cE}$ be the bilinear form on $\wt \cH$ defined by
\be\label{eq:df}
\wt{\cE}(\wt u,\wt v)=\cE(\wt u\circ \pi,\wt v\circ \pi),\qquad \hbox{ for } \wt u,\wt v\in \wt \cH.
\ee

\begin{proposition}\label{prop:DF2}
Assuming that $\wt \cH$ is dense in $L^2(\wt m)$, then $(\wt \cE,\wt \cH)$ is a Dirichlet form on $L^2(\wt m)$.
\end{proposition}

\begin{proof} Let ${\alpha_0}$ be such that $(\cE.1)$ and $(\cE.2)$ are satisfied by $(\cE,\cH)$.
Then $(\cE.1)$ and $(\cE.2)$ are satisfied by $(\wt \cE,\wt \cH)$ with this ${\alpha_0}$ (and with the same $K\ge 1$).
It is also straightforward to check that  $(\wt \cE,\wt \cH)$ satisfies $(\cE.3)$. 
\end{proof}

Set, for $\ka\in \bbR$, $\cL_\ka=L^2(m)$, and $\cL=L^2(\wt m)$. 
Then, if $\Phi_\ka:\cL\to \cL_\ka$ is the linear operator defined by $\Phi_\ka \wt f=\wt f\circ \pi$, then $\cL_\ka$ converges to $\cL$. Note that $\sup_\ka\norm{\Phi_\ka}<\infty$, thus Theorem \ref{th:correspondance} can be applied in this setting:
if $(\cE^\ka,\cH)$ Mosco-converges as $\ka \to\infty$, then the corresponding resolvents and semigroups strongly converge.

\begin{theorem}\label{thm:mosco2}
If one assumes that for all $u\in \cH$,
\be\label{eq:asym2} \cE^a(u,v)=0, \quad \forall v\in \cH \qquad\Longleftrightarrow \qquad \exists \wt u\in \wt \cH \hbox{ such that } u=\wt u\circ \pi \,\ee
then, as $\ka\to \infty$, $(\cE^\ka,\cH)$ Mosco-converges to $(\wt \cE,\wt \cH)$.
\end{theorem}
\begin{proof}
We recall that, for $u\in \cH$, $\Theta^\ka(u)=\sup_{\{v\in \cH;\cE^\ka_{{\alpha_0}+1}(v,v)=1\}}\cE^\ka_{{\alpha_0}+1}(v,u)$.
To prove that $(\cE^\ka,\cH)$ Mosco-converges to $(\wt \cE,\wt \cH)$, we have to prove that for any sequences $\kappa_n\uparrow\infty$, $(\cE^{\ka_n},\cH)$ Mosco-converges to $(\wt \cE,\wt \cH)$.
Using Theorem \ref{thm:mosco}, to prove the Mosco-convergence, it suffices to prove that for all sequence $\kappa_n\uparrow\infty$,
\begin{enumerate}
\item For all $\ka$, $\cE^{\ka,s}_{\alpha_0}$ is a positive definite closed form and the weak sector assumption is satisfied (with this ${\alpha_0}$, but with a constant $K_\ka \ge 1$). 
\item If $(u_\ka)$ with $u_\kappa\in\cH$ weakly converges to $\wt u$ and if $\liminf_{\ka\to\infty} \Theta^\ka(u_\ka)<\infty$, then $\wt u\in \wt \cH$.
\item For any sequence $(u_n)$ with $u_n\in \cH$  weakly converging to a $\wt u\in \wt \cH$ such that $\sup_n\Theta^{\ka_n}(u_n)<\infty$ and for any $\wt v\in \wt \cH$, there is a sequence $(v_n)$ with $v_n\in \cH$ strongly converging to $\wt v$ with \be\liminf_{n\to\infty}\cE^{\ka_n}(v_n,u_n)=\wt\cE(\wt v,\wt u).\ee
\end{enumerate}

(1) is immediate since $\cE^{\ka,s}=\cE^{0,s}$. Before proving (2) and (3), let us state a short lemma
\begin{lemma}\label{lem:temp}
Let $(u_\ka)$ be weakly converging to $\wt u$ such that $\liminf_{\ka\to\infty} \Theta^\ka(u_\ka)<\infty$, then $\wt u$ belongs to $\wt \cH$  and $(u_\ka)$ weakly converges in $\cH$ to $\wt u \circ \pi$.
\end{lemma}

\begin{proof}
Let $(u_\ka)$ be weakly converging to $\wt u$ such that $\liminf_{\ka\to\infty} \Theta^\ka(u_\ka)<\infty$. 
Then, for all $\wt v\in L^2(\wt m)$, $\langle u_\ka,\wt v\circ\pi\rangle_{L^2(m)}\to \langle \wt u,\wt v\rangle_{L^2(\wt m)}$ and $\sup_\ka \norm{u_\ka}_{L^2(m)}<\infty$. 
There is also a sequence $(u_{\ka_n})$ with $u_{\ka_n}\in \cH$, $\ka_n\to\infty$ and such that $\sup_n\norm{u_{\ka_n}}_{\cH}\le \sup_n\Theta^{\ka_n}(u_{\ka_n})<\infty$ by Lemma \ref{lem:theta}.
This implies that there is a subsequence, we also denote $(u_{\ka_n})$, weakly converging in $\cH$ (and in $L^2(m)$) to a $u\in \cH$.

This shows that for all $\wt v\in L^2(\wt m)$, $\langle u, \wt v\circ\pi\rangle_{L^2(m)}=\langle \wt u\circ\pi, \wt v\circ\pi\rangle_{L^2(m)}$, i.e. $\wt u\circ\pi$ is the orthogonal projection of $u$ onto $\{\wt v\circ\pi:\;\wt v\in L^2(\wt m)\}$.

Note also that $\lim_{n\to\infty}\cE^a(v,u_{\ka_n})=\cE^a(v,u)$, for all $v\in \cH$. However, since $\sup_n\Theta^{\ka_n}(u_{\ka_n})<\infty$, it must holds that $\lim_{n\to\infty}\cE^a(v,u_{\ka_n})=0$, and thus $\cE^a(v,u)=0$ for all $v\in \cH$. This shows that $u=\wt u\circ \pi$.
 \end{proof}

\smallskip
{\it Proof of (2)}\,: It is a direct consequence of Lemma \ref{lem:temp}.

\smallskip
{\it Proof of (3)}\,: Let $(u_n)$ with $u_n\in \cH$ be weakly converging to a $\wt u\in \wt \cH$ such that $\liminf_{n\to\infty} \Theta^{\ka_n}(u_n)<\infty$. 
Lemma \ref{lem:temp} implies that for all $v\in L^2(m)$, 
\be\label{eq:projection}\lim_{n\to\infty}\langle  v,u_n\rangle_{L^2(m)}=\langle  v,\wt u\circ \pi\rangle_{L^2(m)}.\ee

Let $\wt v\in \wt \cH$ and set $v_n=\wt v\circ\pi$ for all $n$. Then $(v_n)$ strongly converges to $\wt v$. Moreover, 
$\cE^{\ka_n}(v_n,u_n)
=\cE(\wt v\circ\pi,u_n)$.
By Lemma \ref{lem:temp}, $u_n$ weakly converges in $\cH$ to $\wt u\circ \pi$. Thus $\lim_{n\to\infty} \cE^{\ka_n}(v_n,u_n)=\cE(\wt v\circ\pi,\wt u\circ\pi)=\wt \cE(\wt v,\wt u)$.
\end{proof}

Suppose that $(\cE,\cH)$ and $(\wt \cE,\wt \cH)$ are regular. Then $(\cE^{\ka},\cH)$ is also regular for all $\kappa\in\mathbb{R}$. Applying Theorem \ref{th:diff}, for all $\kappa\in\mathbb{R}$, to the Dirichlet form $(\cE^{\ka},\cH)$ (resp. $(\wt\cE,\wt\cH)$) is associated $X^\ka$ (resp. $\wt X$), a Hunt process on $M$ (resp. $\wt M$) with life time $\zeta^\kappa$ (resp. $\wt \zeta$) taking its values in $M\cup\{\Delta\}$ (resp. $\wt M\cup \{\wt \Delta\})$, with $\Delta$ (resp. $\wt \Delta$) a cemetery point we adjoin to $M$ (resp. $\wt M$). We extend $\pi$ to $M\cup\{\Delta\}$ by setting $\pi(\Delta)=\wt \Delta$.
The processes $\pi(X^\ka)$ and $\wt X$ are random variables taking their values in $D(\mathbb{R}^+,\wt M\cup\{\wt\Delta\})$, the set of c\`adl\`ag functions from $\mathbb{R}^+$ to $\wt M\cup\{\wt\Delta\}$, equipped with the Skorohod topology. 

\begin{theorem}\label{th:finite}
Assume that 
\begin{itemize}
\item $(\cE,\cH)$ and $(\wt\cE, \wt \cH)$ are regular and that $\cE^a$ satisfies \eqref{eq:asym} and \eqref{eq:asym2};
\item $\pi(X_0^\ka)$ converges in law to $\wt X_0$;
\item For all $\ka$, the law of $X_0^\ka$ (resp. $\wt X_0$) has a density with respect to $m$ (resp. to $\wt m$) and this density belongs to $L^2(m)$ (resp. to $L^2(\wt m)$).  
\end{itemize}
Then the finite dimensional distributions of  $\pi(X^\ka)$ converges to those of $\wt X$ as $\ka$ goes to $+\infty$.
\end{theorem}

\begin{proof}
It is a consequence of Theorem \ref{thm:mosco2}, Theorem \ref{th:correspondance} and of the definition of strong convergence.
\end{proof}

Let us finish this section by pointing out some properties of the Dirichlet form $(\wt \cE, \wt \cH)$.
\begin{lemma}\label{lem:core} Suppose that $\pi:M\to\wt M$ is continuous and that $(\cE,\cH)$ possesses the local property, then $(\wt \cE,\wt \cH)$ possesses the local property. \end{lemma}
\begin{proof}
Straightforward after having remarked that for all $\wt u\in \wt \cH$, $\hbox{Supp}(\wt u\circ \pi)=\pi^{-1}(\hbox{Supp}(\wt u))$.
\end{proof}

\begin{remark}
Let us consider $\sigma(\pi)$ the smallest complete $\s$-algebra on $M$ such that $\pi$ is measurable. 
Then $u:M\to\bbR$ is $\sigma(\pi)$ measurable if and only if $u$ is Borel measurable and there is $\wt u$ measurable on $\wt M$ such that $u=\wt u\circ \pi$ $m$-a.e. 
Then, one can introduce the space $L^2(m,\sigma(\pi))$ of functions in $L^2(m)$ which are $\sigma(\pi)$ measurable.
 For $u\in L^2(m)$, we denote $\Pi u$ the orthogonal projection of $u$ onto $L^2(m,\sigma(\pi))$ (i.e the conditional expectation $\bbE(u|\sigma(\pi))$ if $m$ is a probability measure). 
Assumption of Proposition \ref{prop:DF2} and Assumptions $(2)$ and $(3)$ of Theorem \ref{th:finite} can be formulated using this orthogonal projection:
\begin{enumerate}
\item If one assumes that for $u\in \cH$, there is a version of $\Pi u$ in $\cH$ and that $u\mapsto\Pi u$ is a bounded operator from $\cH$ to $\cH$ (i.e. $\norm{\Pi}_\cH<\infty$) then $\wt \cH$ is dense in $L^2(\wt m)$. This is equivalent to: there exists $C<\infty$ such that for all $u\in \cH$, $\Pi u\in \cH$ and
\be\label{eq:h1}
\cE (\Pi u,\Pi u)\leq C\cE(u,u).
\ee
\item Additionally, suppose that $(\cE,\cH)$ is regular. Assume also that for $u\in C_c(M)$ there is a version of $\Pi u$ in $C_c(M)$, that $u\mapsto\Pi u$ is a bounded operator from $C_c(M)$ to $C_c(M)$ (i.e. there is $c>0$ such that $\norm{\Pi u}_\infty\leq c\norm{u}_\infty$ for all $u\in C_c(M)$), that $\pi:M\to\wt M$ is continuous and that for all compact $K$ in $\wt M$, $\pi^{-1}(K)$ is compact in $M$.
Then $(\wt \cE,\wt \cH)$ is regular.
\item Condition in Equation \eqref{eq:asym2} is equivalent to $L^2(m,\sigma(\pi))\cap \cH$ is the kernel of $\cE^a$.
\end{enumerate}

\begin{proof}[Proof of (1)]
Let $\wt u\in L^2(\wt m)$ and set $u:=\wt u\circ \pi\in L^2(m)$. Since $\cH$ is dense in $L^2(m)$, 
there exists a sequence $(u_n)$ in $\cH$ converging to $u$ in $L^2(m)$.
Let $\wt{u}_n\in L^2(\wt m)$ be defined by 
$\wt u_n\circ\pi=\Pi u_n$. Then, by assumption, $\Pi u_n\in \cH$ and that entails $\wt u_n\in\wt \cH$. Using that $\wt u\circ\pi=\Pi u$ and that the orthogonal projection $\Pi$ from $L^2(m)$ onto $L^2(m,\sigma(\pi))$ is bounded,
\be \norm{\wt u-\wt u_n}_{L^2(\wt m)}
= \norm{\Pi(u-u_n)}_{L^2(m)}\le\norm{u-u_n}_{L^2(m)}. \ee
Thus $(\wt u_n)$ converges to $\wt u$ in $L^2(\wt m)$. This proves that $\wt \cH$ is dense in $L^2(\wt m)$.
\end{proof}

\begin{proof}[Proof of (2)] Let us first prove that $C_c(\wt M)\cap \wt \cH$ is dense in $C_c(\wt M)$.
Let $\wt u\in C_c(\wt M)$, then $u:=\wt u\circ \pi\in C_c(M)$
and there is a sequence $(u_n)$ in $C_c(M)\cap \cH$ converging to $u$ in $C_c(M)$. 
Define $\wt u_n$ by $\wt u_n\circ\pi=\Pi u_n$. Then $\wt u_n\in C(\wt M)\cap \wt \cH$ since $\wt u_n\circ \pi=\Pi u_n\in \cH$.  Thus $\norm{\wt u_n-\wt u}_\infty\le c\norm{u_n-u}_\infty$ and $(\wt u_n)$ converges to $\wt u$ in $C_c(\wt M)$.

Let us now prove that $C_c(\wt M)\cap \wt \cH$ is dense in $\wt \cH$.
Let $\wt u\in\wt \cH$, then set $u:=\wt u\circ \pi\in \cH$, there is a sequence $(u_n)$ in $C_c(M)\cap \cH$ converging to $u$ in $\cH$. Again, define $\wt u_n$ by $\wt u_n\circ\pi=\Pi u_n$. 
Then $\wt u_n\in C_c(\wt M)\cap \wt \cH$ and $\norm{\wt u_n-\wt u}_{\wt \cH}\le \norm{\Pi}_\cH \norm{u_n-u}_\cH$, which implies that $(\wt u_n)$ converges to $\wt u$ in $\wt \cH$.
\end{proof}
\end{remark}

\section{An averaging principle for diffusions}\label{sec:avg}

In this section, $M$ is a smooth oriented Riemannian manifold. Denote by $\cB$ its Borel $\s$-algebra, by $\cM(M)$ (resp. $\cM_+(M)$) the set of signed (resp. nonnegative) Borel measures on $M$, and by $\cP(M)$ the subset of $\cM_+(M)$ of probability measures.
We equip these sets with the narrow topology (i.e. the weak-$*$ topology with respect to continuous bounded functions on $M$). 

Let $m\in \cM_+(M)$ be equivalent to the volume form on $M$, with a positive $C^1$ density also denoted $m$ for the sake of simplicity.
We will say that a set $N$ is negligible if it is $m$-negligible i.e. if $m(N)=0$.


\subsection{Ergodic decomposition}\label{sec:ergdec}

Let $V$ be a $C^1$ vector field on $M$. Suppose that $V$ is complete, i.e. $V$ generates a flow of diffeomorphisms $\phi=(\phi_t)_{t\in\bbR}$ on $M$.

Let us recall the following definitions. 
\begin{itemize}
\item A set $A\in\mathcal{B}$ is called {\bf invariant} if $\phi^{-1}_t(A)=A$ for all $t$.
\end{itemize}

Let $\mu\in \mathcal{M}_+(M)$.  
\begin{itemize}
\item $\mu$ is called {\bf invariant} if for all $A\in\mathcal{B}$,  $\mu(\phi^{-1}_t(A))=\mu(A)$ for all $t$ (i.e. ${\phi_t}_*\mu=\mu$ for all $t$).
\item $\mu$ is called {\bf ergodic} if $\mu$ is an invariant probability measure such that for all invariant set $A\in \cB$, then either $\mu(A)=0$ or $\mu(A)=1$.
\item A Borel function $u$ is called {\bf $\mu$-almost invariant} if, for all $t$, $u=u\circ\phi_t$ $\mu$-a.e.
\end{itemize}
Henceforth, we denote by $E$ the set of ergodic probability measures.

Using the fact that for $u\in C_c^1(M)$, $\int_M Vu \,dm=-\int_M u \,\Div_m V\,dm$, one can prove
\begin{proposition} \label{divmF=0}
$m$ is invariant if and only if $\Div_m V= m^{-1}\Div(mV) =0$.
\end{proposition}

We have the following proposition (see \cite{hasselblatt} Section 3.6.a)
\begin{proposition}\label{prop:inv}
An invariant probability measure $\mu$ is ergodic if and only if every $\mu$-almost invariant function $f$ is $\mu$-essentially constant, i.e. there is a constant $c$ such that $f(x)=c$, $\mu(dx)$-a.e.
\end{proposition}

\begin{proposition}\label{prop:ED}
Let $C$ be the subset of $\cP(M)$ of invariant probability measures. Suppose that $C$ is non empty. Then 
\begin{enumerate}
\item$C$ is convex and the set of its extreme points is $E$.
\item $E$ is a $G_{\de}$-subset of $\cP(M)$ and therefore is a Polish space (metrizable, complete, separable) with the same topology as $\cP(M)$. 
\item Moreover for any invariant probability measure $\mu$, there is a unique probability measure $\wh \mu$ on $E$ such that 
$$
\mu(dx)=\int_{E}p(dx)\wh \mu(d p).
$$
\end{enumerate}
\end{proposition}
\begin{proof}
The first item is a consequence of Proposition \ref{prop:inv} (see \cite{hasselblatt}, Theorem 4.2.4 or \cite{phelps}, Proposition 12.4). 
Item (2) and existence in (3) are consequences of Choquet's Theorem on the structure of convex subsets of locally convex topological linear spaces (Theorem 4.2 in \cite{bishop.deleeuw59}). 
The uniqueness is a consequence of Choquet's Theorem given in \cite{phelps} p.60 Section 10, see also Theorem p.77 in Section 12, using item (2).
\end{proof}


\begin{assumption} \label{hyp:m}
$m$ is invariant and there is an increasing sequence of compact invariant sets $(M_n)_{n\ge 1}$ with $\bigcup_{n\ge 1} M_n=M$.
\end{assumption}

Note that this assumption is satisfied as soon as there is a $C^1$-function $H:M\to \mathbb{R}$, such that $VH=0$ and that for all $n\ge 1$, $M_n:=\{x\in M : H(x)\le n\}$ is compact (i.e. $\lim_{x\to\infty} H(x)=\infty$).

Suppose that $m$ is an invariant probability. Let us denote by $\cI$, the $\s$-field generated by the invariant sets of $\cB$ and completed by the negligible sets of $\cB$.
For $p\in [1,\infty]$, let $L^p_{inv}(m)$ be the space of $m$-almost invariant functions in $L^p(m)$, or equivalently of $\cI$-measurable functions in $L^p(m)$.
For $f\in L^1(m)$, denote by $\Pi f$ the conditional expectation $\bbE_m[f|\cI]$ of $f$ with respect to $\cI$ and $m$.
Then $\Pi f\in L^1_{inv}(m)$, and when $f\in L^2$, $\Pi f$ is the orthogonal projection on $L^2_{inv}(m)$.

From \cite{neveu} (corollary of Proposition V-4-4), there is a function $P: M\times\cB\to [0,1]$, called a regular conditional probability with respect to $\cI$ and $m$, such that
\begin{itemize}
\item for all $x\in M$, $P(x,\cdot)$ is a probability measure on $\cB$
\item for all $A\in \cB$, $P(\cdot,A)$ is a $\cI$-measurable version of $\bbE_m[\1_A|\cI]$ (in the probability space $(M,\cB,m)$).
\end{itemize}
If $P$ is a regular conditional probability with respect to $\cI$ and $m$, then for any $f\in L^1(m)$, $x\mapsto Pf(x):=\int_Mf(y)P(x,dy)$ is a version of $\bbE_m[f|\cI]$.

\begin{theorem}\label{THM:CB} Assume that $m$ is an invariant probability measure, then there is a regular conditional probability $P$ with respect to $\cI$ and $m$ such that
\begin{enumerate}
\item $P$ is a measurable map from $ M$ to $E$;
\item For $f\in L^1(m)$, we have
\be
Pf(x)=\lim_{t\to\infty}\frac1t\int_0^tf(\phi_s(x)) ds \qquad m(dx)-a.e.
\ee
The convergence \eqref{eq:conv} holds in $L^p(m)$ for $f\in L^p(m)$ and $1\le p<\infty$;
\item Let $\wh m$ be the unique probability measure on $E$ such that $m=\int_{E}p\,\wh m(d p)$.
We have $P_*m=\wh m$ and $P$ induces a bijection between $L^2_{inv}(m)$ and $L^2(\wh m)$.
\end{enumerate}
\end{theorem}

\begin{proof}
Let us recall Birkhoff's Theorem (\cite{hasselblatt} Theorem 3.5.2): if $\mu$ is an invariant probability measure and $f\in L^1(\mu)$, then
\be\label{eq:conv}
\bbE_{\mu}[f|\cI](x)=\lim_{t\to\infty}\frac1t\int_0^tf(\phi_s(x))ds \qquad \mu(dx)-a.e.
\ee

{\it Item (1)}: using \eqref{eq:conv} with $\mu=m$, the construction of a regular conditional probability $P$ with respect to $\cI$ and $m$ given by Neveu (\cite{neveu}, corollary of Proposition V-4-4) can be modified such that $P:M\to \cP(M)$ defined by $P(x)(A)=P(x,A)$ is $\cI$-measurable and $P(x)$ is invariant for all $x\in M$.
Using again \eqref{eq:conv} with $\mu=P(x)$, we obtain that, $m(dx)$-a.e., $P$ is also a regular conditional probability with respect to $\cI$ and $P(x)$. Thus, using \cite{raugi}, Proposition 3.3, it proves that $P(x)$ is ergodic $m(dx)$-a.e, and $P$ can modified such that $P(x)\in E$ for all $x\in M$.

{\it Item (2)} follows from \eqref{eq:conv} with $\mu=m$ and $P$ defined in item (1). The convergence in $L^p$ is von Neumann's ergodic Theorem (see \cite{hasselblatt} Theorem 3.5.1).

{\it Item (3)} is a consequence of the uniqueness of $\wh m$ from Proposition \ref{prop:ED}.
The bijection between $L^2_{inv}(m)$ and $L^2(\wh m)$ is then defined by $u\mapsto \hat{u}$, with $\hat{u}(p)=pu$ and then when $u$ is invariant, $\hat{u}\circ P=u$.
\end{proof}

\begin{remark}
Note that if $P$ and $Q$ are regular conditional probabilities with respect to $\cI$ and $m$ then for all $f\in L^1$, $m(dx)$-a.e., $Pf(x)=Qf(x)$. 
\end{remark}

\begin{proposition}\label{prop:P}
If Assumption \ref{hyp:m} holds, then there is a measurable map $P:M\to E$ such that for all $f\in L^1(m)$, we have
\be\label{eq:43}
Pf(x)=\lim_{t\to\infty}\frac1t\int_0^tf(\phi_s(x))ds \qquad m(dx)-a.e.
\ee
Moreover, for all $f\in L^2(m)$, $Pf$ is the orthogonal projection of $f$ in $L^2(m)$ onto the space of $\mathcal{I}$-measurable functions. In particular, $\langle Pf,g\rangle_{L^2(m)}=\langle Pf,Pg\rangle_{L^2(m)}$ for all $f,g$ in $L^2(m)$.
\end{proposition}

\begin{proof}
If $m$ is an invariant probability measure, this proposition is just Theorem \ref{THM:CB}.
Suppose now that Assumption \ref{hyp:m}-(2) is satisfied. Then, $m_n=m_{|M_n}/m(M_n)$ is an invariant probability measure on $M_n$ to which we apply Theorem \ref{THM:CB}. Let $P_n$ be a corresponding regular conditional probability satisfying the statement of Theorem \ref{THM:CB} and let us define $P(x)=P_{n+1}(x)$ for $x\in M_{n+1}\setminus M_n$, so that $P:M\to E$ is measurable. 
Note that for all $f\in L^1(m)$ and $n\ge 1$, $f_n:=f\1_{M_n}\in L^1(m_n)$ and for $m_n(dx)$-a.e., $Pf=Pf_n=P_nf_n$.

Let now $f\in L^2(m)$ and $g$ a $\mathcal{I}$-measurable function in $L^2(m)$, and set for all $n$, $f_n:=f\1_{M_n}$ and $g_n:=g\1_{M_n}$.
Since $\|Pf\|_{L^2(m)}=\lim_{n\to\infty} \|\1_{M_n}P_nf_n\|_{L^2(m)}$ and since for all $n$, $\|\1_{M_n}P_nf_n\|^2_{L^2(m)}\le 
m(M_n)\|P_n f_n\|^2_{L^2(m_n)}\le m(M_n)\|f_n\|^2_{L^2(m_n)}\le\|f\|^2_{L^2(m)}$, we have $Pf\in L^2(m)$ and
\begin{align*}
\langle Pf,g\rangle_{L^2(m)}
&= \lim_{n\to\infty} \langle Pf_n,g_n\rangle_{L^2(m)}\\
&= \lim_{n\to\infty} m(M_n)\langle P_nf_n,g_n\rangle_{L^2(m_n)}\\
&= \lim_{n\to\infty} m(M_n)\langle f_n,g_n\rangle_{L^2(m_n)} = \langle f,g\rangle_{L^2(m)},
\end{align*}
where we have used that $P_nf_n$ is a version of $\bbE_{m_n}[f_n|\cI]$. 
\end{proof}

\subsection{Main Results}

Let $\Gamma$ be a carr\'e du champ operator associated to a second order differential operator $S$ as in Section \ref{sec:DFM} and let $V_0$ be a $C^1$-vector field. Suppose that $\Gamma$ and $V_0$ are such that \eqref{eq:hyp(3)} and \eqref{eq:sector1} are satisfied.
Applying Proposition \ref{prop:DF0}, set $(\cE,\mathcal{H})$ the Dirichlet form on $L^2(m)$ associated to $\Gamma$ and $V_0$.

Let $V$ be a $C^1$ complete vector field on $M$. Suppose that
\begin{itemize}
\item $m$ is invariant for the flow $(\phi_\cdot)$ associated to $V$, i.e. $\hbox{div}_m V=0$;
\item the weak sector condition \eqref{eq:sector1} is satisfied by $V$;
\item $f\circ \phi_t \in \mathcal{H}$ for all $f\in C^\infty_c(M)$ and all $t\in\mathbb{R}$.
\end{itemize}
Note that this last condition is satisfied for example if $V$ is $C^\infty$ or if $M$ is an open subset of $\mathbb{R}^d$ and if $\mathcal{H}\supset \mathcal{H}^1_{loc}(M)$ (which is satisfied when $\Gamma$ is uniformly elliptic).

For $\kappa\in\bbR$, define the bilinear form $\cE^\kappa:=\cE+\kappa \cE^V$, with $\cE^{V}$ the form defined by
\be\label{eq:EV}
\cE^{V}(f,g)=-\langle Vf,g\rangle_{L^2(m)},\quad\hbox{ for } f,g\in \cH.
\ee
$\cE^V$ is antisymmetric. Then, by Proposition \ref{prop:DF0}, it holds that $(\cE^{\ka},\mathcal{H})$ is a regular Dirichlet form on $L^2(m)$, for all $\kappa\in\mathbb{R}$.

We denote by $A$ (resp. $A^\ka$) the infinitesimal generator of $(T_t)$ (resp. $T_t^\ka$) associated to $\cE$ (resp. to $\cE^\ka$). Then, for all $\ka$, it holds that $A$ and $A^\ka$ have the same domain (i.e. $\cD(A^\ka)=\cD(A)$) 	and $A^\ka=A+\ka V$. 
Note that $C^\infty_c(M)\subset \cD(A)\subset \cH$, and for $f\in C^\infty_c(M)$, $Af=(S+V_0) f$.

\smallskip
We suppose 
\begin{assumption}\label{hyp:F}
There is a  $C^2$-function $H:M\to [0,+\infty[$, such that
\begin{itemize}
\item $VH=0$;
\item There is $\lambda>0$ such that for all $x\in M$, $AH(x)\le \lambda(1+H(x))$;
\item For all $n\ge 1$, $M_n:=\{x\in M : H(x)\le n\}$ is compact.
\end{itemize}
\end{assumption}
Note that this assumption ensures that Assumption \ref{hyp:m} holds. 
We also suppose
\begin{assumption}\label{hyp:pi}
There is a continuous application $\pi:M\to \wt M$, with $\wt M$ a locally compact, separable, metric space, and a measurable application $p: \wt M\to  E $ such that $P:=p\circ \pi$ is a regular conditional probability with respect to $\cI$ and $m$.
\end{assumption}
Note that this assumption is satisfied when $P$ is continuous, by simply taking $\pi=P$, $\wt M=E$ and $p(\mu)=\mu$ for all $\mu\in E$.

Let us now apply the results of Section \ref{sec:contract} with $\pi:M\to \wt M$, define $\wt \cH$ as in \eqref{eq:dom}, and set $\wt m=\pi_*m$ a measure on $\wt M$. Finally, we suppose
\begin{assumption}\label{hyp:tight}
\begin{enumerate}[(i)]
\item $C_c(\wt M)\cap \wt \cH$ is dense in $L^2(\wt m)$.
\item $C_c(\wt M)\cap \wt \cH$ is dense in $\wt \cH$.
\item There is a vector space $\wt C$ subset of $\{\wt u\in \wt \cH, \wt u\circ \pi\in C^2_c(M)\}$ dense in $C_c(\wt M)$. 
\end{enumerate} 
\end{assumption}

\begin{proposition}\label{prop:uinv} Let $u\in \cH$. Then the following items (i), (ii) and (iii) are equivalent.
\begin{enumerate}[(i)]
\item $u$ is invariant;
\item There is $\wt u\in\wt \cH$ such that $u=\wt u\circ \pi$ $m$-a.e.;
\item For all $v\in \mathcal{H}$, $\mathcal{E}^V(u,v)=0$.
\end{enumerate} 
\end{proposition}

\begin{proof}
Fix $u\in \mathcal{H}$. 

$(i)\Rightarrow (ii)$:
If $u\in \mathcal{H}$ is invariant, then $Pu=u$ by \eqref{eq:43}.
Define $\wt u:\wt M\to \mathbb{R}$ by $\wt u(\wt x)=p(\wt x)u$ (recall that $p(\wt x)$ is a probability measure on $M$).
Then $m(dx)$-a.e., $\wt u\circ \pi(x)=p\circ \pi(x) u=Pu(x)=u(x)$.
Moreover,
$\wt u\in L^2(\wt m)$ since  $\norm{\wt u}^2_{L^2(\wt m)}=\int (p(\wt x)u)^2 \wt m(d\wt x) = \int (Pu)^2(x) \, m(d x) = \norm{u}^2_{L^2(m)}$.
We thus have that, $\wt u\circ \pi=u$  and $\wt u\in \wt \cH$. 

$(ii)\Rightarrow (iii)$: If $u\in\mathcal{H}$ is such that there is $\wt u\in\wt \cH$ for which $u=\wt u\circ \pi$ $m$-a.e. then for all $v\in C_c^\infty(M)$, 
\begin{align*}
\mathcal{E}^V(u,v)= -\langle u, Vv\rangle_{L^2(m)}= - \langle Pu, Vv\rangle_{L^2(m)}= - \langle Pu, PVv\rangle_{L^2(m)}.
\end{align*}
Since $m(dx)$-a.e., 
\begin{align}
PVv(x)=&\lim_{t\to\infty} \frac{1}{t}\int_0^\infty Vv(\phi_s(x)) ds
= \lim_{t\to\infty} \frac{1}{t}\int_0^\infty \frac{d}{ds} v(\phi_s(x)) ds\\
=& \lim_{t\to\infty} \frac{1}{t}\left( v(\phi_t(x))-v(x)\right) ds=0.
\end{align}
Therefore $\mathcal{E}^V(u,v)=0$ for all $v\in C_c^\infty(M)$ and by density of $C_c^\infty(M)$ in $\mathcal{H}$ we prove (iii).

$(iii)\Rightarrow (i)$: 
Note that for all $u\in C^\infty_c(M)$ and $v\in C_c^\infty(M)$ it holds that for all $t>0$,  
\begin{align}
\bra{u\circ \phi_t,v}_{L^2(m)}-\bra{u,v}_{L^2(m)}&=\int_0^t\bra{Vu\circ \phi_s,v}_{L^2(m)}\,ds\\
&=\int_0^t\bra{Vu,v\circ \phi_s^{-1}}_{L^2(m)}\,ds.
\label{eq:uinv0}
\end{align}
Since $v\circ \phi_s^{-1}\in \mathcal{H}$, we get that for all $u\in \mathcal{H}$ and $v\in C^\infty_c(M)$,
\be \bra{u\circ \phi_t,v}_{L^2(m)}-\bra{u,v}_{L^2(m)}=\int_0^t\mathcal{E}^V(u,v\circ \phi_s^{-1}) \,ds. \label{eq:uinv1}\ee
Equation \eqref{eq:uinv1} implies that if (iii) holds for some $u\in \mathcal{H}$, then for all $v\in C^\infty_c(M)$, $\bra{u\circ \phi_t,v}_{L^2(m)}=\bra{u,v}_{L^2(m)}$, and therefore that $u$ is invariant.
\end{proof}


\begin{remark}
When condition (1) of Proposition \ref{prop:WScond} is satisfied, then for all $u\in \mathcal{H}$, $Vu\in L^2(m)$ and $u\in \mathcal{H}$ is invariant if and only if $Vu=0$.
\end{remark}

Proposition \ref{prop:uinv} proves that $\cE^{V}$ satisfies \eqref{eq:asym2}.

Let $(\wt \cE,\wt \cH)$ be the Dirichlet form on $L^2(\wt m)$ obtained by contracting $(\cE,\cH)$ on $\wt M$ using Proposition \ref{prop:DF2}. To apply Theorem \ref{th:finite} we need to prove that $(\wt \cE,\wt \cH)$ is regular.

\begin{proposition}\label{prop:regular}
The Dirichlet form $(\wt \cE,\wt \cH)$ is regular and possesses the local property. Moreover, $(\wt \cE,\wt \cH)$ is a contraction of  $(\cE^\ka,\cH)$ for all $\ka$.
\end{proposition}
\begin{proof} 
Assumption \ref{hyp:tight} ensures that $(\wt\cE, \wt \cH)$ is regular. 
By Lemma \ref{lem:core}, $(\wt\cE, \wt \cH)$ possesses the local property. 
Then, since for all $\wt u\in \wt \cH$, $\wt u\circ\pi\in \cH$ is invariant and thus
$\cE^a(\wt u\circ \pi,\wt v\circ \pi)=0$ for all $\wt v\in \wt\cH$, the contraction of $(\cE^\ka,\cH)$ is $(\wt \cE,\wt \cH)$ for all $\ka$.
\end{proof}

Note that Assumption \ref{hyp:F} ensures that the Dirichlet forms $(\cE^\kappa,\cH)$ are conservatives for all $\kappa$ (i.e. the life times of their associated Hunt processes are infinite a.s.).
Denote by $\wt A$ the infinitesimal generator of the semigroup associated to $(\wt \cE,\wt \cH)$. Denote by $\wt H$ the mapping on $\wt M$ defined by $\wt H(\wt x)= p(\wt x) H$. Then Assumption \ref{hyp:F} implies that $\wt H\circ\pi=H$, $\wt H\in \cD(\wt A)$  and $\wt A \wt H\le \lambda(1+\wt H)$, which implies that $(\wt \cE,\wt \cH)$ is also conservative.  

By Theorem \ref{th:diff}, for all $\ka$, there is $X^\ka$ a diffusion on $M$ associated to $\cE^\ka$ and $\wt X$ a diffusion on $\wt M$ associated to $\wt \cE$.
The processes $\pi(X^\ka)$ and $\wt X$ are random variables taking their values in $C(\mathbb{R}^+,\wt M)$ equipped with the topology of uniform convergence on compact sets. Our main result is
\begin{theorem}\label{th:main}
Assume that 
\begin{itemize}
\item $\pi(X_0^\ka)$ converges in law to $\wt X_0$;
\item For all $\ka$, the law of $X_0^\ka$ (resp. $\wt X_0$) has a density with respect to $m$ (resp. to $\wt m$) and this density belongs to $L^2(m)$ (resp. to $L^2(\wt m)$);
\item $\sup_\ka \bbE[H(X_0^\ka)]<\infty$.
\end{itemize}
Then, $\pi(X_t^\ka)_{t\ge 0}$ converges in law to $(\wt X_t)_{t\ge 0}$ as $\ka\to\infty$.
\end{theorem}

The proof follows usual technics: we prove first the tightness and then apply the convergence of the finite dimensional distributions (proved in Theorem \ref{th:finite}).

\begin{proposition}\label{prop:tight}
 $\{\pi(X^{\ka}),\;\ka\ge 0\}$ is tight in $C(\mathbb{R}^+,\wt M)$.
\end{proposition}

\begin{proof} 
Note that, $\pi$ being continuous, we have $\pi(M_n)$ is compact for all $n$.

To prove this proposition, we apply Theorem 9.1 of \cite{ethier.kurtz}, which states that the collection $\{\pi(X^{\ka}),\;\ka\ge 0\}$ is tight in $C(\mathbb{R}^+,\wt M)$ as soon as for all $T>0$,
\be\label{eq:tight}\lim_{n \to\infty} \sup_\ka \bbP\left[\pi(X^\ka_s)\not\in\pi(M_n) \hbox{ for some } s\in [0,T]\right]=0\ee
and $\{\wt u\circ \pi(X^{\ka}),\;\ka\ge 0\}$ is tight in $C(\bbR^+, \bbR)$ for all $\wt u\in \wt C$.

Note that \eqref{eq:tight} holds when $M$ is compact (and so $\pi(M)$ is compact). 
When $M$ is not compact. By applying the Itô Formula, $Y_t^{\ka}:=e^{-\lambda t}(H(X_t^{\ka})+1)$ is a positive supermartingale. Set $\tau^{\ka}_n=\inf\{t,H( X_t^{\ka})\ge n\}$. Then, $\tau_n^\ka$ is a stopping time and
\be
\bbE[Y^{\ka}_0]\ge\bbE[Y^{\ka}_{T\wedge\tau_n^{\ka}}]\ge\bbE[Y^{\ka}_{\tau_n^{\ka}}\1_{T\ge \tau_n^{\ka}}]\ge (n+1)e^{-\lambda T}\bbP[T\ge \tau_n^{\ka}]
\ee
Thus 
\be
\bbP[\sup_{t\le T}H(X^\ka_s)\ge n]=\bbP[T\ge \tau_n^{\ka}] \le \frac{e^{\lambda T}\bbE[H(X_0^\ka)+1]}{n+1}.
\ee
Using that $\sup_{\ka}\bbE[H(X_0^{\ka})]<+\infty$, we get that
$$\lim_{n \to\infty} \sup_\ka \bbP[\sup_{t\le T}H(X^\ka_s)\ge n]=0.$$
Note that $H(x)\le n$ implies that $\pi(x)\in \pi(M_n)$. 
Therefore \eqref{eq:tight} is proved.

Let $\wt u\in \wt C$. Then $\wt u\circ\pi=\pi u\in C^2_c(M)\subset \cD(A)$ by Assumption \ref{hyp:tight}, $A^{\ka}(\wt u\circ \pi)=A(\wt u\circ \pi)$, and
\be M_t^\ka := \wt u\circ \pi(X_t^{\ka}) - \wt u\circ \pi(X_0^{\ka}) - \int_0^tA(\wt u\circ \pi)(X^{\ka}_s)\,ds\ee
is a continuous local martingale with quadratic variation given by
\be
\bra{M^\ka}_t=\int_0^t\G(\wt u\circ \pi,\wt u\circ \pi)(X^{\ka}_s)\,ds.
\ee
Fix $T>0$.
Using that $\wt u\circ \pi\in C^2_c$, using Burkh\"older-Davies-Gundy inequality, for $\gamma> 2$ and for all $0\le s<t\le T$,
$$\bbE[|(\wt u\circ \pi(X_t^{\ka})-\wt u\circ \pi(X_s^{\ka})|^\gamma]
\le C_{\gamma,T} (t-s)^{\frac{\gamma}{2}}$$
where $C_{\gamma,T}$ is a constant depending only on $\gamma$, $T$ and on $\wt u\circ \pi$.
Using Theorem 12.3 (and the remark that follows this theorem) of Billingsley \cite{billingsley}, we get that $\{\wt u\circ \pi(X^{\ka}),\;\ka\ge 0\}$ is tight in $C(\bbR^+, \bbR)$. This finishes the proof of the proposition.
\end{proof}

\begin{proof}[Proof of Theorem \ref{th:main}]
We use Theorem 8.1 Chap. 2 of \cite{billingsley}, which states that if a family of laws of processes is tight and converges in the sense the finite dimensional distributions, then this family converges. 
Tightness is proved in Proposition \ref{prop:tight}. Convergence of finite dimensional distributions is proved in Theorem \ref{th:finite}. Therefore $(\pi(X^\ka))$ converges in law in $C(\bbR^+, \wt M)$. 
\end{proof}

\section{Random perturbations of Hamiltonian systems in $\mathbb{R}^2$}\label{sec:hamilt}
In this section, we recover results from \cite{freidlin.wentzell12} and \cite{BvR}. We noticed that the proof of the regularity of the Dirichlet form $\wt \cE$ in \cite{BvR} is incorrect (in particular, there is a mistake in the proof of density of continuous functions with compact support in the domain of $\wt\cE$). In order to correct this mistake, we had to modify the assumptions on the Hamiltonian $H$.

In this section, we set $M=\bbR^2$ and $m$ the Lebesgue measure on $\bbR^2$. 
Let $S$ be a second order differential on $\bbR^2$, with $S1=0$ satisfying conditions (i), (ii) and (iii) as in section \ref{sec:DFM}. Denote by $\Gamma$ its associated carr\'e du champ operator. Let $V_0$ be a $C^1$ vector field on $\bbR^2$.
We assume that \eqref{eq:hyp(3)}, \eqref{eq:sector1}, \eqref{eq:hyp(1)} and \eqref{Guelliptic} hold. Recall that a sufficient condition for \eqref{eq:sector1} to hold is that $V_0$ is a bounded vector field.

Let $(\cE,\cH)$ be the regular Dirichlet form on $L^2(m)$ associated to $\Gamma$ and $V_0$, defined in Proposition \ref{prop:DF0}.
Then $\cH=H^1(\mathbb{R}^2)$. This Dirichlet form is associated to the diffusion on $\bbR^2$ with generator $A=S+V_0$.

Let $H:\bbR^2\to\bbR^+$ be a $C^2$ function and set $V=(-\del_{x_2}H,\del_{x_1}H)$. We suppose that $\|\nabla H\|_\infty <\infty$ so that condition (1) in Proposition \ref{prop:WScond} is satisfied. Since $\hbox{div}V=0$, Proposition \ref{prop:DF0} permits to define $(\mathcal{E}^\kappa,H^1(\mathbb{R}^2))$ the Dirichlet form on $L^2(m)$ associated to $\Gamma$ and $V_0+\kappa V$ for all $\kappa\in\mathbb{R}$.

Note that $VH=V_1\del_{x_1}H+V_2\del_{x_2}H=0$.
Set $\cS$ the set of stationary points of $V$ (or equivalently the set of critical points of $H$).
We will assume that 
\begin{itemize}
\item Assumption \ref{hyp:F} holds and that the stationary points of $V$ are non-de\-ge\-ne\-ra\-te. 
\end{itemize}
This assumption implies that $\cS$ is locally finite.
The vector field $V$ is complete and generates a flow we denote by $\phi=(\phi_t)_{t\in \bbR}$. Assumption \ref{hyp:F} implies that the orbits of $\phi$ are compacts, and since $\hbox{div}V=0$, the measure $m$ is invariant for this flow. 

Our results improve the one obtained by Freidlin, Wentzell and their coauthors (exposed in \cite{freidlin.wentzell12}, Chapter $8$) in the sense that, with our setting, the Hamiltonian $H$ needs only to be $C^2$, $\mathcal{S}$ can be infinite and it is possible to add a drift $V_0$. 
The description of the limiting process is also easier with the framework of Dirichlet forms.

\smallskip
In order to apply the results of Section \ref{sec:avg}, we need to construct a set $\wt M$, mappings $\pi:M\to \wt M$ and $p:\wt M\to E$ satisfying Assumption \ref{hyp:pi}.
We then show that Assumption \ref{hyp:tight} holds. 
With this in hand, the averaging principle of section \ref{sec:avg} can be used. In Section \ref{sec:desc}, we describe the Dirichlet form $\wt \cE$.
 
\subsection{Description of $E$, the set of ergodic measures.}
The Poincaré-Bendixson Theorem (Theorem 14.1.1 p.452 in \cite{hasselblatt2}) states that the recurrent orbits of $\phi$ are periodic, and thus that the support of an ergodic measure is a periodic orbit.

Let $h\ge 0$. Then $H^{-1}(h)=\cup_{i\in I(h)} \g_i(h)$, 
where $I(h)$ is a finite set, $\g_i(h)$ is connected and compact for all $i$, and $\g_i(h)\cap \g_j(h)=\emptyset$ if $i\ne j$. 
Then $\g_i(h)$ is a periodic orbit provided $\g_i(h)\subset \cS$ or $\g_i(h)\cap \cS=\emptyset$. On the converse, if $\g$ is a periodic orbit, then there is $i$ and $h$ such that $\g=\g_i(h)$.

Let $\g$ be a periodic orbit of $\phi$, with period denoted $T_\g$. Note that there is $i$ and $h$ such that $\g=\g_i(h)$. When $T_\g=0$, then $\g=\{x\}\subset \cS$ and we set $\mu_\g:=\delta_x$. When $T_\g>0$, for all $x\in\g$, the measure $\frac{1}{T_\g}\int_0^{T_\g} \delta_{\phi_t(x)} \,dt$ does not depend on $x$, and we denote this measure  $\mu_\g$. Then $E$ consists of all $\mu_\g$, with $\g$ a periodic orbit.

\subsection{Construction of $\wt M$, $\pi$ and $p$}
{\it The set $\wt M$.}
For $x\in \cS$, denote by $\cF_x$ the connected component of $\{y\in\mathbb{R}^2: H(y)=H(x)\}$ containing $x$. 
Let $\cF=\cup_{x\in \cS}\cF_x$. Then $\Om:=\cF^c$ is the union of all periodic orbits $\g$, with $T_\g>0$, and there is an at most countable partition $(\Om_i)_{i\in I}$ of $\Om$ into open connected sets.
For $i\in I$, set $m_i:=m_{|\Om_i}$, $l_i:=\inf\{H(x):x\in \Om_i\}$ and $r_i:=\sup\{H(x):x\in \Om_i\}$. 
Then, $H(\Om_i)=]l_i,r_i[$, with $r_i=+\infty$ for at most one $i\in I$. 
Note that $H^{-1}(l_i)\cap\overline{\Om}_i\subset \cF_x$ for some $x\in\cS$, and if $r_i<+\infty$, $H^{-1}(r_i)\cap\overline{\Om}_i\subset \cF_{x'}$ for some $x'\in\cS$. 
For $h\in [l_i,r_i]$, set $\g(h,i):=H^{-1}(h)\cap \overline{\Om}_i$ and $T_i(h)=T_{\gamma(h,i)}$, then $\g(h,i)$ is a $C^1$ periodic orbit when $h\in ]l_i,r_i[$ and $T_i(h)$ is the associated period. For $\g$ a $C^1$ curve and $f$ a bounded measurable function on $\g$, denote by $\oint_\g f d\ell$ the curve integral of $f$ on $\g$.

\begin{proposition}\label{prop:coaire}
Fix $i\in I$.
\begin{enumerate}
\item For all $f\in L^1(m_i)$, $\int_{\Om_i} f d m_i = \int_{l_i}^{r_i} \oint_{\g(h,i)} f \frac{d\ell}{|\nabla H|} dh$. 
\item For all $h\in ]l_i,r_i[$, setting $\g=\g(h,i)$, we have $T_{\g}=\oint_{\g}\frac{d\ell}{\abs{\grad H}}$ and
\begin{align}\label{eq:average}
\mu_{\g}(d\ell)=\frac1{T_{\g}}\frac{d\ell}{\abs{\grad H}}.
\end{align}
\item For $l_i<h<k<r_i$ and $f$ a $C^1$-function on $\Om_i$, setting $\Om_i^{h,k}=\Om_i\cap H^{-1}(]h,k[)$, we have
\begin{align}\label{eq:stokes}
\oint_{\g(k,i)}f|\nabla H| d\ell -\oint_{\g(h,i)} f|\nabla H| d\ell
= \int_{\Om_i^{h,k}} \Div (f\nabla H) dm.
\end{align}
\item For all continuous function $f$ on $\Om_i$, $h\mapsto \int_{\g(h,i)} f |\nabla H| d\ell$ is continuous on $]l_i,r_i[$. If in addition $f$ is bounded continuous on $\overline{\Om}_i\setminus\cS$, then this function has finite limits at $l_i$ and at $r_i$ (when $r_i<\infty$).
\item For all $C^1$-function $f$ on $\Om_i$, $h\mapsto \oint_{\g(h,i)} f|\nabla H| d\ell$ is $C^1$ on $]l_i,r_i[$, with derivative at $h$ given by:
\begin{align}
\label{eq:derive}
\frac{d}{dh} \oint_{\g(h,i)} f|\nabla H| d\ell =\oint_{\g(h,i)} \Div (f\nabla H) \frac{d\ell}{|\nabla H|}.
\end{align}
\end{enumerate} 

\end{proposition}
\begin{proof}
(1) follows from the coarea formula. (2) follows from the change of variable formula. (3) is a consequence of Stokes formula applied to the vector field $U=f\nabla H$. When $f$ is $C^1$, (4) follows from (3), and this can be extended to continuous functions with  a density argument. To prove (5), it suffices to use \eqref{eq:stokes}, the coarea formula given in (1) applied to the right hand side of \eqref{eq:stokes} and (4) applied to the function $\frac{\Div (f\nabla H)}{|\nabla H|^2}$.
\end{proof}
A first consequence of Proposition \ref{prop:coaire} is that, for $h\in ]l_i,r_i[$
$$\frac{d}{dh} T_i(h)=\oint_{\g(h,i)} \Div \left(\frac{\nabla H}{|\nabla H|^2}\right) \frac{d\ell}{|\nabla H|}.$$

\medskip
Set $\hat{M}:=\cup_{i\in I} [l_i,r_i]\times \{i\}$ (using the abuse of notation $[l,+\infty]=[l,+\infty[$) and define the equivalent relation $(h,i)\sim (k,j)$ if $(h,i)=(k,j)$ or  if $h=k\in\{l_i,r_i\}\cap\{l_j,r_j\}$ and there is $x\in\cS$ such that $\overline{\Om}_i\cap \cF_x$ and $\overline{\Om}_j\cap \cF_x$ are both non empty.
Set $\wt M=\hat M/\sim$. By abuse of notation, an $(h,i)\in \hat M$ will be identified to its equivalent class in $\wt M$.
For $i\in I$, set $E_i=]l_i,r_i[\times\{i\}$. Then $\wt M=\cup_i E_i \cup \cV$, with $\cV$ being the sets of the equivalent classes of $(l_i,i)$ and of $(r_i,i)$, for all $i\in I$. Then $\wt M$ is a connected metric graph (see e.g. Hajri and Raimond \cite{hr}). The set of edges is $\{E_i:i\in I\}$, the set of vertices is $\cV$ and the distance $d$ on $\wt M$ is such that if $(h,k)\in [l_i,r_i]^2$, $d((h,i),(k,i))=\abs{h-k}$. 

{\it Construction of $\pi$ and $p$.} Let us define $\pi: M\to \wt M$ by $\pi(x)=(H(x), i(x))$ where $i(x)=\min\{i\in I, x\in \widebar{\Om_i}\}$. 
To construct $p:\wt M\to E$, fix $\mu_0\in E$. If $\wt x\in \cV$, we set $p(\wt x)=\mu_0$.
If $\wt x=(h,i)\in E_i$ for some $i\in I$, set $p(\wt x)$ the ergodic measure associated to the periodic orbit $\g(h,i)$.
 
\begin{proposition}
$\wt M$, $\pi$ and $p$ satisfy Assumption \ref{hyp:pi}. 
\end{proposition}

\begin{proof}
By construction, $\wt M$ is a locally compact separable metric space. 
Note that $x\mapsto i(x)$ is constant on $\Om_i$.
The fact that $\pi$ is continuous is a simple consequence of the fact that $H$ is continuous and that for $(x,y)\in\overline{\Om}_i\times \overline{\Om}_i$, $d(\pi(y),\pi(x))=\abs{H(y)-H(x)}$.
Note that Proposition \ref{prop:coaire}-(4) implies that  $\wt x\mapsto p(\wt x)$ is continuous in $E_i$ for all $i\in I$. The set of vertices of $\wt M$ being at most countable, $p$ is measurable.

It remains to check that $p\circ \pi(x)=P(x)$ $m(dx)$-a.e. Note that, for $x\in \Om_i$, $p\circ\pi(x)$ is ergodic and that $x$ belongs to the support of $p\circ \pi(x)$, which is $H^{-1}\{ H(x)\} \cap \Om_i$. This implies that $p\circ \pi(x)=P(x)$ for all $x\in \Om=\cup_{i\in I}\Om_i$. Since $m(\Om^c)=0$, we get that $p\circ \pi(x)=P(x)$ $m(dx)$-a.e. \end{proof}

Note that $\wt M$ is a {\it tree} i.e. there is no self-intersecting closed path (or cycle) in $\wt M$. Each edge $E_i$, $i\in I$, is given an orientation by the function $H$. 
For $v\in \cV$, a vertex, set $I^+_v=\{i\in I, (l_i,i)=v\}$, $I^-_v=\{i\in I, (r_i,i)= v\}$ and $I_v=I^+_v\cup I^-_v$. Choose $i\in I_v$ and set $h_v=l_i$ if $v=(l_i,i)$ or $h_v=r_i$ if $v=(r_i,i)$. Remark that $h_v$ does not depend on the particular choice $i\in I_v$. For $v\in \cV$, we define $\g (v):=\cup_{(h,i)\sim v}\g(h,i)$, the connected level set associated to the vertex $v$.
Set $d(v):=|I_v|$, the cardinal of the set $I_v$, which is the degree of $v$ in the graph $\widetilde M$.

\begin{proposition}
For $i\in I$, set $\wt m_i(dh):=H_*m_i(dh)$. Then, we have that $\wt m_i(dh)=T_i(h)\1_{[l_i,r_i]}(h) dh$ is a finite measure on $[l_i,r_i]$ if $r_i<\infty$ and is $\sigma$-finite on $[l_i,\infty[$ if $r_i=\infty$. 
\end{proposition}
\begin{proof}
Follows from Proposition \ref{prop:coaire}.
\end{proof}

\subsection{The space $\wt\cH$ and Assumption \ref{hyp:tight}} \label{sec:wtcH}
For a function $f:\wt M\to \mathbb{R}$ and $i\in I$, define $f_i:]l_i,r_i[\to\mathbb{R}$ by $f_i(h)=f(h,i)$.
Set $\wt m:=\pi_*m$.
If $f\in L^1(\wt m)$, then $\int f d\wt m=\sum_{i\in I}\int f_i d\wt m_i$.
For $i\in I$ and $h\in ]l_i,r_i[$, define
\begin{equation}\label{eq:aideh}
a_i(h)=\int_{\g(h,i)} |\nabla H| d\ell.
\end{equation}
For $(h,i)\in E_i$, set $a(h,i)=a_i(h)$ and $T(h,i)=T_i(h)$.

\begin{lemma}\label{lem:wtcH}
Let $f\in\wt \cH$. Then, for each $i$, the map $f_i$ is weakly differentiable on $]l_i,r_i[$ and we have that
\begin{align} \label{eq:wtcH}
\norm{f}^2_{\wt \cH} = & \sum_i \int_{l_i}^{r_i} \big(f'_i(h)\big)^2 a_i(h) dh + \sum_i \int_{l_i}^{r_i} \big(f_i(h)\big)^2 T_i(h) dh 
\end{align}
\end{lemma}
\begin{proof}
The fact that $f_i$ is weakly differentiable  follows from the fact that on $\Om_i$, $\nabla H\ne 0$ and $f\circ\pi=f_i\circ H$. Now, since on $\Omega_i$, $\nabla (f\circ \pi) = (f'_i\circ H) \nabla H$, $\norm{f}_{\wt \cH}=\norm{f\circ\pi}_{H^1(\mathbb{R}^2)}$ easily yields \eqref{eq:wtcH}. 
\end{proof}

\begin{proposition}\label{prop:Htilde5}
Set $\mathcal{V}_1:=\{v\in \mathcal{V}:\; d(v)=1\}$.
A function $f:\wt M\to\mathbb{R}$ belongs to $\wt \cH$ if and only if $f$ is continuous on $\wt M\setminus \mathcal{V}_1$, the mappings $(f_i,\;i\in I)$ are weakly differentiable, and $\|f\|_{\wt \cH}<\infty$.
\end{proposition}
\begin{proof}
This will be explained in the appendix.
\end{proof}

For a measurable function $f:\wt M\to\mathbb{R}$ such that for all $i\in I$, $f_i$ is differentiable (or weakly differentiable), define a measurable function $f':\wt M\to\mathbb{R}$ such that $f'(h,i)=f'_i(h)$, $\wt m_i(dh)$-a.e.

\begin{proposition}\label{prop:Ctilde1}
Assumption \ref{hyp:tight} is satisfied.
\end{proposition}
\begin{proof}
Let $\wt C$ be the space of all $f\in C_c(\wt M)$ such that $f_i\in C^2(]l_i,r_i[)$ for all $i\in I$ and such that $f'$ and $f''$ are continuously extendable by continuity at all $v\in \cV$. Then $\wt C\subset \{f: f\circ\pi\in C^2_c(M)\}\subset \wt \cH$.

The fact that $\wt C$ is dense in $C_c(\wt M)$ follows from the Stone Weierstrass Theorem: $\wt C$ is an algebra and it is a simple exercise to show that $\wt C$ separates the points in $\wt M$. This proves that \ref{hyp:tight}-(iii) is satisfied.

Since $\wt C\subset C_c(\wt M)\cap \wt \cH$, to prove that \ref{hyp:tight}-(i) is satisfied, it suffices to check that $\wt C$ is dense in $L^2(\wt m)$. 
Note that $\wt m(V)=m(\cF)=0$.
Let $f\in L^2(\wt m)$.
Fix $\de>0$. 
Since $\norm{f}^2_{L^2(\wt m)}=\sum_{i\in I}\norm{f\1_{E_i}}^2_{L^2(\wt m)}<\infty$, there is a finite set $J$ such that $\norm{f-\sum_{i\in J} f\1_{E_i}}_{L^2(\wt m)}<\de/2$.
Denote by $n$ the cardinality of $J$.
For each $i\in J$, there is a function $g_i\in \wt C$ with support in $E_i$ such that $\norm{g_i-f_i}_{L^2(\wt m)}<\de/(2n)$.
Then $g:=\sum_{i\in J} g_i\in\wt C$ and $\norm{f-g}_{L^2(\wt m)}< \de/2 + n\de/(2n)=\de$.

The proof of \ref{hyp:tight}-(ii) is more complicated and is given in the Appendix.
\end{proof}

\subsection{Application of Theorem \ref{th:main}}
Let $(\wt \cE,\wt \cH)$ be the Dirichlet form on $L^2(\wt m)$ obtained by contracting $(\cE,\cH)$ on $\wt M$ using Proposition \ref{prop:DF2}. 
By Proposition \ref{prop:regular},  $(\wt \cE,\wt \cH)$ is regular possesses the local property. 
Moreover, $(\wt \cE,\wt \cH)$ is a contraction of  $(\cE^\ka,\cH)$ for all $\ka$, where $\cE^\ka=\cE+\ka \cE^V$.
Then we can apply Theorem \ref{th:main}, recovering results of \cite{freidlin.wentzell12} and \cite{BvR}

\subsection{The Dirichlet form $\wt\cE$ and its generator} \label{sec:desc} 
The aim of this section is to give a more comprehensive description of the limiting diffusion process $\wt X$ on $\wt M$ associated to the regular Dirichlet form $(\wt\cE,\wt\cH)$ by computing the infinitesimal generator of the associated semigroup. 

\begin{proposition}
For $i\in I$ and $h\in ]l_i,r_i[$, set 
\begin{align}
\sigma^2_i(h)&=p(h,i)\big( 2\Gamma(H,H)\big),&c_i(h)&=p(h,i)\big(V_0H\big).
\end{align}
Then, for $(f,g)\in \wt{\cH}^2$,
\begin{equation}\label{eq:etildeuv}
\wt \cE(f,g)=\frac{1}{2}\sum_{i\in I} \int_{l_i}^{r_i} \sigma^2_i f'_ig'_i T_i \,dh - \sum_{i\in I} \int_{l_i}^{r_i} c_i f'_ig_i T_i \,dh
\end{equation}
\end{proposition}
\begin{proof}
Let $(f,g)\in \wt\cH^2$. Then 
\begin{align*}
\wt \cE(f,g)
&= \cE(f\circ\pi,g\circ \pi)\\
&=  \int \Gamma(f\circ\pi,g\circ\pi) dm - \int V_0(f\circ\pi)g\circ\pi dm\\
&= \sum_{i\in I}\left( \int_{\Omega_i} \Gamma(f_i\circ H,g_i\circ H) \;dm_i - \int_{\Omega_i} V_0(f_i\circ H)g_i\circ H \; dm_i\right)\\
&= \sum_{i\in I}\left( \int_{\Omega_i} \Gamma(H,H) (f'_ig'_i)\circ H \;dm_i - \int_{\Omega_i} (V_0H)(f'_ig_i)\circ H \;dm_i\right)
\end{align*}
and then one can conclude using Proposition \ref{prop:coaire}-(1), and the definition of $p(h,i)$.
\end{proof}

Using the notation $\sigma(h,i)=\sigma_i(h)$ and $c(h,i)=c_i(h)$, Equation \eqref{eq:etildeuv} can be written in the shorter form:
\begin{equation}\label{eq:etildeuvb}
\wt \cE(f,g)=\frac{1}{2} \int \sigma^2 f'g' \,d\wt m - \int c f'g \,d\wt m.
\end{equation}

For $i\in I$, set $b_i:=\frac{1}{2T_i}(\sigma^2_iT_i)'+c_i$.
Since $\frac{\Gamma(H,H)}{|\nabla H|^2}$ is bounded continuous on $\mathbb{R}^2\setminus\cS$, Proposition \ref{prop:coaire}-(4) implies that the limits $\alpha_i^+:=\lim_{h\to l_i+} \frac{1}{2}(\sigma^2_iT_i)(h)$ and  $\alpha_i^-:=\lim_{h\to r_i-} \frac{1}{2}(\sigma^2_iT_i)(h)$ (when $r_i<\infty$) exist and are finite. Moreover, if $i\in I_v^{\pm}$ with $v\in \mathcal{V}$ such that $d(v)=1$, then $\alpha_i^{\pm}=0$.

\begin{proposition}\label{prop:atilde} A function $f$ belongs to $\mathcal{D}(\wt A)$ if and only if 
 $f\in \wt \cH$ and for all $i\in I$, 
\begin{enumerate}[(i)]
\item $f'_i$ is continuous and is weakly differentiable;
\item If $i\in I_v^+$ (resp. $I_v^-$) with $d(v)\ge 2$, then 
$f'_i$ has a finite limits at $l_i$ (resp. $r_i$ if $r_i<\infty$);
\item  If $i\in I_v^+$ (resp. $I_v^-$) with $d(v)=1$, then 
$\lim_{h\to l_i+}\sigma_i^2T_if'_i(h)=0 $ (resp. $\lim_{h\to r_i-}\sigma_i^2T_if'_i(h)=0$);
\item $A_if_i:= \frac{1}{2}\sigma^2_if_i'' + b_if'_i$ belongs to $L^2(\wt m_i)$;
\end{enumerate}
and for all $v\in \cV$,
\begin{equation}\label{eq:boundarycond3}
\sum_{i\in I_v^+} \alpha_i^+ f'_i(l_i+)=\sum_{i\in I_v^-} \alpha_i^- f'_i(r_i-).
\end{equation}
Moreover, if $f\in \mathcal{D}(\wt A)$ and if $\wt x=(h,i)\in E_i$, then $\wt A f(h,i)=A_if_i(h)$.
\end{proposition}
\begin{proof}
Let $f\in\wt \cH$ be such that (i), (ii), (iii) and (iv) are satisfied.
Then $\wt A f$ defined such that on $E_i$, $\wt A f(h,i)=A_if_i(h)$,  belongs to $L^2(\wt m)$. Let now $g\in \wt\cH\cap C_c(\wt M)$, then
\begin{align*}
\wt\cE(f,g)
&= - \sum_{i\in I}\int_{l_i}^{r_i} (A_if_i) g_i T_i\,dh + \frac{1}{2}\sum_{i\in I}[\sigma^2_iT_i f'_i g_i]_{l_i}^{r_i}\\
&= - \int (\wt A f) g \, d\wt m + \sum_{v\in V:d(v)\ge 2} g(v)\left(\sum_{i\in I_v^-} \alpha_i^- f'_i(r_i-)-\sum_{i\in I_v^+} \alpha_i^+ f'_i(l_i+) \right).
\end{align*}
If \eqref{eq:boundarycond3} is satisfied, then $\wt\cE(f,g)=- \int (\wt A f) g \, d\wt m$.
Therefore $\wt\cE(f,\cdot)$ is a linear form on $\wt \cH$, continuous with respect to $\|\cdot\|_{L^2(\wt m)}$, which implies that $f\in  \mathcal{D}(\wt A)$.

On the converse, let $f\in  \mathcal{D}(\wt A)$. Then $f\in \wt\cH$ and
$\wt \cE(f,g)=\langle \wt A f,g\rangle_{L^2(\wt m)}$ for all $g\in \wt\cH$. This implies that for all $i\in I$, $f_i\in H^2_{loc}(]l_i,r_i[)$ and so $f_i$ satisfies (i) (choose $g\in C^{\infty}_c(]l_i,r_i[)$). This also implies that for all $i\in I$, $\frac1{\sqrt{T_i}}
(\sigma^2_iT_if'_i)'\in L^2(]l_i,r_i[)$, thus $\sigma^2_iT_if'_i$ is uniformly continuous on $]l_i,r_i[$. Using the asymptotics of $T_i$ as $h\to l_i$ and as $h\to r_i$ given in Lemma \ref{lem:asymptoticsat}, it can easily be proved that $f$ satisfies (ii), (iii) and (iv).
To conclude, since (i), (ii), (iii) and (iv) are satisfied, the fact that $\wt\cE(f,\cdot)$ is a linear form on $\wt \cH$, continuous with respect to $\|\cdot\|_{L^2(\wt m)}$, implies that \eqref{eq:boundarycond3} is satisfied for all $v\in\mathcal{V}$.
\end{proof}

Proposition \ref{prop:atilde} shows that the Markov process on $\wt M$ associated to $\wt\cE$ evolves as a diffusion with generator $A_i$ on $E_i$ with reflecting boundary conditions at a vertex $v$ given by \eqref{eq:boundarycond3} with \textit{transmission parameters} $(\al_i^{\pm})_{i\in I_v}$ at a vertex $v$.

\begin{proposition} Let $v\in V$ and $i\in I_v$. Then, for the diffusion on $E_i$ of generator $A_i$, $v$ is an entrance boundary point if $d(v)=1$ and is a regular boundary point if $d(v)\ge 2$.
\end{proposition}

\begin{proof}
Let us fix $h_0\in]l_i,r_i[$. The $1$-dimensional diffusion with generator $A_i$ on $]l_i,r_i[$ has scale function $u$ and speed measure $\rho(h)dh$ given by
\begin{align}
u(h)&:=\int_{h_0}^h\exp\left(-\int_{h_0}^k \frac{2b_i(s)}{\sigma_i^2(s)}\,ds\right)\,dk=\int_{h_0}^h \frac1{\sigma_i^2T_i(k)}\exp\left(-\int_{h_0}^k \frac{2c_i(s)}{\sigma_i^2(s)}\,ds\right)\,dk,\\
\rho(h)&:=\frac1{\sigma_i^2(h)u'(h)}=T_i(h)\exp\left(\int_{h_0}^h \frac{2c_i(s)}{\sigma_i^2(s)}\,ds\right).
\end{align}
Let us assume that $v=(l_i,i)$, we need to compute the asymptotics of the two functions $u$ and $\rho$ as $h\to l_i+$.
Since Equations \eqref{eq:hyp(1)} and \eqref{Guelliptic} hold, Lemma \ref{lem:asymptoticsat} ensures that $\sigma^2 T$ satisfies the same asymptotics as $a$.
Thus, as $h\to l_i^+$
\begin{itemize}
\item if $d(v)=1$, $\sigma^2_i(h)T_i(h)\sim c_0 (h-l_i)$ and $T_i(h)\sim c_1$,
\item if $d(v)\ge 2$, $\sigma^2_i(h)T_i(h)\sim c_0$ and $T_i(h)\sim c_1|\log(h-l_i)|$,
\end{itemize}
where $c_0$ and $c_1$ are two finite positive constants.
Note that by Stokes' Theorem,
\be\label{eq:stokes2}
c_i(h)=\frac1{T_i(h)}\oint_{\gamma(h,i)}V_0\cdot \frac{\grad H}{|\grad H|}d\ell=\frac1{T_i(h)}\int_{\Omega^h_i}\text{div} (V_0)(x)\,dx
\ee
where $\Omega^h_i$ is the compact set enclosed by the orbit $\gamma(h,i)$. Since $V_0$ is a $C^1$ vector field, by the same arguments as Lemma \ref{lem:asymptoticsat}, $c_i(h)T_i(h)= O(1)$ as $h\to l_i$ and furthermore, if $d(v)=1$, $c_i(h)T_i(h)=O(h-l_i)$. 

This gives the asymptotics of $u$ and $\rho$ as $h\to l_i^+$: for some positive constants $c'_0,c'_1$,
\begin{itemize}
\item if $d(v)=1$, $u(h)\sim c'_0\ln(h-l_i) \quad\text{and}\quad \rho(h)\sim c'_1$,
\item if $d(v)\ge 2$, $u(h)\sim c'_0$ and $\rho(h)\sim c'_1|\ln(h-l_i)|.$
\end{itemize}

Therefore (see e.g. \cite{breiman} pp. 369),
\begin{itemize}
\item if $d(v)=1$, $u(l_i+)=-\infty$, $\int_{l_i}|u(h)|\rho(h)dh<+\infty$, and $v$ is an entrance boundary point; 
\item if $d(v)\ge2$, $u(l_i+)$ is finite, $\int_{l_i}(u(h)-u(l_i+))\rho(h)dh <+\infty$, and $\int_{l_i}\rho(h)dh<+\infty$, and $v$ is a regular boundary point.
\end{itemize}
\end{proof}

\section{An example on $\bbR^3$}\label{sec:R3}

In this section, we set $M=\bbR^3$ and $m(dx)=e^{-\mathcal{W}(x)} \,dx$, where $\mathcal{W}\in C^2(\mathbb{R}^3)$. 
For $f\in C^2(\mathbb{R}^3)$ and $x\in\mathbb{R}^3$, set
\begin{equation}
Sf(x)=\hbox{div}_m(\nabla f)(x)=\D f(x) - \langle \nabla \mathcal{W}(x),\nabla f(x)\rangle.
\end{equation}
Then $S$ is a symmetric operator on $L^2(m)$ and its associated carr\'e du champ operator is given, for $f\in C^1(\mathbb{R}^3)$, by $\Gamma(f,f):=|\grad f|^2$.
Applying Proposition \ref{prop:DF0}, there is a unique (symmetric) Dirichlet form $(\mathcal{E},\mathcal{H})$ on $L^2(m)$ associated $\Gamma$ and $V_0=0$. Moreover, $\mathcal{H}=H^1(m)$.
This Dirichlet form is associated to the diffusion on $\bbR^3$ with generator $S$.

Let $V$ be the vector field defined for $x=(x_1,x_2,x_3)\in\bbR^3$
\be\label{eq:MP}
V(x)=\left(
\begin{aligned}
x_2x_3&\\x_1x_3&\\-2x_1&x_2
\end{aligned}\right).
\ee
The flow $(\phi_t)_{t\in \mathbb{R}}$ generated by $V$ leaves invariant the quantities $\|x\|^2$ and $x_1^2-x_2^2$.
Suppose that $\mathcal{W}$ is such that
\begin{enumerate}[(i)]
\item $\mathcal{W}$ is strictly convex, i.e. $\hbox{Hess} \mathcal{W} \ge \rho>0$;
\item $\mathcal{W}(x)=W(\|x\|)$ for some function $W:\mathbb{R}^+\to\mathbb{R}^+$;
\item $\int e^{\|V(x)\|^2} e^{-\mathcal{W}(x)}dx<\infty$.
\end{enumerate} 
For example, one can take $\mathcal{W}(x)=(1+\|x\|^2)^\alpha$, with $\alpha>2$.

Condition (ii) ensures that $\text{div}_m V=0$, and thus that the measure $m$ is invariant for the flow $(\phi_t)_{t\in \mathbb{R}}$ generated by $V$. Note that Assumption \ref{hyp:F} holds for the function $H$ defined by $H(x):=\|x\|^2=x_1^2+x_2^2+x_3^2$.

Conditions (i) and (iii) ensures that $\Gamma$ and $V$ satisfy condition (3) in Proposition \ref{prop:WScond}. 
Proposition \ref{prop:DF0} permits to define $(\mathcal{E}^\kappa,\mathcal{H})$ the Dirichlet form on $L^2(m)$ associated to $\Gamma$ and $V_0+\kappa V$ for all $\kappa\in\mathbb{R}$.
\medskip

This example is a variant of a model studied by Mattingly and Pardoux in \cite{MP}. Their primary interest is the limit as $\kappa\to \infty$ of the invariant measure of the diffusion $X^\kappa$ in $\mathbb{R}^3$ with generator $\frac12 (\del_{11}^2+\del_{22}^2) + \kappa V$.
In order to obtain the limiting measure, they prove the convergence in law towards a diffusion in $\mathbb{R}_+^2$ of $(u(X^\kappa_{\cdot}),v(X^\kappa_{\cdot}))$ where $u(x)=2x_1^2+x^2_3$ and $v(x)=2x_2^2+x^3_2$ . 

Significant differences exist between our models. In order to use the Dirichlet form approach, we need a confinement potential $\cW$ that the authors in \cite{MP} do not need, and we require that the diffusion be uniformly elliptic (Equation \eqref{Guelliptic}). In \cite{MP}, $S=\del_{11}^2+\del_{22}^2$ and
 the authors could not add noise on the $x_3$-coordinate.

To avoid the use of a confinement potential $\mathcal{W}$, one could also replace the vector field $V$ by $V_g(x):=g(\|x\|)V(x)$ with a well chosen $C^1$ function $g$. Then, $(iii)$ could be replaced by the assumption that $V_g$ is bounded ($(1)$ of Proposition \ref{prop:WScond}) and $m$ could be taken as the Lebesgue measure.

We slightly adapt their notations and computations to our framework.

\subsection{Construction of $\wt M$ and $\pi$.}
The set of stationary points $\cS$ is $\{x_1=x_2=0\}\cup\{x_2=x_3=0\}\cup \{x_3=x_1=0\}$. 
Set 
\begin{align*}
M_1=\{x\in\mathbb{R}^3\backslash \cS:\; x_1>|x_2|\}
&;&M_3=\{x\in\mathbb{R}^3\backslash \cS:\; x_1<-|x_2|\};\\
M_2=\{x\in\mathbb{R}^3\backslash \cS:\; x_2>|x_1|\}
&;&M_4=\{x\in\mathbb{R}^3\backslash \cS:\; x_2<-|x_1|\}.
\end{align*}
Then, we get $\mathbb{R}^3=\overline{\cup_{i=1}^4M_i}$
 and $M_1$, $M_2$, $M_3$ and $M_4$ are invariant for $\phi$. 

Denoting $r_1(x_1,x_2,x_3)=(x_2,x_1,x_3)$ and $r_2(x_1,x_2,x_3)=(-x_1,-x_2,x_3)$, $r_1$ and $r_2$ commute with $\phi$ (in the sense that $r_i\circ\phi_t=\phi_t\circ r_i$ for $i\in\{1,2\}$ and $t\in \mathbb{R}$). We have that $M_2=r_1(M_1)$, $M_3=r_2(M_1)$ and $M_4=r_1\circ r_2(M_1)$.

\subsubsection*{The set $\wt M$}
Set now 
\begin{align*}
C_1=\{y\in\mathbb{R}^2: y_1\ge y_2\ge 0\}
&;& C_3=\{y\in\mathbb{R}^2: y_1\le y_2\le 0\};\\
C_2=\{y\in\mathbb{R}^2: y_2\ge y_1\ge 0\}
&;&C_4=\{y\in\mathbb{R}^2: y_2\le y_1\le 0\}.
\end{align*}
For $i\in\{1,2\}$, define $s_i:\mathbb{R}^2\to\mathbb{R}^2$ by $s_1(y)=(y_2,y_1)$ and $s_2(y)=-y$. Then  $s_1(C_1)=C_2$, $s_1(C_3)=C_4$, $s_2(C_1)=C_3$ and $s_2(C_2)=C_4$. 
Set $\wt M=\cup_{i=1}^4 C_i$. 
For  $y\in D:=[0,\infty[$, we identify $(y,y)\in C_1\cap C_2$ and $(-y,-y)\in C_3\cap C_4$ with $y$. Set also $D^*=]0,\infty[$,  $C_i^*=C_i\setminus\{0,0\}$ for $i\in I$ and $C^*=C\setminus\{(0,0)\}$.

We equip $\wt M$ with the distance $d$ induced by the euclidean distance on each $C_i$ (setting $|y|:=\sqrt{y_1^2+y_2^2}$):
$$d(y,y'):=\begin{cases}|y-y'|& \text{ if } y,y'\in C_1\cup C_2 \text{ or if } y,y'\in C_3\cup C_4\\
|y-s_2(y')|&\text{ if } (y,y')\in C_2\times C_3 \text{ or if } (y,y')\in C_1\times C_4\\
|y-s_1\circ s_2(y')| &\text{ if } (y,y')\in C_1\times C_3 \text{ or if } (y,y')\in C_2\times C_4.
\end{cases}
$$
Recall that when $y'\in C_3$ (resp. $C_4$), then $s_2(y')\in C_1$ (resp. $C_2$) and $s_1\circ s_2(y')\in C_2$ (resp. $C_1$).
Then, $\wt M=\cup_i C_i$ is a $4$-pages book, with $D$ being the binding. 
Set $I:=\{1,2,3,4\}$.
\subsubsection*{The function $\pi$}
Let us define $\pi:\mathbb{R}^3\to \wt M$ by
\begin{align*}
\pi(x)=\left\{ \begin{array}{ll}
\left(\sqrt{x_1^2+\frac{x_3^2}{2}},\sqrt{x_2^2+\frac{x_3^2}{2}}\right)
&\quad\hbox{ if } x\in \overline{M_1\cup M_2};\\
\left(-\sqrt{x_1^2+\frac{x_3^2}{2}},-\sqrt{x_2^2+\frac{x_3^2}{2}}\right)
&\quad\hbox{ if } x\in \overline{M_3\cup M_4}.
\end{array}\right.
\end{align*} 
We then have that
\begin{proposition} The mapping $\pi:\mathbb{R}^3\to\wt M$ defined above is continuous. Moreover, for each $i\in I$, $\pi(M_i)=\ring{C}_i$.
\end{proposition}

\subsection{Description of $E$, the set of ergodic measures.}\label{sec:6.descriptionE}
We first describe the orbits of the flow generated by $V$. 
For each $i\in I$, let $\Phi_i: \ring{C}_i \times [0,2\pi[\to \cup_i M_i$ be the $C^1$-diffeomorphism defined by
\begin{align*}
\begin{array}{ll}
\Phi_1(y,\theta)=\left(
\sqrt{y_1^2-y_2^2\sin^2\theta},
y_2\cos\theta,\sqrt{2}y_2\sin\theta
\right)
&\quad\hbox{ if } (y,\theta)\in \ring{C}_1\times [0,2\pi[;\\
\Phi_2(y,\theta)=r_1\circ\Phi_1(s_1(y),\theta)
&\quad\hbox{ if } (y,\theta)\in \ring{C}_2\times [0,2\pi[;\\
\Phi_3(y,\theta)=r_2\circ\Phi_1(s_2(y),\theta)
&\quad\hbox{ if } (y,\theta)\in \ring{C}_3\times [0,2\pi[;\\
\Phi_4(y,\theta)=r_1\circ r_2\circ \Phi_1(s_1\circ s_2(y),\theta)
&\quad\hbox{ if } (y,\theta)\in \ring{C}_4\times [0,2\pi[.
\end{array}
\end{align*} 
Then, for $i\in I$ and $y\in \ring{C}_i$, $\gamma_y:=\{\Phi_i(y,\theta);\;\theta\in [0,2\pi[\}$ is a periodic orbit living in $M_i$.
We also have that for $i\in I$, $\pi\circ\Phi_i(y,\theta)=y$ for all $(y,\theta)\in\ring{C}_i\times [0,2\pi[$, so that $\pi(M_i)=\ring{C}_i$. 

For $i\in I$ and $(y,\theta)\in \ring{C}_i \times [0,2\pi[$, set $x=\Phi_i(y,\theta)$. The Jacobian matrix of $\Phi_1$ is, 
$$\text{Jac}(\Phi_1)(y,\theta)=
\begin{pmatrix}
\frac{y_1}{x_1}& -\frac{y_2\sin^2\theta}{x_1} & -\frac{y_2^2\sin\theta\cos\theta}{x_1}\\
0&\cos\theta & -y_2\sin\theta\\
0&\sqrt{2}\sin\theta & \sqrt{2}y_2\cos\theta
\end{pmatrix}
$$
and the Jacobian of $\Phi_1$ is $J_1(y,\theta)=\frac{\sqrt{2}y_1y_2}{x_1}$. Similarly, the Jacobian $J_2$, $J_3$ and $J_4$ of $\Phi_2$, $\Phi_3$ and $\Phi_4$ are given by:
$$J_2(y,\theta)=\frac{\sqrt{2}y_1y_2}{x_2};\quad
J_3(y,\theta)=\frac{\sqrt{2}y_1y_2}{x_1};\quad
J_4(y,\theta)=\frac{\sqrt{2}y_1y_2}{x_2}.$$

Therefore for all $f\in L^1(m)$,
\begin{equation}
\label{eq:changevariable}
\int_{\mathbb{R}^3} f dm = \sum_i \int_{C_i\times [0,2\pi)} f\circ \Phi_i(y,\theta) |J_i(y,\theta)| e^{-W(\|y\|)} dyd\theta.
\end{equation}
Set $\wt m:=\pi_*m$ and for $i\in I$ and $y\in C_i$, set $h_i(y)=\int_0^{2\pi}| J_i(y,\theta)|d\theta$. 
Then, \eqref{eq:changevariable} yields that 
\begin{equation}\label{eq:wtmr3}
\wt m(dy)= \sum_{i=1}^4 \1_{C_i}(y) h_i(y)e^{-W(\|y\|)} \,dy.
\end{equation}

Let us denote, for $t\in [0,1[$,
$$K(t)=\int_0^{\pi/2}\frac{d\theta}{\sqrt{1-t^2\sin^2\theta}} \quad\hbox{ and }\quad E(t)=\int_0^{\pi/2} \sqrt{1-t^2\sin^2\theta} d\theta.$$
Then
$$
\begin{array}{ll}
h_1(y)=4\sqrt{2}  y_2K\left(\frac{y_2}{y_1}\right);\quad & h_3(y)=4\sqrt{2} |y_2|K\left(\frac{y_2}{y_1}\right);\\
h_2(y)=4\sqrt{2} y_1K\left(\frac{y_1}{y_2}\right);\quad & h_4(y)=4\sqrt{2} |y_1|K\left(\frac{y_1}{y_2}\right).
\end{array} $$

For all $y\in\wt M\setminus D$, let $\nu_y$ be the probability measure on $\mathbb{R}^3$ defined by
$$\nu_y (f) = 
\frac{1}{h_i(y)}
\int_0^{2\pi} f\circ \Phi_i(y,\theta) |J_i(y,\theta)| d\theta \quad\text{ if } y\in \ring{C}_i,$$ 
$\nu_y=\de_{(y_1,0,0)}$ if $y_2=0$ and $\nu_y=\de_{(0,y_2,0)}$ if $y_1=0$.

Then (Proposition 7 in \cite{MP}), the set of ergodic measures is $E=\{\de_{x}, x\in\cS\}\cup \{\nu_y: y\in\bigcup_i \mathring{C}_i\}$. 
Moreover if $y\in\bigcup_i \mathring{C}_i$ converges to $(y_0,y_0)\in D$, then $\nu_y$ narrowly converges to $\frac12(\de_{(0,0,\sqrt{2}y_0)}+\de_{(0,0,-\sqrt{2}y_0)})$ (Proposition 6 in \cite{MP}).

For $f\in L^1(m)$, it may be useful to use the change of variable formula \eqref{eq:changevariable} written in the form
\begin{equation}
\label{eq:changevariable2}
\int f(x) m(dx) = \int \nu_y(f) \wt m(dy).
\end{equation}

\subsection{Assumptions \ref{hyp:pi} and \ref{hyp:tight}.} 
It can easily be checked that Assumption \ref{hyp:pi} is satisfied with the map $p: \wt M \to E$ defined by
\begin{align}
p(y)=\left\{\begin{array}{ll}
\nu_y&\text{ if } y\in \wt M\setminus D\\
\de_{(0,0,0)} &\text{ if } y\in D.
\end{array}\right.
\end{align}

\begin{proposition}
Assumption \ref{hyp:tight} is satisfied.
\end{proposition}
\begin{proof}
Let $\wt C$ be the set of all $f\in C_c(\wt M)$ such that $0\not\in \text{Supp}(f)$, $f$ is $C^2$ on $\ring{C}_i$ for all $i$ and such that $D f(y)$ and $\text{Hess}(f)(y)$ are continuously extendable at all $y\in D$.

If $f\in \wt C$, then $f\circ \pi \in C^2_c(\mathbb{R}^3)$.
Indeed, $f\circ \pi$ is continuous with support included in $\mathbb{R}^3\setminus\{x_1=x_3=0 \hbox{ or } x_2=x_3=0\}$ and since the mappings $x\mapsto \text{Jac}(\pi)(x)$ and $x\mapsto \text{Hess}(\pi)(x)$ are continuous on $\mathbb{R}^3\setminus\{x_1=x_3=0 \hbox{ or } x_2=x_3=0\}$, we have that $f\circ \pi \in C^2_c(\mathbb{R}^3)$.

Items $(i)$ and $(iii)$ of Assumption \ref{hyp:tight} can be proved with this set $\wt C$ in the same way as in the proof of Proposition \ref{prop:Ctilde1}.
The proof of item $(ii)$ of Assumption \ref{hyp:tight} is more complicated and is given in the appendix.
\end{proof}

\subsection{Application of Theorem \ref{th:main}}
Let $(\wt \cE,\wt \cH)$ be the Dirichlet form on $L^2(\wt m)$ obtained by contracting $(\cE,\cH)$ on $\wt M$ using Proposition \ref{prop:DF2}. 
By Proposition \ref{prop:regular},  $(\wt \cE,\wt \cH)$ is regular possesses the local property. 
Moreover, $(\wt \cE,\wt \cH)$ is a contraction of  $(\cE^\ka,\cH)$ for all $\ka$, where $\cE^\ka=\cE+\ka \cE^V$.
Then we can apply Theorem \ref{th:main}.

\subsection{The Dirichlet form $\wt\cE$ and its generator} \label{sec:desc2} 
The aim of this section is to give a more comprehensive description of the limiting diffusion process $\wt X$ on $\wt M$ associated to the regular Dirichlet form $(\wt\cE,\wt\cH)$ by computing the infinitesimal generator of the associated semigroup. 

Recall that $m(dx)=e^{-\mathcal{W}(x)} \,dx$, that the carré du champ is given by $\Gamma(f,f)=|\grad f|^2$, for $f\in C^1(\mathbb{R}^3)$, and that $V_0=0$. Recall also that, $\mathcal{H}=H^1(m)$.

\begin{proposition}
For $y\in \bigcup_{i\in I}\ring{C}_i$, set
$
a(y):=\nu_y(\nabla\pi\otimes\nabla\pi)
.$
Then
\begin{equation} \label{eq:Etildefg}
\wt\cE(f,g)=\int_{\wt M}\sum_{1\leq k,\ell \leq 2}a^{k\ell}\del_kf\del_\ell g\,d\wt m, \quad \text{for all $(f,g)\in\wt\cH^2$}.
\end{equation}
\end{proposition}

\begin{proof}
Let $(f,g)\in\wt\cH^2$. We get
\begin{equation}
\wt\cE(f,g)=\int \Gamma(f\circ \pi,g\circ \pi)\,dm =\int \sum_{1\leq k,\ell\leq 2}(\nabla\pi^k\cdot\nabla\pi^\ell)\,\big((\del_kf \del_\ell g)\circ\pi\big) \,dm.
\end{equation}
Then \eqref{eq:Etildefg} follows from the change of variable formula \eqref{eq:changevariable2}.
\end{proof}

The matrix $a$ is given in the Appendix (Equation \eqref{eq:a.matrix}) and takes the form:
\begin{align*}
a^{11}(y)&=1-\frac{\nu_y(x_3^2)}{4y_1^2},&
a^{12}(y)&=a^{21}(y)=\frac{\nu_y(x_3^2)}{4y_1y_2},&
a^{22}(y)&=1-\frac{\nu_y(x_3^2)}{4y_2^2}.
\end{align*}
As $y$ goes to $(y_0,y_0)\in D$, using the remark at the end of section \ref{sec:6.descriptionE}, $\nu_y(x_3^2)$ converges to $2y_0^2$. Thus $a^{11}$, $a^{12}=a^{21}$ and $a^{22}$ all converges to $\frac12$.

For $y\in\wt M$, set $\wt\cW(y):=W(\|y\|)$ and for $k\in\{1,2\}$, define $b^k:\bigcup_{i\in I}\ring{C}_i\to\mathbb{R}$ such that for each $i\in I$,
\begin{equation}\label{eq:driftR3}
b^k_i=\sum_{l=1,2}{h_i^{-1}e^{\wt\cW}}{\del_l(h_ie^{-\wt\cW}a^{kl}_i)}.
\end{equation}
Let $f:\wt M\to\mathbb{R}$. For $i\in I$, set $f_i:=f_{|C_i}$.
For $y\in \partial C_i^*$, set $n_i(y)\in \mathbb{R}^2$ the unit outward normal vector at $y$.

For $y\in C^*$, set $t(y)=\frac{|y_1|\wedge |y_2|}{|y_1|\vee |y_2|}$, $r(y)=\|y\|$ and $\theta(y)=\arctan t(y)$. Then $(t(y),r(y),\theta(y))\in [0,1]\times ]0,\infty[\times [0,\pi/4]$.
Note that the mapping $y\mapsto (r(y),\theta(y))$ is a $C^1$-diffeomorphism from $\ring{C}_i$ onto $]0,\infty[\times ]0,\pi/4[$. As noticed in the appendix, for $i\in I$ and $y\in C_i$, the eigenvalues of $a_i(y)$ are $1$ and $\lambda(t(y))$. 
For $\theta\in [0,\pi/4]$, set $\lambda_\theta=\lambda(\tan\theta)$ and $h_\theta=4\sqrt{2}K(\tan\theta)\sin\theta$.
Then, using the coordinates $(r,\theta)$ on $C_i$, $\wt m_i(dy)=h_\theta r^2 e^{-W(r)}dr d\theta$ and for $(f,g)\in \wt \cH^2$,
\begin{align*}
\wt\cE(f,g)
&=\sum_{i\in I}\int_0^{\pi/4}\int_0^\infty \left(\partial_r f_i \partial_r g_i + \frac{\lambda_\theta}{r^2}  \partial_\theta f_i \partial_\theta g_i\right) h_\theta r^2e^{-W(r)} dr d\theta.
\end{align*}

For $t\in [0,1]$ and $\theta=\arctan t$, set 
$n_1(t)=(-\sin\theta,\cos\theta)$, $n_2(t)=(\cos\theta,-\sin\theta)$,
$n_3(t)=-n_1(t)$ and $n_4(t)=-n_2(t)$. Note that if $i\in I$ and $y\in D^*$, then $n_i(y)=n_i(1)$.

For $i\in I$, $y\in C_i$ and $f$ such that $\nabla f_i$ is defined at $y$, set $\partial_{v}f_i(y)=v\cdot \nabla f_i(y)$ for all $v\in \mathbb{R}^2$. 
Note in particular that $\partial_\theta f_i(y)=\partial_{n_i(t)} f_i(y)$ (with $t=t(y)$) and $\partial_r f_i(y)=\partial_{v} f_i(y)$, with $v=\frac{y}{r}$.

For $(i,t)\in I\times [0,1]$, set $D_i^t:=\{y\in C_i:\, t(y) = t\}$ and denote by $\rho_i^t$ the isometry between $D$ and $D_i^t$. 

\begin{proposition}\label{prop:atilde2} A function $f$ belongs to $\mathcal{D}(\wt A)$ if and only if 
 $f\in \wt \cH$ and 
\begin{enumerate}[(i)]
\item For all $i\in I$, $f_i\in H^2_{loc}(\ring{C}_i)$ and $A_if_i:= \sum_{1\leq k,\ell\leq2}a^{k\ell}_i\del^2_{k\ell}f_i+\sum_{k=1,2}b^k_i\del_kf_i$ belongs to $L^2(\wt m_i)$;
\item For all $i\in I$, we have that as $t\to 1$,
$(\partial_{n_i(t)}f_i)\circ \rho_i^t$ converges weakly in $L^2(e^{-W(r)}dr)$. Denoting this limit $\partial_{n_i}f_i$, we have:
\be\label{eq:boundarycond4}
\sum_{i\in I}\partial_{n_i} f_i=0.
\ee
\end{enumerate}
Moreover, if $f\in \mathcal{D}(\wt A)$ and if $y\in \ring{C}_i$, then $\wt A f(y)=A_if_i(y)$.
\end{proposition}

\begin{proof}
Suppose first that $f\in\mathcal{D}(\wt A)$. Then by definition, $f\in \wt H$ and there exists $u=\wt A f\in L^2(\wt m)$ such that for all $g\in \wt \cH$,
\begin{equation}\label{eq:Atildef}
\wt \cE(f,g)=-\langle u, g\rangle_{L^2(\wt m)}.
\end{equation} 

\textit{(i):} Fix some $i\in I$ and an open set $B\Subset \ring{C}_i$. 
Then \eqref{eq:Atildef} implies that for all $g_i\in H^1_0(B)$, we have that 
\begin{equation}
\int_B (a_i^{k\ell} \partial_k f_i \partial_\ell g_i) h_i e^{-\wt W} \,dy
= -\int_B u_i g_i h_i e^{-\wt W} \,dy.  
\end{equation}
Since $h_i e^{-\wt W} a_i\in C^1(\ring{C}_i)$ and is elliptic, we have (following the proof of $C_1$ p185 in \cite{Brezis}) that the restriction of $f_i$ to $B$ belongs to $H^2(B)$. 
This shows that $f_i\in H^2_{loc}(\ring{C}_i)$
and that $u_i=\sum_{k,l}h_i^{-1}e^{\wt W}\del_l(h_ie^{-\wt W}a^{k\ell}_i\del_kf_i)$. 
This proves \textit{(i)}.

$(ii):$ 
For $i\in I$ and $0\le s<t\le 1$ set $C_i^{s,t}:=\{y\in C_i:\, s\le t(y)\le t\}$.
We have that for all $g\in \wt H$,
\begin{align*}
\wt \cE (f,g) = \lim_{(s,t)\to (0,1)} \sum_{i\in I}  \int_{C_i^{s,t}} (a_i^{k\ell} \partial_k f_i \partial_\ell g_i) h_i e^{-\wt W} \,dy.
\end{align*}

Using on each $C_i$ the polar coordinates $(r,\theta)$, setting $\theta_1=\arctan s$ and $\theta_2=\arctan t$,
\begin{align*}
\int_{C_i^{s,t}} (a_i^{k\ell} \partial_k f_i \partial_\ell g_i) h_i e^{-\wt W}\,dy
&=\int_{\theta_1}^{\theta_2}\int_0^\infty \left(\partial_r f_i \partial_r g_i + \frac{\lambda_\theta}{r^2}  \partial_\theta f_i \partial_\theta g_i\right) h_\theta r^2 e^{-W(r)} dr d\theta.
\end{align*}

Integrating by parts, we obtain that for all $g\in \wt\cH$,
\begin{align}
\sum_{i\in I} \int_{C_i^{s,t}} (a_i^{k\ell} &\partial_k f_i \partial_\ell g_i) h_i e^{-\wt W} \,dy
= -\sum_{i\in I} \int_{C_i^{s,t}} u_i g_i h_i e^{-\wt W} \,dy\\\label{term1}
&+ \lambda_{\theta_2}h_{\theta_2} \sum_{i\in I} \int_{D_i^t} (\partial_\theta f_i) g_i \; e^{-W(r)} \,dr\\\label{term2}
&- \lambda_{\theta_1}h_{\theta_1} \sum_{i\in I} \int_{D_i^s} (\partial_\theta f_i) g_i \; e^{-W(r)} \,dr.
\end{align}

Since 
\begin{align*}
\lim_{(s,t)\to (0,1)}\sum_{i\in I} \int_{C_i^{s,t}} u_i g_i h_i e^{-\wt W} \,dy=\langle u, g\rangle_{L^2(\wt m)}, 
\end{align*}
the equality \eqref{eq:Atildef} implies that both \eqref{term1} and \eqref{term2} converge to $0$. 
Using Lemma \ref{lem:asymp}, we get that $\lim_{\theta\to \pi/4} \lambda_\theta h_\theta=4$. Hence, \textit{(ii)} holds.

\smallskip
Suppose now that $f\in \wt \cH$ is such that \textit{(i)} and \textit{(ii)} are satisfied.
Let $g\in C_c(\wt M)\cap \wt\cH$.
As in the proof of \textit{(ii)} above, we integrate by parts. Since \textit{(ii)} is satisfied, \eqref{term1} converges to $0$ as $t\to 1$, and since $g\in C_c(\wt M)$, \eqref{term2} converges to $0$. Therefore, \eqref{eq:Atildef} holds for all $g\in C_c(\wt M)\cap \wt\cH$ and since $C_c(\wt M)\cap \wt\cH$ is dense in $\wt\cH$, \eqref{eq:Atildef} holds for all $g\in \wt\cH$.
\end{proof}

When $f\in \mathcal{D}(\wt A)$,  using the polar coordinates $(r,\theta)$, we have for all $i\in I$,
\begin{align*}
A_if_i&=r^{-2}e^{W(r)} \partial_{r}\left(r^{2}e^{-W(r)}\partial_rf_i\right)+ (r^2h_\theta)^{-1} \partial_\theta \left(h_\theta \lambda_\theta \partial_\theta f_i\right)\\
&=\partial^2_{rr} f_i + \left(\frac{2}{r} -W'(r)\right)  \partial_rf_i + \frac{\lambda_\theta}{r^2} \left[\partial^2_{\theta\theta} f_i 
+  b_\theta \partial_\theta f_i\right],
\end{align*}
where $b_\theta:=\partial_\theta \log (\lambda_\theta h_\theta)$. Note that as $\theta$ goes to $0$, then $b_\theta\sim \theta^{-1}$ and as $\theta$ goes to $\frac{\pi}{4}$, $b_\theta \sim -\log(\pi/4 - \theta)$ and $\lambda_\theta \sim \frac{-2}{\log(\pi/4 - \theta)}$.

\section{An averaging principle for stochastic flows}\label{sec:avgstoflows}

In Section \ref{sec:avg}, for all $\kappa$, $X^\kappa$ is a diffusion in a manifold $M$ with generator $A^\kappa=A+\kappa V$, with $V$ a vector field, and 
we have proved the convergence in law as $\kappa\to\infty$ of $\pi(X^\kappa)$ towards a Markov process $\wt X$ in $\wt M$, where $\pi:M\to \wt M$ is a measurable mapping determined by $V$.

We extend in this section this result to the convergence of $K^\kappa$ the stochastic flow of kernel (SFK) solution to the following SDE on $M$
\begin{equation}
K^\kappa_{s,t}f(x)=f(x)+\int_s^t K_{s,u}^\kappa(Wf(du))(x) + \int_s^t K^\kappa_{s,u}(A^\kappa f)(x) du,
\end{equation}
where $W$ is a vector field white noise of covariance $C$ (see Section 5 in \cite{FCN}). 
The main result of this section is Theorem \ref{th:mainflows} which states that $\wt K^\kappa$ converges in law (in the sense of the finite dimensional distributions)  to a stochastic flow on $\wt M$, where $\wt K^\kappa_{s,t}f(\wt x)=p(\wt x) K^\kappa_{s,t} (f\circ \pi)$ (recall that $p(\wt x)$ is an ergodic probability measure for the vector field $V$). 

The law of a SFK is described by a consistent and exchangeable family of $n$-point motions, which can be a system of $n$ particles solving a given SDE. 
Using the framework developed in Section \ref{sec:avg}, we prove that the $n$-point motion of $K^\kappa$ converges to the $n$-point motion of $\wt K$, a SFK on $\wt M$.

In Section \ref{sec:hyp71}, the one point motion ($X^\kappa$) of $K^\kappa$ is described.
From Section \ref{sec:whitenoise} through \ref{sec:DFn}, we recall the notion of the covariance for a vector field valued white noise on $M$ and define the $n$-point motion of $K^\kappa$.
In Sections \ref{sec:hyp7} and \ref{sec:DFtilden}, we apply Theorem \ref{th:main} to prove the convergence of the $n$-motions, and therefore to prove Theorem \ref{th:main2}. 
In Sections \ref{sec:consistency} and \ref{sec:avgflows}, we recall the definition of a SFK and use Theorem \ref{th:main2} to prove Theorem \ref{th:mainflows}. We finish this section by revisiting the Examples of Section \ref{sec:hamilt} and Section \ref{sec:R3}.

\subsection{The one-point motion}\label{sec:hyp71}
We use the framework of the Section \ref{sec:avg}. That is :
\begin{itemize}
\item $M$ is a smooth oriented Riemannian manifold and $m\in \cM_+(M)$ is equivalent to the volume form on $M$, with a $C^1$ density.
\item $\Gamma$ is a carré du champ operator associated to a second order differential operator $S$, symmetric on $L^2(m)$, as in section \ref{sec:DFM}.
\item $V_0$ and $V$ are $C^1$-vector fields.
\end{itemize}
We suppose that $V_0$ satisfies  \eqref{eq:hyp(3)} and that there is a measurable function $v:M\to \mathbb{R}$ such that $|V_0f|^2+|Vf|^2\le v^2\Gamma(f,f)$ and that one of the three following conditions is satisfied\begin{enumerate}
\item $v\in L^\infty(m)$.
\item For some $q>2$, $v\in L^{nq}(m)$ for all $n\ge 1$, $m$ is probability measure and $(\Gamma,m)$ satisfies a Sobolev inequality of dimension $q>2$.
\item $\int e^{\lambda v^2} dm<\infty$ for all $\lambda>0$, $m$ is a probability measure and $(\Gamma,m)$ satisfies a logarithmic Sobolev inequality.
\end{enumerate}
Then Proposition \ref{prop:WScond} implies that $V_0$ satisfies \eqref{eq:sector1}. 
Applying Proposition \ref{prop:DF0}, set $(\cE,\mathcal{H})$ the Dirichlet form on $L^2(m)$ associated to $\Gamma$ and $V_0$.

We suppose that $V$ is complete and that
\begin{itemize}
\item $m$ is invariant for the flow $(\phi_\cdot)$ associated to $V$, i.e. $\hbox{div}_m V=0$;
\item $f\circ \phi_t \in \mathcal{H}$ for all $f\in C^\infty_c(M)$ and all $t\in\mathbb{R}$.
\end{itemize}
For all $\kappa\in\bbR$, define the bilinear form $\cE^\kappa:=\cE+\kappa \cE^V$, with $\cE^{V}$ the antisymmetric form defined by \eqref{eq:EV}.
The vector field $V$ satisfies \eqref{eq:sector1} and 
by Proposition \ref{prop:DF0}, it holds that $(\cE^{\ka},\mathcal{H})$ is a regular Dirichlet form on $L^2(m)$, for all $\kappa\in\mathbb{R}$.

We denote by $A$ (resp. $A^\ka$) the infinitesimal generator of $(T_t)$ (resp. $T_t^\ka$) associated to $\cE$ (resp. to $\cE^\ka$). Then, for all $\ka$, it holds that $A$ and $A^\ka$ have the same domain (i.e. $\cD(A^\ka)=\cD(A)$) 	and $A^\ka=A+\ka V$. 
Note that $C^\infty_c(M)\subset \cD(A)\subset \cH$, and for $f\in C^\infty_c(M)$, $Af=(S+V_0) f$.

We suppose that Assumptions  \ref{hyp:F} and \ref{hyp:pi} are satisfied. 
The results of section \ref{sec:contract} are applied with $\pi:M\to \wt M$. Define $\wt \cH$ as in \eqref{eq:dom}, set $\wt m=\pi_*m$ a measure on $\wt M$. Finally, we suppose that Assumption \ref{hyp:tight} is satisfied.

\subsection{Vector field valued white noises}\label{sec:whitenoise}
\begin{definition}[Definition 5.1 in \cite{FCN}]
$C$ is a {\it covariance function} on the space of vector fields on $M$, 
if $C:T^*M\times T^*M\to \bbR$ is a symmetric map whose restriction to $T^*_xM\times T^*_yM$ is bilinear and such that for any $n$-uples $(\xi_1,\dots,\xi_n)\in (T^*M)^n$, 
$$\sum_{i,j}C(\xi_i,\xi_j)\ge 0.$$ 
\end{definition}
For $(f,g)\in C^1_c(M)\times C^1_c(M)$, define $C(f,g):M\times M\to\bbR$ by $C(f,g)(x,y):=C(df(x),dg(y))$. 
The covariance $C$ is assumed to be {\it continuous} in the sense that $C(f,g)$ is continuous for all $(f,g)\in C^1_c(M)\times C^1_c(M)$.

In Definition 5.3 of  \cite{FCN}, {\it a vector field valued white noise of covariance $C$} is defined as a two-parameter family $W=(W_{s,t},s\le t)$ satisfying (b) and (c) (with $K$ replaced by $W$), 
such that for $s\le t\le u$, $W_{s,u}=W_{s,t}+W_{t,u}$ and for all $s\le t$,  $\{\langle W_{s,t},\xi\rangle,\xi\in T^*M\}$ is centered Gaussian with covariance function given by
\be \bbE[\langle W_{s,t},\xi\rangle\langle W_{s,t},\xi'\rangle]=(t-s)C(\xi,\xi'),\hbox{ for } \xi,\xi'\in T^*M.\ee

Following Section 1 in \cite{IBVF}, the covariance $C$ can be written in the form $C=\sum_k U_k\otimes U_k$, with $(U_k)_k$ a family of vector fields.
Thus, a vector field valued white noise $W$ of covariance $C$ can be constructed out of an independent family $(W^k)_k$ of standard white noises by $W=\sum_k W^k U_k$.

Let $C$ be a continuous covariance such that $C$ is dominated by the carr\'e du champ $\Gamma$ in the sense that 
\be\label{diffusion form} C(f,f)(x,x)\le \Gamma(f,f)(x),\hbox{ for all $f\in C^1(M)$ and $x\in M$}.\ee
When \eqref{diffusion form} is an equality, we will say that there is {\it no pure diffusion}, and that there is a {\it pure diffusion} otherwise.
We will say that there is a {\it uniformly pure diffusion} if there exists $\de>0$ such that
\be C(f,f)(x,x)\le (1-\delta)\Gamma(f,f)(x),\hbox{ for all $f\in C^1(M)$ and $x\in M$}.\label{u-pure diffusion}\ee

\subsection{The generator $A^{(n),\ka}$}

For $x\in M^n$, we identify $T_{x_i}M$ and $T^*_{x_i}M$ respectively to linear subspaces of $T_x M^n$ and $T^*_xM^n$. For $f\in C^\infty_c(M^n)$ and $x\in M^n$, let $d_if(x)$ be the restriction of $df(x)$ to $T_{x_i}M$.
Then $d_if(x)\in T^*_{x_i}M$ and $df(x)=\sum_{i=1}^n d_if(x)$. 
For $u\in T_{x}M$ and $\xi\in T^*_xM$, denote the dual pairing $(u,\xi):=\xi(u)$ or, in local coordinates,  $(u,\xi)=\sum_k u^k\xi_k$.

For $n\ge 1$, let $A^{(n),\ka}$ be the second-order differential operator on $M^n$ defined by  (with $A_\ka=A+\ka V$)
\be A^{(n),\ka} f(x)=\sum_i A_\ka f_i(x_i) \prod_{k\ne i} f_k(x_k) + \sum_{i\ne j} C(f_i,f_j)(x_i,x_j) \prod_{k\ne i,j} f_k(x_k), \ee
for all $f\in C^2_c(M^n)$ of the form $f(x)=\prod_{i=1}^n f_i(x_i)$, with $f_i\in C^2_c(M)$. 

Set $A^{(n)}:=A^{(n),0}$ and note that $A^{(n),\kappa}=A^{(n)}+\kappa V^{(n)}$, where $V^{(n)}$ is the complete $C^1$-vector field on $M^n$ defined by 
\be \label{eq:V(n)} V^{(n)} f(x) =  \sum_{i=1}^n (V(x_i), d_if(x)), \hbox{ for $f\in C^\infty_c(M^n)$}.\ee

For $1\le i, j\le n$, define $d_id_jf(x)\in T^*_{x_i}M\otimes T^*_{x_j}M$. Again, we identify $T^*_{x_i}M\otimes T^*_{x_j}M$ to a linear subspace of $T^*_{x}M^n\otimes T^*_{x}M^n$. 
Then the Hessian of $f$ is $d^2f(x)=\sum_{i,j} d_id_j f(x)$.

The covariance function $C$ can be extended to a function 
\be C : \bigcup_{x_1,x_2\in M} T^*_{x_1}M\otimes T^*_{x_2}M \to \bbR. \ee
Then, for $f\in C^\infty_c(M^2)$, one can define $Cf(x_1,x_2):=C(d_1d_2f(x))$.
For $i\ne j$, define $C_{i,j}:\cup_{x} T^*_xM\otimes T^*_xM \to \bbR$ by
\be C_{i,j}(d^2f(x))=C(d_id_jf(x)), \hbox{ for $f\in C^\infty_c(M^n)$.}\ee
In the following, we set $C_{i,j}f(x):=C(d_id_jf(x))$.

\begin{lemma}
For $f\in C^\infty_c(M)$ and $x\in M^n$, we have 
\be
A^{(n),\ka} f(x) = \sum_{i=1}^n \left( A_i f(x)+ \kappa(V(x_i),d_if(x))\right) +   \sum_{i\ne j} C_{i,j} f (x)
\ee
with $A_i f$ such as $A_if(x)=Af_i(x_i) \prod_{k\ne i} f_k(x_k)$ when $f$ can be written in the form $f(x)=\prod_{k=1}^n f_k(x_k)$.
\end{lemma}

\begin{remark}
Consider the case $M=\bbR^d$. Let $V$ be a vector field on $\mathbb{R}^d$. Let $(U_k)$ be a family of vector fields such that $C=\frac12\sum_k U_k\otimes U_k$, suppose that $\G(f,g)=\frac12 \grad f\cdot\grad g$ and let $m$ be a measure absolutely continuous with respect to the Lebesgue measure on $\mathbb{R}^d$, with density denoted $e^{-\mathcal{W}}$.
Suppose that the assumptions given in subsection \ref{sec:hyp71} and suppose that $C(f,g)(x,x)=(1-\epsilon^2)\Gamma(f,g)(x)$.
 Then $A=\frac{1}{2}\Delta + V_0 - \nabla \mathcal{W}$.  and $A^{(n),\kappa}$ is the infinitesimal generator of the diffusion $X=(X_1,\dots,X_n)$ on $M^n$ solution of the SDE
\be
d X_i(t)=\epsilon dB_i(t)+\sum_k U_k(X_i(t))d W^k(t) + (V_0-\nabla \mathcal{W}+\ka V)(X_i(t))dt
\ee
where $(B_i)_i$ and $(W^k)$ are independent families of independent Brownian motions respectively in $\mathbb{R}^d$ and in $\mathbb{R}$. 
\end{remark}

\subsection{Integration by parts formulas and the carr\'e du champ $\Gamma^{(n)}$}\label{sec:integrationbyparts}
A $C^1$-vector field $U$ on $M^n$ can be written in the form $U=\sum_{i=1}^n U_i$, where for $1\le i\le n$ and $x\in M^n$, $U_i(x)\in T_{x_i}M$ is defined by
\be (U_i(x),d_if(x))=(U(x),d_if(x)), \hbox{ for $f\in C^1(M^n)$.}\ee
Set $U_if(x):=(U(x),d_if(x))$.
Fixing $x_j,\,j\ne i$, $U_i$ can be viewed as a vector field (depending on $x_j,\,j\ne i$) on $M$.   
One can define divergence operators $\delta_i$, $1\le i\le n$, depending on $m$, by:
\be \delta_i U=\Div_m U_i.\ee

We then have the following integration by parts formula, which is a direct consequence of \eqref{eq:IBP1}:
\begin{lemma} for $f,g\in C^\infty_c(M^n)$ and $1\le i\le n$,
\be \langle U_if, g \rangle_{L^2(m^{\otimes n})}
= - \langle f,U_ig \rangle_{L^2(m^{\otimes n})} - \int_{M^n} f g (\delta_i U) dm^{\otimes n}.\ee
\end{lemma}
Note that $\sum_i \delta_i=\hbox{div}_{m^{\otimes n}}$, and we obtain the integration by parts formula \eqref{eq:IBP1} on $M^n$  by taking the sum in $i$.

\smallskip
To the vector field $V$ on $M$, define vector fields $V_i$, $1\le i\le n$, on $M^n$ such that $V_i(x)=V(x_i)\in T_{x_i}M\subset T_xM^n$.

When the covariance function $C$ is $C^1$ (i.e. $C(f,g)\in C^1(M^2)$ for all $f,g\in C^\infty(M)$), the vector fields on $M^n$, $\delta_i C_{i,j}$ and $\delta_j C_{i,j}$, are well-defined. 
For $i\ne j$, note that $\delta_i C_{i,j}(x)=\delta_1 C(x_i,x_j)\in T_{x_j}M$ and $\delta_j C_{i,j}(x)=\delta_2 C(x_i,x_j)\in T_{x_i}M$. 
Note also that $\delta_i\delta_jC_{i,j}(x)=\delta_1\delta_2 C(x_i,x_j)=\delta_1\delta_2 C(x_j,x_i)$ depends only on $x_i$ and $x_j$ and is symmetric.
In the case where $\delta_i C_{i,j}=0$ for $i\ne j$, the covariance $C$ will be said \textit{$m$-divergent free}. Note that in this case the stochastic flow generated by a Brownian vector field of covariance $C$ preserves the measure $m$.

Writing $C=\sum_k U_k\otimes U_k$, so that $C_{i,j}(x_i,x_j)=\sum_k U_k(x_i)\otimes U_k(x_j)$ and setting $D_k=\Div_m U_k$,
\begin{align} 
\delta_i C_{i,j}(x)&=\sum_k D_k(x_i)\, U_k(x_j) \in T_{x_j}M,\\
\delta_j C_{i,j}(x)&=\sum_k U_k(x_i)\,D_k(x_j) \in T_{x_i}M,\\
\delta_i \delta_j C_{i,j}(x)&=\sum_k D_k(x_i)\, D_k(x_j).
\end{align}

The integration by parts formula yields
\begin{align}
\langle C_{i,j}f, g\rangle_{L^2(m^{\otimes n})}
&= - \int_{M^n} C(d_if,d_jg) dm^{\otimes n} - \langle (\delta_jC_{i,j})f ,g \rangle_{L^2(m^{\otimes n})}\label{IBPCij}
\end{align}
\begin{lemma}\label{IBPA(n)} Assume that $C$ is $C^1$. Then,
for $f,g\in C^\infty_c(M^n)$,
\be -\langle A^{(n)} f, g\rangle_{L^2(m^{\otimes n})} =  \int_{M^n} \Gamma^{(n)}(f,g) \,dm^{\otimes n} - \langle V^{(n)}_0 f,g\rangle_{L^2(m^{\otimes n})}\ee
with
\begin{align}
\Gamma^{(n)}(f,g) &= \sum_{i=1}^n \Gamma(d_if,d_ig) +  \sum_{i\ne j} C(d_if,d_jg),\label{eq:defGamma(n)}\\
V^{(n)}_0 f(x) &= V^{(n)}_C f (x) + \sum_{i=1}^n (V_0(x_i), d_if(x))\label{eq:defv0(n)}
\end{align}
where $V^{(n)}_C:=-\sum_{i\ne j} \delta_jC_{i,j}$ is a vector field on $M^n$.
\end{lemma}
\begin{proof}
Recall that $A^{(n)}f$ can be written in the form
\be 
A^{(n)} f(x) = \sum_{i=1}^n \left( S_i f(x)+ (V_0(x_i),d_if(x))\right) + \sum_{i\ne j} C_{i,j} f (x).
\ee
The operator $S_i$ are symmetric in $L^2(m^{\otimes n})$, and
\be
 -\sum_i \langle S_i f, g\rangle_{L^2(m^{\otimes n})} = \sum_i \int_{M^n} \G(d_if,d_ig) \,dm^{\otimes n}.
\ee
Using \eqref{IBPCij}, we get that 
\be -\sum_{i\ne j}\langle C_{i,j}f, g\rangle_{L^2(m^{\otimes n})}
=  \int_{M^n} \sum_{i\ne j}C(d_if,d_jg) dm^{\otimes n} - \langle V^{(n)}_C f ,g \rangle_{L^2(m^{\otimes n})}\ee
and this implies the lemma. \end{proof}

\begin{lemma}\label{lem:gamma(n)ue} Suppose that $C$ is such that there is a uniformly pure diffusion, i.e. \eqref{u-pure diffusion} is satisfied with $\delta>0$. Then for all $f\in C^\infty_c(M^n)$, 
\begin{align}
\Gamma^{(n)}(f,f)
&\ge \delta \sum_{i=1}^n \Gamma(d_if,d_if).
\end{align}
In particular, if $\Gamma$ is uniformly elliptic, i.e. satisfies \eqref{Guelliptic},  then $\Gamma^{(n)}$ is also uniformly elliptic.
\end{lemma}
\begin{proof} This simply follows from the following calculation
\begin{align}
\Gamma^{(n)}(f,f)
&= \sum_{i=1}^n \left(\Gamma(d_if,d_if)-C(d_if,d_if)\right) +  \sum_{i,j} C(d_if,d_jf)\\
&\ge \delta \sum_{i=1}^n \Gamma(d_if,d_if),
\end{align}
where we have used in the last inequality \eqref{u-pure diffusion} and that $C$ is non negative.
\end{proof}

\subsection{The Dirichlet form $\cE^{(n),\kappa}$ associated to $\Gamma^{(n)}$ and $V_0^{(n)}+\kappa V^{(n)}$}\label{sec:DFn}
From now on, we assume that there is a uniformly pure diffusion, i.e. \eqref{u-pure diffusion} holds for some $0< \de\le 1$.
We also suppose that $C$ is $C^2$, that there is a constant $c<\infty$ such that 
\be \label{C:DF1} |(\delta_2 C(x_1,x_2), df(x_1))|^2 \le c \Gamma(f,f)(x_1), \hbox{ for all $x\in M^2$ and $f\in C^\infty_c(M)$} .\ee
and that
\be \label{C:DF2} \|\delta_1\delta_2 C\|_\infty < \infty .\ee

\begin{remark} Suppose that $\Gamma$ is uniformly elliptic, i.e. that \eqref{Guelliptic} holds and that 
$C=\sum _k U_k\otimes U_k$ with $(U_k)_k$, a family of $C^1$-vector fields.
For all $k$, set $D_k=\Div_m U_k$. Then, \eqref{C:DF1} and \eqref{C:DF2} are equivalent to the existence of a finite constant $c$ such that
\be
\delta_1\delta_2 C(x,x)=\sum_kD_k(x)^2\le c \quad \hbox{for all $x\in M.$}
\ee
Condition \eqref{C:DF2} is simply obtained by using Cauchy-Schwartz inequality and \eqref{C:DF1} is obtained by using also \eqref{diffusion form} and \eqref{eq:hyp(1)}.
\end{remark}

\begin{lemma}\label{lemme:Ga(n)}
$\G^{(n)}$ and $V_0^{(n)}$ satisfy \eqref{eq:hyp(3)} and \eqref{eq:sector1}.
\end{lemma}
\begin{proof}
We compute
\begin{align*}
\hbox{div}_{m^{\otimes n}} V_0^{(n)}(x)
&= \sum_i \delta_i V_0^{(n)}(x) = \sum_i \hbox{div}_m V_0(x_i) - \sum_{i\ne j} \delta_i\delta_j C_{i,j}(x_i,x_j)\\
&\ge nc_1 - n(n-1)\|\delta_1\delta_2 C\|_\infty > -\infty
\end{align*}
and \eqref{eq:hyp(3)} is satisfied.

Recall that there is a measurable function $v:M\to\mathbb{R}$ such that $(V_0f)^2\le v^2 \Gamma(f,f)$ for all $f\in C^1(M)$.
To check \eqref{eq:sector1}, we use Proposition \ref{prop:WScond}.
For all $f\in C^1(M^n)$, we have
\begin{align*}
V_0^{(n)}f(x) = \sum_i (V_0(x_i),d_if(x))
- \frac{1}{2} \sum_{i\ne j} (\delta_j C(x_i,x_j),d_if(x)). 
\end{align*}
and so for some constant $K_n$,
\begin{align*}
(V_0^{(n)}f)^2(x) &\le K_n\left(\sum_i v^2(x_i)\Gamma(d_if,d_if)(x)
+ \sum_{i\ne j} c \Gamma(d_if,d_if)(x)\right)\\
&\le  K_n \left(\sum_j v^2(x_j) + nc\right) \left(\sum_i \Gamma(d_if,d_if)(x)\right)\\
&\le v_n^2 \Gamma^{(n)}(f,f)
\end{align*}
with $v_n(x)=(K_n\delta^{-1})^{1/2} \left(\sum_j v^2(x_j) + nc\right)^{1/2}$ and using Lemma \ref{lem:gamma(n)ue} for the last inequality. 

Therefore, if condition (1) in Section \ref{sec:hyp71} holds, then $v_n\in L^\infty(m^{\otimes n})$ and \eqref{eq:sector1} is satisfied by $\Gamma^{(n)}$ and $V_0^{(n)}$.

Suppose now that $\Gamma$ and $V_0$ satisfy condition (2) in Section \ref{sec:hyp71}.
By tensorization, we have the following Sobolev inequality of dimension $nq$ (see Section 6.5 in \cite{BGL}): For all $f\in C^\infty_c(M^n)$,
$$\|f\|^2_{L^{p_n}(m)}\le A_n \|f\|^2_{L^2(m^{\otimes n})} + C_n\int \sum_{i=1}^n\Gamma(d_if,d_if) dm^{\otimes n},$$
with $p_n:=\frac{2nq}{nq-2}$, $A_n\in\mathbb{R}$ and $C_n>0$. Lemma \ref{lem:gamma(n)ue} then implies that $(\Gamma^{(n)},m^{\otimes n})$ satisfies a Sobolev inequality of dimension $nq$. Since $v\in L^{nq}(m)$ and since $m$ is a probability measure, $v_n\in L^{nq}(m^{\otimes n})$ and \eqref{eq:sector1} is satisfied by $\Gamma^{(n)}$ and $V_0^{(n)}$.

Suppose finally that $\Gamma$ and $V_0$ satisfy condition (3) in Section \ref{sec:hyp71}.
By tensorization, we have the following logarithmic Sobolev inequality (see Section 5.2.3 in \cite{BGL}): For all $f\in C^\infty_c(M^n)$,
$$\Ent_{m^{\otimes n}}(f^2)\le 2C_n \int \sum_{i=1}^n\Gamma(d_if,d_if) dm^{\otimes n} + D_n \|f\|^2_{L^2(m^{\otimes n})}$$
with $C_n>0$ and $D_n\ge 0$. Lemma \ref{lem:gamma(n)ue} then implies that $(\Gamma^{(n)},m^{\otimes n})$ satisfies a logarithmic Sobolev inequality (with constants $C_n\delta^{-1}$ and $D_n$).
Since $\int e^{\lambda v^2} dm <\infty$ for all $\lambda>0$, $\int e^{v_n^2} dm^{\otimes n}<\infty$ and \eqref{eq:sector1} is satisfied by $\Gamma^{(n)}$ and $V_0^{(n)}$.
\end{proof}

Applying Proposition \ref{prop:DF0}, set $(\mathcal{E}^{(n)}, \cH^{(n)})$ the Dirichlet form on $L^2(m^{\otimes n})$ associated to $\Gamma^{(n)}$ and $V_0^{(n)}$. 
Note that $\cH^{(n)}=\cH^{\otimes n}$. 
Note also that when $C$ is $m$-divergent free and when $V_0=0$, the Dirichlet forms $\mathcal{E}^{(n)}$ are symmetric for all $n$.

The vector field $V^{(n)}$ is $C^1$ and complete. If $\phi$ denotes the flow generated by $V$. 
The flow generated by $V^{(n)}$ is $\phi^{(n)}:=\phi^{\otimes n}$ and $m^{\otimes n}$ is invariant by $\phi^{(n)}$.
As for $V_0^{(n)}$, it can be checked that $\Gamma^{(n)}$ and $V^{(n)}$ satisfy \eqref{eq:sector1} and that $f\circ \phi_t^{(n)}\in \cH^{\otimes n}$ for all $f\in C^\infty_c(M^n)$ and all $t\in \mathbb{R}$.

Still following Section \ref{sec:avg}, one defines the Dirichlet form $(\cE^{(n),\kappa},\cH^{\otimes n})$ on $L^2(m^{\otimes n})$, with $\cE^{(n),\kappa}=\cE^{(n)}+\kappa\cE^{V^{(n)}}$.

\subsection{Assumptions \ref{hyp:F}, \ref{hyp:pi} and \ref{hyp:tight}}\label{sec:hyp7}

It is straightforward to check that Assumption \ref{hyp:F} is satisfied with $H^{(n)}$ defined by $H^{(n)}(x)=\sum_{i=1}^n H(x_i)$.

Denote by $E$ (respectively $E^{(n)}$) the set of ergodic probability measure of the flow $\phi$ (respectively $\phi^{(n)}$).
Denote by $\cI$ (respectively $\cI^{(n)}$), the $\s$-field generated by the invariant sets of $\cB(M)$ (resp. of $\cB(M^n)$) and completed with the negligible sets of $\cB(M)$ (resp. of $\cB(M^n)$. 
Note that $\cI^{(n)}$ may be different from $\cI^{\otimes n}$, the $n$-product of one point invariant $\s$-field, as is shown in the following example:
\begin{example}\label{ex:periodic}
Let $V$ be the vector field on $M:=\mathbb{R}^2$ defined by $V(x_1,x_2)=(-x_2,x_1)$. 
The orbits of the flow $\phi$ generated by $V$ are the concentric circles with center $(0,0)$. 
All orbits have the same period $T=2\pi$. 
In this case, $\cI$ is generated by the orbits of $\phi$ and when $n=2$, $\cI^{(2)}$ is generated by the sets $A_{r,R,\alpha}:=\{(x,y)\in  M^2: (x,y)=(r e^{i\theta},R e^{i(\theta+\alpha)}) \hbox{ with } \,\theta\in [0,2\pi]\}$, where $(r,R,\alpha)\in [0,\infty[^2\times [0,2\pi[$. 
The set $E^2$ consists of uniform measures on $A_{r,R,\alpha}$ and are not product measures.
In this example, one can take $\wt M=\mathbb{R}_+$ and $\wt M^2=\{(r,R), 0\le r\le R\}\times S^1\neq(\wt M)^2$.
\end{example}

Recall that Assumption \ref{hyp:pi} is satisfied by $V$, i.e. there are $\pi:M\to\wt M$ a continuous mapping and $p:\wt M\to E$ a measurable mapping such that $P:=p\circ \pi$ is a conditional regular probability with respect to $\cI$ and $m$.

We will suppose that the following assumption is satisfied:
\begin{assumption}\label{hyp:mixing}
$m^{\otimes n}(dx)$-a.e., $\otimes_{i=1}^nP(x_i)$ is ergodic.
\end{assumption}

\begin{remark} 
Note that the converse of this assumption is always true: for any $p\in E^{(n)}$ an ergodic measure, the marginal projections $(p_i), i=1,...,n$ of $p$ defined by $p_i(A)=p(M^{i-1}\times A\times M^{n-i})$ for $A\in \cB(M)$ are ergodic measures. 
Assumption \ref{hyp:mixing} entails that essentially all ergodic measures for $\phi^{(n)}$ are products of ergodic measures for $\phi$.
\end{remark}

A measure $\mu$ is said \textit{weakly mixing} for the flow $(\phi_t)_t$ if  for all $f,g\in L^2(\mu)$
\be\label{eq:wmix}
\lim_{t\to+\infty}\frac1t\int_0^t\abs{\bra{f\circ \phi_s,g}_{L^2(\mu)}-(\mu f)(\mu g)}\,ds=0.
\ee
Every weakly mixing measure is ergodic. We have the following lemma.
\begin{lemma}\label{lem:wmix}
If $P(x)$ is weakly mixing $m(dx)$-a.e., then Assumption \ref{hyp:mixing} is satisfied.
\end{lemma}
\begin{proof}
It is easy to check that if $\mu_1$ for a flow $\phi^1$ on $M_1$ and $\mu_2$ for a flow $\phi^2$ on $M_2$ are weakly mixing then $\mu_1\otimes\mu_2$ is weakly mixing (and thus ergodic) for the flow $(\phi^1_t\otimes\phi^2_t)_t$ on $M_1\times M_2$. The lemma follows by repeated applications of this property.
\end{proof}

\begin{remark}\label{rem:periodic}
Note that if $(\phi_t(x))_t$ is periodic then $P(x)$ is ergodic but not weakly mixing. For $(x_1,\dots,x_n)\in M^n$, if each $(\phi_t(x_i))_t$ is periodic with positive minimal period $T_i$, then $\otimes_{i=1}^n P(x_i)$ is ergodic if and only if the periods $(T_1,\dots,T_n)$ are rationally independent.
\end{remark}

\begin{lemma}
If Assumption \ref{hyp:mixing} is satisfied, 
then Assumption \ref{hyp:pi} is satisfied with $\wt M^n= \big(\wt M\big)^n$, $\pi^{(n)}=\pi^{\otimes n}$ and $p^{(n)}:\wt M^n\to E^{(n)}$ a measurable such that  $\wt m^{\otimes n}(d\wt x)$-a.e., $p^{(n)}(\wt x)=\otimes_{i=1}^n p(\wt x_i)$. 
\end{lemma}
\begin{proof}
Note first that $\cI^{\otimes n}=\cI^{(n)}$. Indeed, if $A\in \cI^{(n)}$, then $\1_A=\bbE_{m^{\otimes n}}(\1_A|\cI^{(n)})=P^{(n)}(\cdot)(\1_A)$. Thus since the latter is $\cI^{\otimes n}$ measurable, we have that $\1_A$ is $\cI^{\otimes n}$ measurable and $A\in \cI^{\otimes n}$. The inclusion $\cI^{\otimes n}\subset \cI^{(n)}$ is obvious.

Let $N=\{\wt x\in \big(\wt M\big)^n: \;\otimes_{i=1}^n p(\wt x_i)\text{ is not ergodic}\}$. 
We have that $(\pi^{(n)})^{-1}(N)=\{x\in M^n:\; \otimes_{i=1}^n p\circ \pi(x_i)\text{ is not ergodic}\}$ is negligible. 
Thus $\wt m^{\otimes n}(N)=0$. 
Let us fix $\mu_0\in E^{(n)}$ and define $p^{(n)}$ by letting $p^{(n)}(\wt x)=\mu_0$ for $\wt x\in N$ and $p^{(n)}(\wt x)=\otimes_{i=1}^n p(\wt x_i)$. Then $p^{(n)}:\wt M^n\to E^{(n)}$ is measurable.

Let us now prove that $P^{(n)}:=p^{(n)}\circ\pi^{(n)}$ is a regular conditional probability measure with respect to $\cI^{(n)}$ and $m^{\otimes n}$.
Recall that $\cI^{(n)}=\cI^{\otimes n}$. 
Therefore, it suffices to prove that
\begin{equation}\label{eq:pnrcpm}
\langle P^{(n)}f,g\rangle_{L^2(m^{\otimes n})} = \langle f,g\rangle_{L^2(m^{\otimes n})}
\end{equation}
for all $f=\otimes_{i=1}^n f_i$ and $g=\otimes_{i=1}^n g_i$, where for all $i$, $f_i\in L^2(m)$ and $g_i\in L^2(m)$ is $\cI$-measurable. 
Since $\langle P^{(n)}f,g\rangle_{L^2(m^{\otimes n})}=\prod_{i=1}^n \langle Pf_i,g_i\rangle_{L^2(m)}$, we easily prove \eqref{eq:pnrcpm}.
\end{proof}

Define $\wt \cH^{(n)}$ as in \eqref{eq:dom} and $\wt m^{(n)}:=\pi^{(n)}_* m^{\otimes n}$. Then $\wt \cH^{(n)}=\wt \cH^{\otimes n}$ and $\wt m^{(n)}=\wt m^{\otimes n}$. 
Then it is straightforward that Assumption \ref{hyp:tight} is satisfied by  $\wt \cH^{(n)}$ and $\wt m^{(n)}$, by taking $\wt C^{(n)}$ the vector space spanned by $\{\otimes_{i=1}^n f_i:\; \hbox{ with } (f_1,\dots,f_n)\in \wt C^n\}$. 

\subsection{Application of Theorem \ref{th:main}}\label{sec:DFtilden}
Let $(\wt \cE^{(n)},\wt \cH^{\otimes n})$ be the Dirichlet form on $L^2(\wt m^{\otimes n})$ obtained by contracting $(\cE^{(n)},\cH^{\otimes n})$ on $\wt M^n$ using Proposition \ref{prop:DF2}. 
By Proposition \ref{prop:regular},  $(\wt \cE^{(n)},\wt \cH^{\otimes n})$ is regular and possesses the local property. 
Moreover, $(\wt \cE^{(n)},\wt \cH^{\otimes n})$ is a contraction of  $(\cE^{(n),\ka},\cH^{\otimes n})$ for all $\ka$.

By Theorem \ref{th:diff}, for all $\ka$, there is $X^{(n),\ka}=(X^{1,\kappa},\dots,X^{n,\kappa})$ a diffusion on $M^n$ associated to $\cE^{(n),\ka}$ and $\wt X^{(n)}=(\wt X^{1},\dots,\wt X^{n})$ a diffusion on $\wt M^n$ associated to $\wt{\cE}^{(n)}$.
The processes $\pi^{(n)}(X^{(n),\ka})=(\pi(X^{1,\kappa}),\dots,\pi(X^{n,\kappa}))$ and $\wt X^{(n)}$ are random variables taking their values in $C(\mathbb{R}^+,\wt M^n)$ equipped with the topology of uniform convergence on compact sets. Theorem \ref{th:main} implies that
\begin{theorem}\label{th:main2}
Assume that 
\begin{itemize}
\item $\pi^{(n)}(X_0^{(n),\ka})$ converges in law to $\wt X_0^{(n)}$;
\item For all $\ka$, the law of $X_0^{(n),\ka}$ (resp. $\wt X_0^{(n)}$) has a density with respect to $m^{\otimes n}$ (resp. to $\wt m^{\otimes n}$) and this density belongs to $L^2(m^{\otimes n})$ (resp. to $L^2(\wt m^{\otimes n})$).  
\item $\sup_\ka \bbE\left[H^{(n)}(X_0^{(n),\ka})\right]<\infty$.
\end{itemize}
Then, $\big( \pi^{(n)}(X_t^{(n),\ka})\big)_{t\ge 0}$ converges in law to $(\wt X_t^{(n)})_{t\ge 0}$ as $\ka\to\infty$.
\end{theorem}

\subsection{Consistent and exchangeable family of Dirichlet forms}\label{sec:consistency}
For $n\ge 1$, let $S_n$ be the group of permutations on $\{1,2,\dots, n\}$. 
For $\sigma\in S_n$, by abuse of notation, we denote by $\sigma:M^n\to M^n$ the mapping defined by $\sigma(x)_i=x_{\sigma(i)}$ for all $i$.
Denote by $\Pi_n:M^{n+1}\to M^n$ the mapping defined by $\Pi_n(x_1,\dots,x_{n+1})=(x_1,\dots,x_n)$.

We will say that a family of regular Dirichlet forms $(\cE^{(n)},\cH^{(n)})$ on $L^2(m^{(n)})$ is \textit{consistent and exchangeable} if (setting $\mathcal{P}_{ac}^{(n)}$ the set of probability measures on $M^n$ absolutely continuous with respect to $m^{(n)}$)
\begin{enumerate}
\item For all $n\ge 1$ and $\sigma\in S_n$, $\sigma_* m^{(n)}=m^{(n)}$ and $\cE^{(n)}(f\circ \sigma,g\circ\sigma)=\cE^{(n)}(f,g)$ for all $f,g\in  \cH^{(n)}$.
\item For all $n\ge 1$ and $\mu_{n+1}\in \mathcal{P}_{ac}^{(n+1)}$, $\mu_n:=(\Pi_n)_*\mu_{n+1}\in \mathcal{P}_{ac}^{(n)}$ and if $X^{(n+1)}$ is a Hunt process associated to $(\cE^{(n+1)},\cH^{(n+1)})$ with $X^{(n+1)}_0$ distributed as $\mu_{n+1}$, then $X^{(n)}:=\Pi_n(X^{(n+1)})$ is a Hunt process associated to $(\cE^{(n)},\cH^{(n)})$ with $X^{(n)}_0$ distributed as $\mu_{n}$.
\end{enumerate} 

\begin{remark}
If for all $n\ge 1$, the semigroup associated to $(\cE^{(n)},\cH^{(n)})$ can be modified into a Feller semigroup $P^{(n)}$, then the family of regular Dirichlet forms $(\cE^{(n)},\cH^{(n)})$ on $L^2(m^{(n)})$ is consistent and exchangeable if and only if the family of Feller semigroups $(P^{(n)})_{n\ge 1}$ is consistent and exchangeable as defined in \cite{FCN}.
\end{remark}

By construction, the family of regular Dirichlet forms $(\cE^{(n)},\cH^{(n)})$ on $L^2(m^{\otimes n})$ defined in Section \ref{sec:DFn} is consistent and exchangeable. Moreover if $\Gamma$ is uniformly elliptic, this family is associated to a consistent and exchangeable family of Feller semigroups.

Theorem \ref{th:main2} implies that the family of regular Dirichlet forms $(\wt\cE^{(n)},\wt\cH^{(n)})$ on $L^2(\wt m^{\otimes n})$ is also consistent and exchangeable.

\subsection{Averaging of flows}\label{sec:avgflows}
Recall first the definition of a stochastic flow of kernels on $M$, a locally compact metric space,  as it is defined in \cite{FCN}:
\begin{definition}[Definition 2.3 in \cite{FCN}]\label{def:SFK}
On a probability space $(\Omega,\cA,\bbP)$, a family $(K_{s,t},s\le t)$ measurable mappings from $M\times \Omega$ onto $\cP(M)$ is called a (measurable) {\it stochastic flow of kernels  on $M$} (SFK) if for all $t\in\bbR$, $K_{t,t}(x,\omega)=\delta_x$ and if
\begin{itemize}
\item[(a)] For all $s< t< u$ and $x\in M$, $\bbP$-a.s. $K_{s,u}(x)=\int_M K_{s,t}(x,dy)K_{t,u}(y)$.
\item[(b)] For all $s<t$, $K_{s,t}$ and $K_{0,t-s}$ have the same law. (stationarity)
\item[(c)] For all $t_0<t_1<\cdots<t_n$, the family $\{K_{t_{i-1},t_i},1\le i\le n\}$ is independent. (independent increments)
\item[(d)] For all $f\in C_0(M)$, $(s,t,x)\mapsto K_{s,t}f(x)$ is continuous in $L^2(\bbP)$.
\item[(e)] For all $f\in C_0(M)$ and $s<t$, $K_{s,t}f(x)$ converges to $0$ in $L^2(\bbP)$ as $x\to\infty$.
\end{itemize}
\end{definition}
The mapping $K_{s,t}$ can be viewed as a random kernel, i.e. as a random variable $\omega \mapsto K_{s,t}(\cdot,\omega)$ taking its values in the space of measurable mappings from $M$ onto $\cP(M)$.

\medskip
Let us suppose in this section that for all $n\ge 1$ and $\kappa\in \mathbb{R}$ the Dirichlet forms $(\cE^{(n),\kappa},\cH^{(n)})$ and $(\wt \cE^{(n)},\wt \cH^{(n)})$ defined in Sections \ref{sec:DFn} and \ref{sec:DFtilden} are associated respectively with Feller semigroups we denote $P^{(n)}$ and $\wt P^{(n)}$. 
Then as is noticed in Section \ref{sec:consistency}, the families of Feller semigroups $(P^{(n),\kappa})_{n\ge 1}$ and $(\wt P^{(n)})_{n\ge 1}$ are consistent and exchangeable. 
By Theorem 2.1 in \cite{FCN}, for all $\kappa$, $(P^{(n),\kappa})_{n\ge 1}$ (respectively $(\wt P^{(n)})_{n\ge 1}$)
is associated to a (unique in law) stochastic flow of kernels $K^\kappa$ on $M$ (respectively $\wt K$ on $\wt M$) and  satisfying for all $n\ge 1$,
\begin{equation}
\mathbb{E}\big[(K_{0,t}^\kappa)^{\otimes n}\big] = P^{(n),\kappa}_t \qquad (\text{respectively } \mathbb{E}\big[(\wt K_{0,t})^{\otimes n}\big] = \wt P^{(n)}_t).
\end{equation}

Proposition 5.2 in \cite{FCN} shows that, for all $\kappa\in \mathbb{R}$, to the SFK $K^\kappa$, there is $W$ a vector field valued white noise of covariance $C$ such that $(K^\ka,W)$ is a solution of the $(A^\ka,C)$-SDE if 
for all $f\in C^2_c(M)$, $x\in M$ and $s\le t$,
\be \label{eq:(Akappa,C)-SDE}K^\ka_{s,t}f(x)=f(x)+\int_s^t K^\ka_{s,u}(Wf(du))(x) + \int_s^t K^\ka_{s,u}(A+\ka V)f(x) \,du. \ee
When $W=\sum_k W^kU_k$, the stochastic integral $\int_s^t K^\ka_{s,u}(Wf(du))(x)$ can be written in the more usual form $\sum_k \int_s^t K^\ka_{s,u}(U_kf)(x)W_k(du)$.
Since there is a pure diffusion, the SFKs $K^\ka$ are diffusive.

Let us suppose also that $\Gamma$ is elliptic and that $A$ maps $C^2_c(M)$ onto $C_c(M)$. These conditions ensures that for all $n\ge 1$, $\kappa\in \mathbb{R}$ and $x\in M^n$, the martingale problem associated to $A^{(n),\kappa}$ and $x\in M^n$ is well posed. This ensures that the $(A^\ka,C)$-SDE has a unique solution and that this solution is Wiener (i.e. for all $s\le t$, $\sigma(K^\ka_{u,v}, s\le u\le v\le t)\subset \sigma(W_{u,v}, s\le u\le v\le t)$).

Set $\cP_{ac}(M)$ the set of probability measures $\mu$ on $M$, absolutely continuous with respect to $m$ with a density in  $L^2(m)$ and such that $\mu H=\int_MH(x)\mu(dx)<+\infty$.

\begin{theorem}\label{th:mainflows}
As $\kappa\to\infty$, the family of SFKs $K^\kappa$ converges in distribution to $\wt K$ in the sense that for all $n\ge 1$, for all $\{(s_i,t_i),\mu_i, f_i:\;1\le i\le n\}$ with $s_i<t_i$, $\mu_i\in\mathcal{P}_{ac}(M)$ and $f_i\in  C_b(\wt M)$ for all $i$, 
\begin{equation}\label{eq:CVSFK}
\lim_{\kappa\to\infty} \mathbb{E} \left[ \prod_{i=1}^n \mu_i K^\kappa_{s_i,t_i} (f_i\circ \pi)\right]
= \mathbb{E}\left[ \prod_{i=1}^n \wt\mu_i \wt K_{s_i,t_i} f_i\right]
\end{equation}
where $\wt \mu_i:=\pi_*\mu_i$.
\end{theorem}
\begin{proof}
Using properties (a), (b) and (c) of a SFK, it is enough to prove \eqref{eq:CVSFK} only when $(s_i,t_i)=(0,t)$ for all $1\le i\le n$. Then,
\begin{align*}
\lim_{\kappa\to\infty} \mathbb{E} \left[ \prod_{i=1}^n \mu_i K^\kappa_{0,t} (f_i\circ \pi)\right]
&= \lim_{\kappa\to\infty} \mathbb{E} \left[ \prod_{i=1}^n \pi_* (\mu_i K^\kappa_{0,t}) f_i\right]\\
&=  \lim_{\kappa\to\infty} \mathbb{E} \left[ (\pi^{\otimes n})_* \left((\otimes_{i=1}^n \mu_i) (K^\kappa_{0,t})^{\otimes n}\right) (\otimes_{i=1}^n f_i)\right]\\
&=  \lim_{\kappa\to\infty}  (\pi^{\otimes n})_* \left((\otimes_{i=1}^n \mu_i) P^{(n),\kappa}_t\right) (\otimes_{i=1}^n f_i)\\
&=    (\otimes_{i=1}^n \wt \mu_i) \wt P^{(n)}_t  (\otimes_{i=1}^n f_i)\\
&= \mathbb{E}\left[ \prod_{i=1}^n \wt\mu_i \wt K_{0,t} f_i\right].
\end{align*}
\end{proof}

Let $\mu$ be a probability measure on $M$.
Then $\mu^\kappa_t=\mu K^\kappa_{0,t}$ is a  probability measure on $M$.
Set $\wt \mu=\pi_* \mu$, $\wt \mu^\kappa_t=\pi_*\mu^\kappa_t$ and $\wt \mu_t=\wt \mu \wt K_{0,t}$.
Theorem \ref{th:mainflows} implies that $\{\wt\mu^\kappa_t,\; t\ge 0\}$ converges in law in the sense of the finite distributions towards  $\{\wt\mu_t,\; t\ge 0\}$.

Suppose now that $\mu_t^\kappa$ is absolutely continuous with respect to $m$ and that $u^\kappa_t=\frac{d\mu_t}{dm}\in L^2(m)$. Then $u^\kappa$ is a weak solution (in $L^2(m)$) of the linear SPDE
\begin{align}\label{eq:SPDE}
du^\kappa_t
=&  A^\kappa u^\kappa_t dt - (\Div_m V_0) u^\kappa_t dt + Wu^\kappa_t(dt)  - (\Div_m W(dt)) u^\kappa_t,
\end{align} 
where $W=\sum_k U_k W^k$ and $\Div_m W=\sum_k \Div_m U_k W^k$.

Set $\wt u^\kappa_t:= \frac{d\wt \mu^\kappa_t}{d\wt m}\in L^2(\wt m)$ and $\wt u_t:= \frac{d\wt \mu_t}{d\wt m}\in L^2(\wt m)$. 
With these notations, Theorem \ref{th:mainflows} implies the convergence in law of $\wt u^\kappa$ towards $\wt u$ in the sense that for all $(t_i,f_i)_{1\le i\le n}$, with $t_i\ge 0$ and $f_i\in C_b(\wt M)$,
\begin{equation}\label{eq:CVSPDE}
\lim_{\kappa\to\infty} \mathbb{E} \left[ \prod_{i=1}^n \langle \wt u^\kappa_{t_i},f_i\rangle_{L^2(\wt m)}\right]
= \mathbb{E}\left[ \prod_{i=1}^n \langle \wt u_{t_i},f_i\rangle_{L^2(\wt m)}\right].
\end{equation}

The process $\wt u$ may be interpreted as a weak solution (in $L^2(\wt m)$) of a linear SPDE on $\wt M$ that will take a form similar to \eqref{eq:SPDE}. 
We will not do this in general but only on the examples in Section \ref{sec:explen=2} and \ref{sec:exampleR3}.
Our work is related to a recent work by Cerrai and Freidlin \cite{cerrai.freidlin19}, where the authors prove the convergence in $L^{p}$, $p\ge 1$, of $u^\kappa$ towards $\wt u$, where $u^\kappa$ and $\wt u$ are respectively solutions of a SPDE on $\mathbb{R}^2$ and on a metric graph. Note that the framework of \cite{cerrai.freidlin19} does not include our framework since in their SPDE, there is no term of the form $(Wu^\kappa_t(dt))$ (which corresponds to the transport of the particles by a vector field-valued white noise). 

\subsection{The Dirichlet form $(\wt\cE^{(n)},\wt \cH^{\otimes n})$.}
We have for $(f,g)\in \cH^{\otimes n}\times \cH^{\otimes n}$,
\begin{align*}
\cE^{(n)}(f,g)=\int \Gamma^{(n)}(f,g) \,dm^{\otimes n} - \langle V_0^{(n)}f, g \rangle_{L^2(m^{\otimes n})}
\end{align*}
with $\Gamma^{(n)}$ and $V_0^{(n)}$ defined by \eqref{eq:defGamma(n)} and \eqref{eq:defv0(n)}.
Define the bilinear form $\cE^{(n)}_i$ by
\begin{align*}
\cE_i^{(n)}(f,g)&:=\int \Gamma(d_if,d_ig) dm^{\otimes n} - \int (V_0(x_i),d_if(x)) g(x) dm^{\otimes n}.
\end{align*}
Note that the form $(\sum_{i=1}^n \cE^{(n)}_i,\cH^{\otimes n})$ is the Dirichlet form associated to $n$ independent Hunt processes associated to $(\cE,\cH)$. For this reason, this form will be denoted $(\cE^{\otimes n},\cH^{\otimes n})$.
Define also the bilinear form $\cE^{(n)}_C$ on $L^2(m^{\otimes n})$ by
\begin{align*}
\cE^{(n)}_C(f,g)&:= \sum_{i\ne j}\int C(d_if,d_jg) dm^{\otimes n} - \langle (V^{(n)}_C)f,g \rangle_{L^2(m^{\otimes n})}, 
\end{align*}
so that $\cE^{(n)}=\cE^{\otimes n} + \cE^{(n)}_C$.

By definition, for $(f,g)\in \wt \cH^{\otimes n}\times \wt \cH^{\otimes n}$,
$\wt \cE^{(n)}(f,g)=\cE^{(n)}(f\circ\pi^{\otimes n},g\circ\pi^{\otimes n})$.
We then have that $\wt \cE^{(n)}= \wt\cE^{\otimes n}+\wt \cE^{(n)}_C$, where $\wt\cE^{\otimes n}$ is the Dirichlet form associated to $n$ independent Hunt processes associated to $(\wt \cE,\wt \cH)$ and where 
\begin{align*}
\wt \cE^{(n)}_C(f,g)&= \cE^{(n)}_C(f\circ\pi^{\otimes n},g\circ\pi^{\otimes n}).
\end{align*}
Note that a case of interest would be when the covariance is $m$-divergent free, i.e. when the flow directed only by the Brownian vector field is $m$-incompressible (the measure $m$ is preserved by the flow). In this case, $\cE^{(n)}_C$ and $\wt \cE^{(n)}_C$ are symmetric and take a simpler form since $\delta_j C_{i,j}=0$.

For $f,g\in \wt \cH^{\otimes n}$, we define, for $y\in\wt M^n$ the different averaged quantities (using for a function $F$ on $M^n$, the notation $p^{\otimes n}(y) F=\int_{M^n} F(x) p(y_1,dx_1)\dots p(y_n,dx_n)$)
\begin{align*}
\wt \G^{(n)}(f,g)(y)&=p^{\otimes n}(y)\G^{(n)}(f\circ\pi^{\otimes n},g\circ\pi^{\otimes n}),\\
\wt \Gamma_i(f,g)&=p^{\otimes n}(y)\G(d_i(f\circ\pi^{\otimes n}),d_i(g\circ\pi^{\otimes n})),\\
\wt C_{ij}(f,g)(y)&=p^{\otimes n}(y)C(d_i(f\circ\pi^{\otimes n}),d_j(g\circ\pi^{\otimes n})), \text{ for $i\neq j$}.
\end{align*}
Define also, for $V^{(n)}$ a vector field on $M^n$ and $f\in \wt \cH^{\otimes n}$, 
$$\wt V^{(n)}f(y)=p^{\otimes n}(y)V^{(n)}(f\circ\pi^{\otimes n}).$$

Then for $f,g\in \wt \cH^{\otimes n}$, we have
\begin{align*}
\wt \cE^{(n)}_C(f,g)&:= \sum_{i\ne j}\int \wt C_{ij}(f,g) d\wt m^{\otimes n} -  \langle \wt V^{(n)}_Cf,g \rangle_{L^2(\wt m^{\otimes n})}.
\end{align*}

Also, we have
\begin{align*}
\wt \cE^{(n)}(f,g)&=\int \wt\Gamma^{(n)}(f,g) d\wt m^{\otimes n} - \langle \wt V_0^{(n)}f, g \rangle_{L^2(\wt m^{\otimes n})}
\end{align*}
and $\wt\Gamma^{(n)}(f,g)= \sum_{i=1}^n \wt \Gamma_i(f,g) +  \sum_{i\ne j} \wt C_{ij}(f,g)$.
Moreover, we get that the generator $\wt A^{(n)}$ is given on $\cD(\wt A)^{\otimes n}$ as
\begin{align}\label{eq:aaverage}
\wt A^{(n)}f(y)&=p^{\otimes n}(y)A^{(n)}(f\circ\pi^{\otimes n})
= \sum_{i=1}^n \wt A_if(y) + \sum_{i\ne j} \wt C_{ij} f (y)
\end{align}
where $\wt A_i f$ is such as $\wt A_if(y)=\wt Af_i(y_i) \prod_{k\ne i} f_k(y_k)$ when $f$ can be written in the form $f(y)=\prod_{k=1}^n f_k(y_k)$, and $\wt C_{ij} f(y)=p^{\otimes n}(y)C(d_id_j(f\circ\pi^{\otimes n}))$.

To go further, we need to explicitly compute the averaged quantities $\wt \G^{(n)}$, $\wt C$ and $\wt V_0^{(n)}$. We express them for the two examples of Section \ref{sec:hamilt} and Section \ref{sec:R3}.

\subsection{Example: random perturbations of Hamiltonian systems in $\mathbb{R}^2$.}\label{sec:explen=2}
In this subsection, we use the framework of Section \ref{sec:hamilt}. We suppose that the vector field $V$ satisfies Assumption \ref{hyp:mixing} and we let $C$ be a continuous covariance function on $M:=\mathbb{R}^2$ such that \eqref{u-pure diffusion} is satisfied for some $\delta>0$, i.e. there is a pure diffusion. We assume also \eqref{C:DF1} and \eqref{C:DF2}. 

Recall that $\wt M$ is a metric graph and that to every $y\in \wt M\setminus \mathcal{V}$ (with $\mathcal{V}$ the set of vertices) is associated a periodic orbit of the flow generated by $V$ and that its period is given by $T(y)$.

\begin{remark}
Remark \ref{rem:periodic} entails that Assumption \ref{hyp:mixing} is satisfied if $T$ is $C^1$ on each edge and that the set $\{y\in\wt M\setminus\cV, T'(y)=0\}$ is $\wt m$-negligible (in order to avoid the situation of Example \ref{ex:periodic}).
\end{remark}

In Section \ref{sec:hamilt}, we have already given the Dirichlet form $(\wt \cE,\cH)$. 
Let us now describe the Dirichlet form $(\wt\cE^{(n)},\wt\cH^{\otimes n})$. Recall that $\wt \cE^{(n)}=\wt \cE^{\otimes n}+\wt \cE^{(n)}_C$.
We have (using the notation $x=(x_1,\dots,x_n)\in \big(\mathbb{R}^2\big)^n$ and $x_i=(x_i^1,x_i^2)$ for $1\le i\le n$)
\begin{align*}
&C(d_if(x),d_jg(x))=\sum_{k,\ell}
\frac{\partial f}{\partial x_i^k}(x) \, C^{k\ell}(x_i,x_j)\, \frac{\partial g}{\partial x_j^\ell}(x),\\
&(\delta_j C_{i,j})f(x)=\sum_{k,\ell}\frac{\partial f}{\partial x_i^k}(x)\,\frac{\partial }{\partial x_j^\ell}C^{k\ell}(x_i,x_j).
\end{align*}

Recall that $m$ is the Lebesgue measure and that $\wt m(dy)=T(y)dy$ on each edge.
As in Section \ref{sec:desc}, we define the averaged quantities.
For $y\in \wt M$, recall from Section \ref{sec:desc} that $\wt \G(y):=\frac12\sigma^2(y)= p(y) \G(H,H)$, $c(y)=\wt V_0(y):=p(y)(V_0H)$.
For $(y_1,y_2)\in \wt M^2$, set
\begin{align*}
\wt C(y_1,y_2):=& (p(y_1)\otimes p(y_2)) (C(H,H)),\\
\wt {\delta C}(y_1,y_2):=& \int \left(\sum_{k,\ell}\partial_k H(x_1)\frac{\partial }{\partial x_2^\ell}C^{k\ell}(x_1,x_2)\right) p(y_1,dx_1) p(y_2,dx_2).
\end{align*}

Then we have
\begin{align*}
\wt \cE^{(n)}_C(f,g) =&\sum_{i\ne j} \int \wt C(y_i,y_j) \partial_i f(y)\partial_j g(y) \wt m^{\otimes n}(dy)\\
& +  \sum_{i\ne j}\int \wt {\delta C}(y_i,y_j) \partial_i f(y) g(y) \wt m^{\otimes n}(dy).
\end{align*}
and 
\begin{align*}
\wt \cE^{(n)}(f,g) =&\sum_i \int \wt \G(y_i) \partial_i f(y)\partial_i g(y) \wt m^{\otimes n}(dy)+\sum_{i\ne j} \int \wt C(y_i,y_j) \partial_i f(y)\partial_j g(y) \wt m^{\otimes n}(dy)\\
& +\sum_i \int \wt V_0(y_i) \partial_i f(y) g(y) \wt m^{\otimes n}(dy)+  \sum_{i\ne j}\int \wt {\delta C}(y_i,y_j) \partial_i f(y) g(y) \wt m^{\otimes n}(dy).
\end{align*}

The generator $\wt A^{(n)}$ of $(\wt \cE^{(n)},\wt \cH^{\otimes n})$ is given for $f\in \mathcal{D}(\wt A^{(n)})=\mathcal{D}(\wt A)^{\otimes n}$ by
\begin{align*}
\wt A^{(n)} f(y) = \sum_{i=1}^n \wt A_i f(y) + \sum_{i\ne j} \wt C(y_i,y_j) \partial_i \partial_j f (y),
\end{align*}
with $\wt A_i f$ such as $\wt A_if(y)=\wt Af_i(y_i) \prod_{k\ne i} f_k(y_k)$ when $f$ can be written in the form $f(y)=\prod_{k=1}^n f_k(y_k)$ and $\wt A$ is given in Proposition \ref{prop:atilde}.

For all $n\geq 1$, $(\wt \cE^{(n)},\wt\cH^{\otimes n})$ is associated to a Hunt semigroup $\wt P^{(n)}$ on the metric graph $\wt M$. 
Admitting that these semigroups are Fellerian, the family of Dirichlet forms $(\wt\cE^{(n)})_n$ is associated to a SFK $\wt K$ on $\wt M$. 
We didn't find a simple argument to prove this Feller property. We postpone the proof of such Feller property to a future paper.


In Section \ref{sec:avgflows} we have seen that that for all $\kappa\in \mathbb{R}$, to the SFK $K^\kappa$ on $\mathbb{R}^2$ associated to the family of Dirichlet forms $(\cE^{(n),\kappa})_n$, there is $W$ a vector field valued white noise of covariance $C$ such that $(K^\ka,W)$ is a solution of the $(A^\ka,C)$-SDE.

Formally, the SFK $\wt K$ solves a SDE: There is $\wt W$ a vector field valued white noise of covariance $\wt C$ and for all $f\in \cD(\wt A)$, $y\in \wt M$ and $s\le t$,
\be \label{eq:(tildeA,C)-SDE}\wt K_{s,t}f(y)=f(y)+\int_s^t \wt K_{s,u}(\wt Wf(d u))(y) + \int_s^t \wt K_{s,u}\wt Af(y) \,du, \ee
with $\wt Wf(y)(du)=\sum_k \wt U_k(y) f'(y) dW^k(u)$ and where $W^1,W^2,\cdots$ are independent white noises and $\wt U_k(y)=p(y)(U_k H)$, i.e. $\wt W$ is a vector field valued white noise on $\wt M$ of covariance $\wt C$, since $\sum_k \wt U_k(y_1)\wt U_k(y_2)=\wt C(y_1,y_2)$.

In other words, the $n$-point motion $(Y^1,\dots,Y^n_t)$ of $\wt K$ are $n$ correlated diffusions on $\wt M$ solution of the SDE:
$$dY^i_t = \wt \s(Y^i_t) dB^i_t + \sum_k \wt U_k(Y^i_t)dW^k_t +b(Y^i_t) dt,$$
where $B^1,\dots, B^n, W^1,W^2,\dots$ are independent Brownian motions and $\frac12 \wt \s(y)^2=\wt \G(y)-\wt C(y,y)$, and $b$ is given Proposition \ref{prop:atilde} as $b(y)=\frac{1}{T(y)}(\wt \G T)'(y)+\wt V_0(y)$.
The particle $Y^i$ being reflected at each vertex $v$ on each edge $E$ adjacent to this vertex, with transmission parameters $\alpha_k^{\pm}$, $k\in I_v^{\pm}$.
Note that $\wt \s$ is well defined since we have assumed the uniformly pure diffusion (Equation \eqref{u-pure diffusion}).
 
The process $\wt u$ defined at the end of Section \ref{sec:avgflows} is a weak solution (in $L^2(\wt m)$) of the linear SPDE on $\wt M$
\begin{align}\label{eq:SPDEtilde}
d\wt u_t
=&  \wt A \wt u_t dt - (\Div_{\wt m} \wt V_0) \wt u_t dt + \wt W \wt u_t(dt) - (\Div_{\wt m} \wt W(dt)) \wt u_t,
\end{align} 
with $\wt W=\sum_k \wt U_k W^k$, $\Div_{\wt m} \wt W=\sum_k \Div_{\wt m}( \wt U_k) W^k$, and for a vector field $\wt U$ on $\wt M$, $\Div_{\wt m}(\wt U)(y)={T}^{-1}(y)(T\wt U)'$.

\subsection{Example: on $\mathbb{R}^3$.}\label{sec:exampleR3}
In this subsection, we use the framework of Section \ref{sec:R3}. We let $C$ be a continuous covariance function on $M=\mathbb{R}^3$ such that \eqref{u-pure diffusion}, \eqref{C:DF1} and \eqref{C:DF2} are satisfied.
Recall that $\wt M=\cup_iC_i$ is a gluing of four leaves along the half-line $D$. 

\begin{lemma}\label{lem:mixingR3}
Assumption \ref{hyp:mixing} is satisfied.
\end{lemma}
\begin{proof}
For $i\in I$ and $y\in\ring{C}_i$, the period associated to the orbit $\gamma_y$ (see Section \ref{sec:6.descriptionE}) is 
$$
T(y)=\oint_{\gamma_y}\frac{d\ell}{\|V\|}
=\int_0^{2\pi}\frac{\|\del_{\theta}\Phi_i(y,\theta)\|}{\|V(\Phi_i(y,\theta))\|}d\theta
=\frac{h_i(y)}{2y_1y_2}.
$$
For example, when $y\in \ring{C}_1$, $T(y)=\frac{2\sqrt{2}}{y_1}K\left(\frac{y_2}{y_1}\right)$. Thus, for all $y\in \cup_{i\in I} \ring{C}_i$, we have that $\nabla T(y)\ne 0$. This entails that Assumption \ref{hyp:mixing} is satisfied.
\end{proof}

In Section \ref{sec:R3}, we have given the Dirichlet forms $\cE$ and $\wt \cE$. We now describe the Dirichlet form $(\wt \cE^{(n)},\wt\cH^{\otimes n})$.

We have (using the notation $x=(x_1,\dots,x_n)\in \big(\mathbb{R}^3\big)^n$ 	and $x_i=(x_i^1,x_i^2,x_i^3)$ for $1\le i\le n$)
\begin{align*}
&C(d_if(x),d_jg(x))=\sum_{k,\ell}\frac{\partial f}{\partial x_i^k}(x) \, C^{k\ell}(x_i,x_j)\, \frac{\partial g}{\partial x_j^\ell}(x)\\
&(\delta_j C_{i,j})f(x)=\sum_{k}\frac{\partial f}{\partial x_i^k}(x)\,(\delta_2C)^{k}(x_i,x_j),
\end{align*}
where $(\delta_2C)^{k}(x_1,x_2)=e^{\cW(x_2)} \sum_\ell
\frac{\partial }{\partial x_2^\ell}\left(e^{-\cW(x_2)}C^{k\ell}(x_1,x_2)\right)$.

Recall that in Section \ref{sec:R3}, we have considered $\G(f)=|\nabla f|^2$ and $V_0=0$.
For $(y_1,y_2)\in \wt M^2$ and $r,s\in\{1,2\}$, from Section \ref{sec:desc2}, we have $\wt \G^{rs}(y_1):=a^{rs}(y_1)= p(y_1) (\G(\pi^r,\pi^s))$ and set
\begin{align*}
\wt C^{rs}(y_1,y_2):
&=(p(y_1)\otimes p(y_2)) (C(\pi^r,\pi^s))\\
\wt {\delta C}^r(y_1,y_2):&= \int \left(\sum_{k}\partial_k \pi^r(x_1)(\delta_2C)^{k}(x_1,x_2)\right) p(y_1,dx_1) p(y_2,dx_2).
\end{align*}
Note that $\wt {\delta C}=0$ when $\delta_2C=0$, i.e. when $C$ is $m$-divergent free.

We then have:
\begin{align*}
\wt \cE^{(n)}_C(f,g) =&\sum_{i\ne j} \int \sum_{r,s}\wt C^{rs}(y_i,y_j) \,\frac{\partial f }{\partial y_i^r}(y)\frac{\partial g }{\partial y_j^s}(y)\, \wt m^{\otimes n}(dy)\\
& +  \sum_{i\ne j}\int \sum_{r}\wt {\delta C}^r(y_i,y_j)  \frac{\partial f }{\partial y_i^r}(y) g(y) \wt m^{\otimes n}(dy).
\end{align*}
and 
\begin{align*}
\wt \cE^{(n)}(f,g) =&\sum_{i} \int \sum_{r,s}\wt\G^{rs}(y_i) \,\frac{\partial f }{\partial y_i^r}(y)\frac{\partial g }{\partial y_j^s}(y)\, \wt m^{\otimes n}(dy)\\
&+\sum_{i\ne j} \int \sum_{r,s}\wt C^{rs}(y_i,y_j) \,\frac{\partial f }{\partial y_i^r}(y)\frac{\partial g }{\partial y_j^s}(y)\, \wt m^{\otimes n}(dy)\\
&+  \sum_{i\ne j}\int \sum_{r}\wt {\delta C}^r(y_i,y_j)  \frac{\partial f }{\partial y_i^r}(y) g(y) \wt m^{\otimes n}(dy).
\end{align*} 

We obtain that the generator $\wt A^{(n)}$ of $(\wt \cE^{(n)},\wt \cH^{\otimes n})$ is given for $f\in \mathcal{D}(\wt A^{(n)})=\mathcal{D}(\wt A)^{\otimes n}$ by
\begin{align*}
\wt A^{(n)} f(y) = \sum_{i=1}^n \wt A_i f(y) + \sum_{i\ne j} \wt C (d_i d_j f) (y),
\end{align*}
with $\wt A_i f$ such as $\wt A_if(y)=\wt Af_i(y_i) \prod_{k\ne i} f_k(y_k)$ when $f$ can be written in the form $f(y)=\prod_{k=1}^n f_k(y_k)$, $\wt A$ is given in Proposition \ref{prop:atilde2}, and
$$
\wt C (d_i d_j f) (y)=\sum_{r,s}\wt C^{rs}(y_i,y_j) \left(\frac{\partial^2 f}{\partial y_i^r\partial y_j^s}\right)(y).
$$

Assume that for all $n\geq 1$, $(\wt \cE^{(n)},\wt\cH^{\otimes n})$ is associated to a Feller semigroup $\wt P^{(n)}$ on  $\wt M$. 
Then, the family of Dirichlet forms $(\wt\cE^{(n)})_n$ is associated to a SFK $\wt K$ on $\wt M$.

This SFK $\wt K$ will also solve a SDE of the form given by Equation \eqref{eq:(tildeA,C)-SDE}
with $\wt Wf(y)(du)=\sum_k (\wt U_k f)(y) W^k(du)$ with $\wt U^r_k=p(\cdot)(U_k\pi^r)$ and where $W^1,W^2,\cdots$ are independent white noises.

In other words, the $n$-point motion $(Y^1,\dots,Y^n_t)$ of the SFK $\wt K$ are $n$ correlated diffusions on $\wt M$ and is solution of the SDE:
$$dY^i_t = \wt \s(Y^i_t) dB^i_t + \sum_k \wt U_k(Y^i_t)dW^k_t + b(Y^i_t) dt,$$
where $B^1,\cdots, B^n, W^1,W^2,\cdots$ are independent Brownian motions and $\frac12\wt \sigma^T\wt \sigma(y)=\wt \G(y)-\wt C(y,y)$ and $b$ is given in Equation \eqref{eq:driftR3}. The particle $Y^i$ being reflected at each point $y\in D$ on each leaf $C_i$ uniformly at random. Note again that due to the uniformly pure diffusion assumption (Equation \eqref{u-pure diffusion}), $\wt \s$ is well-defined.

The process $\wt u$ defined at the end of Section \ref{sec:avgflows} is a weak solution (in $L^2(\wt m)$) of the linear SPDE on $\wt M$
\begin{align}\label{eq:SPDEtilde2}
d\wt u_t
=&  \wt A \wt u_t dt - (\Div_{\wt m} \wt V_0) \wt u_t dt + \wt W \wt u_t(dt) - (\Div_{\wt m} \wt W(dt)) \wt u_t,
\end{align} 
with $\wt W=\sum_k \wt U_k W^k$,  $\Div_{\wt m} \wt W=\sum_k \Div_{\wt m}( \wt U_k) W^k$
and for $\wt V$ a $C^1$-vector field on $\wt M$ (i.e. the restriction of $\wt V$ to a domain $\ring{C}_\ell$ is a $C^1$-vector field on this domain) and for $y\in \cup_\ell{\ring{C}_\ell}$, we have $\Div_{\wt m}(\wt V)(y)={\wt m}^{-1}(y)\sum_r\frac{\partial}{\partial y^r}(\wt m(y)\wt V^r(y))$.

\begin{remark}
If $U$ is a $C^1$ vector field on $M$, then set the ``vector field'' $\wt U$ on $\wt M$, $\wt U^r:=p(\cdot)(U\pi^r)$. On an open subset of $\wt M\setminus D$, we have that $\wt U$ is $C^1$ and 
\begin{equation}\label{eq:divavg}
p(\cdot)(\Div_m U)=\Div_{\wt m}\wt U.
\end{equation}
This property is the key to do an integration by parts on $\wt\cE^{(n)}$ to recover $\wt A^{(n)}$ directly. Here, we have just described $\wt A^{(n)}$ from $A^{(n)}$ using Equation \eqref{eq:aaverage}.

To prove Equation \eqref{eq:divavg}, let $f\in C^{\infty}_c(C_i)$ and prove that $\int U(f\circ \pi)d m=\int \wt Uf d\wt m$ directly using the chain rule on $U(f\circ\pi)$, and integrate by parts each side.
\end{remark}

\section{Appendix}
Notation: In a topological space $A\Subset B$ means that $\overline{A}\subset \mathring{B}$ and $\overline{A}$ is compact.
\medskip

\subsection{Weighted Sobolev spaces.}
For $k\ge 1$, the class of Muckenhoupt weights $A_2$ consists all mappings $\omega:\mathbb{R}^k\to [0,\infty]$, for which there is a constant $C$ such that for all ball $B$ in $\mathbb{R}^k$, we have 
$$\frac{1}{|B|}\left(\int_B \omega(x) dx\right)\times \frac{1}{|B|}\left(\int_B \frac{1}{\omega(x)} dx\right) \le C.$$

\medskip
For $m\ge 2$, set $B_{m,1}:=\left(\cup_{i=1}^m ]0,\infty[\times \{i\}\right)\cup \{0\}$, equipped with the distance $d_1$ defined by $d_1((x,i),(y,i))=|x-y|$ and if $i\ne j$,  $d_1((x,i),(y,j))=x+y$, and $d_1((x,i),0)=x$. 
Let $i_1:B_{m,1}\to \{0,\dots,m\}$ be defined $i_1(x,i)=i$ and $i_1(0)=0$.
Then $B_{m,1}$ is a metric graph constituted of $m$ half lines joined at $0$. 
Set also $B_{m,2}:=\mathbb{R}\times B_{m,1}$, equipped with the distance $d_2$ defined by $d_2((x,i),(y,i))=\|x-y\|$, $d_2((x,i),(y,j))=\|x-y'\|$ if $i\ne j$ and where $y'=(y_1,-y_2)$, and $d_2((x,i),y)=\|x-y\|$ if $y=(y_1,0)$, where $\|\cdot\|$ is the Euclidean norm on $\mathbb{R}^2$.
Then $B_{m,2}$ is a metric space constituted of $m$ half planes joined along a line.
Let also $i_2: B_{m,2}\to \{0,\dots,m\}$ be defined by $i_2(x_1,x_2,i)=i$ and $i_2(x_1,0)=0$.
To simplify the notation we will simply denote $d_1$ and $d_2$ by $d$. For $1\le i\le m$, we will use the notation $(0,i)=0\in B_{m,1}$ and for $x_1\in\mathbb{R}$, $(x_1,0,i)=(x_1,0)\in B_{m,2}$.

For $m=1$ and $k\in \{1,2\}$, set $B_{1,k}=\mathbb{R}^k$.

Let $\omega: B_{m,k}\to [0,\infty]$ be such that $\omega\in L^1_{loc}(B_{m,k})$, i.e. such that for all $A\Subset \Omega$, $\mu(A):= \int_A \omega(x) dx < \infty$,  with $dx$ the measure on $B_{m,k}$ that coincides with the Lebesgue measure on $E_i:=\{x\in B_{m,k}: i_k(x)=i\}$ for each $i$. For $i\ne j$, set $E_{i,j}:=E_i\cup E_j\cup E_0$ (which is isometric to $\mathbb{R}^k$, and will be thus identified to $\mathbb{R}^k$).

\medskip
In the following, we fix $k\in\{1,2\}$ and $m\ge 1$ and we let $\Omega$ be an open subset of $B_{m,k}$.
For $i\in \{1,\cdots,m\}$, set $\Omega_i=\Omega\cap E_i$ and for $1\le i\ne j\le m$, set $\Omega_{i,j}=\Omega\cap E_{i,j}$. Then $\Omega_i$ and $\Omega_{i,j}$ are open subsets of $\mathbb{R}^k$.
Denote by $L^2(\Omega,\omega)$ the space of all measurable functions $f$ on $B_{m,k}$ such that $\int_\Omega f^2(x) \omega(x) dx < \infty$.
For a function $f\in L^2(\Omega,\omega)$, weakly differentiable on $\Omega_{i,j}$ for all $1\le i\ne j\le m$, define the norm of $f$ by
\begin{equation}
\label{eq:normw1w}
\|f\|^2_{W^1(\Omega,\omega)}=\int_\Omega \big( f^2 + \|\nabla f\|^2\big)(x) \omega(x) dx.
\end{equation}

Let $W^1(\Omega,\omega)$ be the completion with respect to the norm $\|\cdot\|_{W^1(\Omega,\omega)}$ of the vector space of the functions $f\in L^2(\Omega,\omega)$, that are weakly differentiable on $\Omega_{i,j}$ for all $1\le i\ne j\le m$ and with $\|f\|_{W^1(\Omega,\omega)}<\infty$.
Suppose also that $\frac{1}{\omega}\in L^1_{loc}(\Omega)$.
Then, for all $i\in\{1,\dots,k\}$ (resp. all $1\le i\ne j\le k$), the restriction of $f\in W^1(\Omega,\omega)$ to $\Omega_i$ (resp. to $\Omega_{i,j}$) belongs to $W^{1,1}_{loc}(\Omega_i)$ (resp. to $W^{1,1}_{loc}(\Omega_{i,j})$.

We also define $H^1(\Omega,\omega)$ (resp. $H^1_0(\Omega,\omega)$) to be the completion of $C(\Omega)\cap W^1(\Omega,\omega)$ (resp. of $C_c(\Omega)\cap W^1(\Omega,\omega)$), with respect to $\|\cdot\|_{W^1(\Omega,\omega)}$.
Define also $W^1_0(\Omega,\omega)$ to be the set of all $f\in W^1(\Omega,\omega)$ such that the function $F=f1_\Omega\in W^1(B_{m,k},\omega)$.
Equipped with the inner product
\begin{equation}
\langle f,g\rangle_{W^1(\Omega,\omega)}=\int_\Omega (fg+\nabla f\cdot\nabla g)(x)\omega(x)dx,
\end{equation}
$W^1(\Omega,\omega)$, $W^1_0(\Omega,\omega)$, $H^1(\Omega,\omega)$ and $H^1_0(\Omega,\omega)$ are Hilbert spaces.

\begin{lemma}\label{lem:sob1}
Let $m\ge 1$ and $O\Subset \Omega$ be open subsets of $B_{m,k}$.
Then there is $\delta>0$ such that $K_\delta:=\{x\in B_{m,k}: d(x,O)\le \delta\}$ is a compact set with $O\subset {K}_\delta\subset \Omega$.
Suppose that there is $\bar{\omega}:B_{m,k}\to [0,\infty]$ such that $\bar{\omega}=\omega$ on $\Omega$ and such that 
\begin{itemize}
\item when $m\ge 2$, for all $i\ne j$, the restriction  $\bar{\omega}_{i,j}$ of $\bar{\omega}$ to $E_{i,j}$ belongs to the class $A_2$.
\item when $m=1$, $\bar{\omega}$ belongs to the class $A_2$.
\end{itemize}
Then if $f\in W^1_0(O,\omega)$, there is $g\in L^2(O,\omega)$ such that $g=0$ on $O\setminus K_\delta$ and such that for all $(x,y)\in O^2$, 
$$|f(x)-f(y)|\le d(x,y)\big( g(x)+g(y) \big).$$ 
Moreover, there is a sequence of lipschitzian functions $f_n\in  W^1_0(O,\omega)$ such that $\lim_{n\to\infty}\|f-f_n\|_{W^1(O,\omega)}=0$.
\end{lemma}
\begin{proof}
We only consider the case $m\ge 2$, the case $m=1$ being simpler.

The existence of $\delta$ and $K_\delta$ is a standard exercise.

Let $f\in W^1_0(O,\omega)$ and denote by $f_{i,j}$ the restriction $f$ to $E_{i,j}$, and set $O_{i,j}:=O\cap E_{i,j}$.
Then $f_{i,j}\in W^1_0(O_{i,j},\bar{\omega}_{i,j})$. 
Set $F:=f1_O$ and $F_{i,j}$ the restriction of $F$ on $E_{i,j}$. Then $F_{i,j}\in W^1(E_{i,j},\bar{\omega}_{i,j})$.
Recall that $E_{i,j}$ is isometric to $\mathbb{R}^k$.
Note that for $x\in O_i=O\cap E_i$, we have for all $j\ne i$, $f(x)=f_{i,j}(x)=F_{i,j}(x)$.

Since $\bar{\omega}\in A_2$, we have (see \cite{hajlasz}) $G_{i,j}\in L^2(E_{i,j},\bar{\omega}_{i,j})$ such that for all $(x,y)\in E_{i,j}^2$,
$$|F_{i,j}(x)-F_{i,j}(y)|\le d(x,y)\big(G_{i,j}(x)+G_{i,j}(y)\big).$$
For $x\in \Omega$, define $g(x):=\sum_{j\ne i} G_{i,j}(x)$ if $x\in \Omega_i$. Then $g\in L^2(\Omega,\omega)$, and we have for all $(x,y)\in \Omega^2$, 
$$|f(x)-f(y)|\le d(x,y)\big(g(x)+g(y)\big).$$

Set now $g_\delta:=g 1_{K_\delta}+\delta^{-1}|f|$. 
Then, we have that for all $(x,y)\in \Omega$, 
$$|f(x)-f(y)|\le d(x,y)\big(g_\delta(x)+g_\delta(y)\big).$$
Indeed, this inequality is straightforward to check if $(x,y)\in K_\delta^2$ or if $(x,y)\in (\Omega\setminus O)^2$. 
If $(x,y)\in O\times  (\Omega\setminus K_\delta)$, we have   
$d(x,y)\big(g_\delta(x)+g_\delta(y)\big)\ge d(x,y) \delta^{-1}|f(x)|\ge |f(x)-f(y)|.$
This shows the first part of the lemma.

For the second part, it suffices to follow the proof of Theorem 5 of \cite{hajlasz} (with the function $g_\delta$, and since $f=f_\lambda$ on $E_\lambda$ and that $E_\lambda\supset \Omega\setminus K_\delta$, we have $f_\lambda=0$ sur $\Omega\setminus K_\delta$).
\end{proof}

For an open set $U\subset \Omega$, define the capacity of $U$ by
$$\mathrm{Cap}(U):= \inf\{\|h\|_{W^1(\Omega,\omega)}^2:\; h\ge 1 \text{ on $ U$ and } h\in   W^1(\Omega,\omega)\}.$$

\begin{lemma}\label{lem:sob2}
Let $m\ge 1$, $\Omega$ an open subset of $B_{m,k}$ and $\omega:\Omega\to [0,\infty]$ a measurable mapping.
For $R>0$, set $\Omega_R:=\{x\in \Omega: \|x\| <R\}$.
Suppose that, for all $R>0$, there exists a non decreasing sequence of open subsets $(\Omega_{R,n})_{n\ge 1}$ such that
\begin{enumerate}[(i)]
\item $\cup_{n\ge 1}\Omega_{R,n}=\Omega_R$,
\item $\lim_{n\to\infty} \mathrm{Cap}(\Omega_R\setminus \overline{\Omega}_{R,n})=0$,
\item for all $n\ge 1$, $W^1_0(\Omega_{R,n},\omega)=H^1_0(\Omega_{R,n},\omega)$.
\end{enumerate} 
Then we have $W^1(\Omega,\omega)=H^1_0(\Omega,\omega)$.
\end{lemma}
\begin{proof}
Let $f\in W^1(\Omega,\omega)$ and $\epsilon>0$. For $R>1$, define $f_R:\Omega\to \mathbb{R}$ defined by $f_R(x)=f(x)$ if $\|x\|\le R-1$, $f_R(x)=(R-\|x\|) f(x)$ if $\|x\|\in [R-1,R]$ and $f_R(x)=0$ if $\|x\|>R$. Then we have that $f_R\in W^1_0(\Omega_R,\omega)$ and it is easy to check that there is an $R_0>0$ such that for all $R>R_0$,  $\|f-f_R\|_{W^1(\Omega,\omega)}<\epsilon$.

From now on, we fix $R>R_0$.
Applying Theorem III.2.11 in \cite{ma.rockner92}, conditions $(i)$ and $(ii)$ ensures that there is a sequence of functions $f_{R,n}\in W^1_0(\Omega_{R,n},\omega)$ such that $\lim_{n\to\infty}\|f_R-f_{R,n}\|_{W^1(\Omega_R,\omega)}=0$.

And we conclude using $(iii)$.
\end{proof}

\medskip
Let now $G$ be a connected metric space, and fix $k\in \{1,2\}$. 
Suppose that there is a locally finite covering $(\Omega_\ell)_{\ell\in \mathcal{L}}$ of $G$, with open sets and and with $\mathcal{L}$ a countable set. Suppose also that this covering is such that for each $\ell$, $\Omega_\ell$ is isometric to an open subset of $B_{m,k}$ for some $m\ge 1$.
Suppose that there is a sequence of functions $(\varphi_\ell)_{\ell\in \mathcal{L}}$ such that $\varphi_\ell: G\to [0,1]$, $\varphi_\ell=0$ on $G\setminus \Omega_\ell$, $\varphi_\ell$ restricted to $\Omega_\ell^{i}$ is $C^1$, with bounded derivatives, for each $i\in \{1,\dots, m_\ell\}$ and such that $\sum_{\ell\in\mathcal{L}}\varphi_\ell=1$.

Let $\omega:G\to [0,\infty]$ be a measurable mapping.
For each $\ell$, then there is $m\ge 1$ such that $\Omega_\ell \subset B_{m,k}$, we let $\omega_\ell:B_{m,k}\to [0,\infty]$ be such that $\omega_\ell=\omega$ on $\Omega_\ell$. 
Suppose that $\omega_\ell\in L^1_{loc}(B_{m,k})$ 
and $\frac{1}{\omega_\ell}\in L^1_{loc}(\Omega_\ell)$.

Define $W^1(G,\omega)$ to be the set of all measurable functions $f:G\to\mathbb{R}$ such that for each $\ell$ there is $f_\ell\in W^1(\Omega_\ell,\omega_\ell)$ with $f=f_\ell$ on $\Omega_\ell$ with $\|f\|_{W^1(G,\omega)}<\infty$ where
$$\|f\|^2_{W^1(G,\omega)}=\int_G \big( f^2 + \|\nabla f\|^2\big)(x) \omega(x) dx.$$
Then $W^1(G,\omega)$ equipped with the innner product
$$\langle f,g\rangle^2_{W^1(G,\omega)}=\int_G \big( fg + \nabla f\cdot \nabla g \big)(x) \omega(x) dx$$
is a Hilbert space. Define also $H^1_0(G,\omega)$ as the completion of $C_c(G)\cap W^1(G,\omega)$.
Note that if $f\in W^1(G,\omega)$, we have that for each $\ell$, $f\varphi_\ell\in W^1(\Omega_\ell,\omega_\ell)$.

\begin{lemma}\label{lem:sob3}
Suppose that for all $\ell$ and all $R>0$  there exists a non decreasing sequence of open subsets $(\Omega_{\ell,R,n})_{n\ge 1}$ such that
\begin{enumerate}[(i)]
\item $\bigcup_{n\ge 1}\Omega_{\ell,R,n}=\Omega_{\ell,R}$,
\item $\lim_{n\to\infty} \mathrm{Cap}(\Omega_{\ell,R}\setminus \overline{\Omega}_{\ell,R,n})=0$,
\item for all $n\ge 1$, $W^1_0(\Omega_{\ell,R,n},\omega_\ell)=H^1_0(\Omega_{\ell,R,n},\omega_\ell)$.
\end{enumerate} 
Then we have $W^1(G,\omega)=H^1_0(G,\omega)$.
\end{lemma}
\begin{proof}
The proof of this lemma is almost identical to the one of Lemma \ref{lem:sob2}. We write $f=\sum_\ell f_\ell$, where $f_\ell:=f\varphi_\ell$. For each $\epsilon>0$ there is a finite $\mathcal{L}_0\subset\mathcal L$ such that $\|f-\sum_{\ell\in \mathcal{L}_0}f_\ell\|_{W^1(G,\omega)}<\epsilon$. And then we can follow the proof of Lemma \ref{lem:sob2}, for $\Omega=\Omega_\ell$ and $f=f_\ell$ and to find, for all $\ell\in\mathcal{L}_0$, an integer $n\ge 1$ and a function $g_\ell\in C_c(\Omega_{\ell,R,n})\cap W^1(\Omega_\ell,\omega_\ell)$ such that $\|f_\ell -g_\ell\|_{W^1(G,\omega)}<\frac{\epsilon}{|\mathcal{L}_0|}$. Then 
$\|f -\sum_{\ell\in\mathcal{L}_0} g_\ell\|_{W^1(G,\omega)}<2  \epsilon$.
\end{proof}

\subsection{Regularity for Section \ref{sec:hamilt}}
In this paragraph, we complete the proofs of Proposition \ref{prop:Htilde5} and  of Proposition \ref{prop:Ctilde1}.

Let us first recall some notation of Section \ref{sec:hamilt}. The space $\widetilde M=V\cup\bigcup_{i\in I} E_i$ is a metric graph, with $E_i=]l_i,r_i[\times \{i\}$. There is a continuous map $\pi:\mathbb{R}^2\to \widetilde M$ such that if $x\in \Omega_i$, $\pi(x)=(H(x),i)\in E_i$.
For $v\in V$, $I^+_v=\{i\in I, (l_i,i)=v\}$, $I^-_v=\{i\in I, (r_i,i)= v\}$ and $I_v=I^+_v\cup I^-_v$. Let $d(v)$ be the cardinal of the set $I_v$, which is the degree of $v$ in the graph $\widetilde M$. 
Choose $i\in I_v$ and set $h_v=l_i$ if $v=(l_i,i)$ or $h_v=r_i$ if $v=(r_i,i)$. Remark that $h_v$ does not depend on the particular choice $i\in I_v$. For $v\in V$, we define $\gamma (v):=\bigcup_{(h,i)\sim v}\gamma_i(h)$, the connected level set associated to the vertex $v$. Then for all $x\in \gamma(v)$, $H(x)=h_v$.

The space $\widetilde{ \mathcal{H}}=\{f: f\circ\pi\in H^1(\mathbb{R}^2)\}$ equipped with the inner product $\langle f,g\rangle_{\widetilde H}:=\langle f\circ\pi,g\circ \pi\rangle_{H^1(\mathbb{R}^2)}$ is a Hilbert space.
Let now $f\in \widetilde {\mathcal{H}}$. For $i\in I$, let $f_i:]l_i,r_i[\to \mathbb{R}$ be defined by $f_i(h)=f(h,i)$. 
Then (see Lemma \ref{lem:wtcH}), for all $i\in I$, $f_i$ is weakly differentiable and we have 
$$\|f\|^2_{\widetilde H}
=\sum_{i\in I} \int_{l_i}^{r_i} (f'_i(h))^2 a_i(h) dh + \sum_{i\in I} \int_{l_i}^{r_i} (f_i(h))^2 T_i(h) dh$$
where $\displaystyle a_i(h)=\oint_{\gamma_{h,i}} |\nabla H| d\ell$ and $\displaystyle T_i(h)=\oint_{\gamma_{h,i}} \frac{d\ell}{|\nabla H|}$.
For $\widetilde x=(h,i)\in E_i$, set $a(\widetilde x)=a_i(h)$ and $T(\widetilde x)=T_i(h)$.

In the following lemma we recall asymptotics given in chapter 8 in \cite{freidlin.wentzell12}. 
\begin{lemma}\label{lem:asymptoticsat}
Let $v\in V$ and $i\in I(v)$. Then $v=(h_v,i)$ and as $h\to h_v$, with $(h,i)\in E_i$, 
\begin{enumerate}
\item
If $d(v)=1$, 
$$a_i(h)\sim a_{i,v} |h-h_v| \quad\text{ and }\quad T_i(h)\sim t_{i,v}$$
\item If $d(v)\ge 2$,
$$a_i(h)\sim a_{i,v}  \quad\text{ and }\quad T_i(h)\sim t_{i,v} |\log|h-h_v||$$
\end{enumerate}
where $a_{i,v}$ and $t_{i,v}$ are two finite positive constants.
\end{lemma}
\begin{proof}
By definition of $v$, either $\gamma(v)$ contains exactly one extremum of $H$ and $d(v)=1$ or $\gamma(v)$ contains a saddle point of $H$ and $d(v)\ge 3$ ($d(v)\neq2$ since the stationary points of $H$ are non-degenerate).

Suppose first that $d(v)=1$. Then $\gamma(v)=\{x^*\}$, with $x^*$ a local extremum of $H$. If $H$ is quadratic at $x^*$ then, the computation of $a(\tilde x)$ and $T(\tilde x)$ can easily be done explicitly and gives the stated asymptotics. In the general case, one can use an asymptotic expansion or the Morse Lemma to conclude.

Suppose now that $d(v)\geq2$. Let $\tilde x=(h,i)\to v$. 
Since $|\nabla H|$ is continuous and that $\gamma(v)$ is compact with finite length, $a(\tilde x)$ converges to a constant $a_{i,v}=\oint_{\gamma_{h_v,i}}|\nabla H|d \ell$ as $\tilde x\to v$. 
The main contribution for $T(h)$ comes when the orbit $\gamma_{h,i}$ is near a stationary point $x^*\in\gamma(v)$ because otherwise $|\nabla H|$ is bounded away from $0$. 
Note that there could be several such stationary points, each is a saddle. If $H$ is quadratic in a small neighborhood $\mathcal{V}$ of $x^*$ then, the computation of an asymptotic can easily be done explicitly and gives $\oint_{\gamma_{h,i}\cap \mathcal{V}}\frac1{|\nabla H|}d \ell\sim C^*|\log|h-h_v||$ for some $C^*>0$ which only depends on the number of times the orbit comes near $x^*$ (one or two times since $x^*$ is not degenerated) and on the Hessian matrix of $H$ at $x^*$. 
In the general case, one can use the Morse Lemma to obtain the same asymptotics. 
Since each saddle point of $\gamma(v)$ gives an asymptotic term of the same order, we obtain the stated result. 
\end{proof}

\begin{proof}[Proof of Proposition \ref{prop:Htilde5}]
Let $f\in\wt\cH$.
Let $v\in\mathcal{V}$ with $d(v)\ge 2$, and $i\in I_v$.
Let $A_i$ be an open set of $]l_i,r_i[$ such that $\bar{A}_i\subset ]l_i,r_i[\cup \{h_v\}$. Note that there is $c>0$ such that $T_i\ge a_i\ge c$ on $A_i$. Thus $f_i$ restricted to $A_i$ belongs to $H^2(A_i)$ and as a consequence $f_i$ can be extended to a continuous function on $]l_i,r_i[\cup\{h_v\}$. 
Set $\Gamma_{i,v}=\partial\Omega_i\cap H^{-1}(h_v)$.
Since $f\circ \pi (x)=f_i\circ H(x)$, we have $f_{|\Gamma_{i,v}}=f_i(h_v)$. 

Let now $j\in I_v$ with $j\ne i$. Suppose first that  $\Gamma_{i,v}\cap\Gamma_{j,v}\ne \emptyset$, then $f_i(h_v)=f_j(h_v)$. If $\Gamma_{i,v}\cap\Gamma_{j,v}= \emptyset$, then there is $(i_0,i_1,\dots,i_\ell)\in I_v^{\ell+1}$ such that $i_0=i$, $i_\ell=j$ and for all $1\le k\le\ell$, $\Gamma_{i_{k-1},v}\cap\Gamma_{i_k,v}\ne \emptyset$ and so $f_i(h_v)=f_{i_1}(h_v)=\dots =f_{i_{\ell-1}}(h_v)=f_j(h_v)$. This proves that $f$ is continuous on $\wt M\setminus\mathcal{V}_1$, and completes the proof of the first implication of the Proposition \ref{prop:Htilde5}. 

Let now $f:M\to\mathbb{R}$ be a measurable map such that $f$ is continuous on $\wt M\setminus\mathcal{V}_1$, $f_i$ is weakly differentiable for all $i\in I$ and $\|f\|_{\wt\cH}<\infty$.
Then, for all $i$, $f\circ\pi\in H^1(\Omega_i)$ and $f\circ \pi$ is continuous on 
$M\setminus \Gamma_1$, where $\Gamma_1$ is the set of local extrema of $H$. Since $\Gamma_1$ is a finite set, $f\circ \pi\in H^1(\mathbb{R}^2)$. This proves that $f\in\wt\cH$.
\end{proof}

\begin{proof}[Proof of Proposition \ref{prop:Ctilde1}]
To complete this proof, it remains to check that Assumption \ref{hyp:tight}-(ii) is satisfied, i.e. that $C_c(\widetilde M)\cap \widetilde{\mathcal{H}}$ is dense in $\widetilde{ \mathcal{H}}$.

Let $i$ and $j$ in $I$ be such that $d(E_i,E_j)=0$.
Denote by $v_{i,j}$ the unique vertex in $V$ such that $v_{i,j}\in \partial E_i\cap \partial E_j$. Set $E_{i,j}:=E_i\cup E_j\cup\{v_{i,j}\}$ and $f_{i,j}=f_{|E_{i,j}}$. 
Let $A_{i,j}$ be an open subset of $E_{i,j}$ such that $v_{i,j}\in A_{i,j}\Subset E_{i,j}$. 
Then (since  there is $c>0$ such that $T(\wt x)\ge a(\wt x)\ge c$ in a neighborhood of $v$),
$f_{i,j}\in H^1(A_{i,j}\cap E_i)$ and $u_{i,j}\in H^1(A_{i,j}\cap E_j)$.
Since $f_{i,j}$ is continuous, we have that $f_{i,j}$ is weakly differentiable.

For each $i$, let $k_i:]l_i,r_i[\to ]0,L_i[$ be defined such that $k_i(l_i)=0$ and $dk_i(h)=\sqrt{\frac{T_i(h)}{a_i(h)}}dh$. Then $k_i$ is 
properly defined (in particular $\displaystyle \int_{l_i+} \sqrt{\frac{T_i(h)}{a_i(h)}}dh <\infty$ and $\displaystyle L_i=\int_{l_i}^{r_i} \sqrt{\frac{T_i(h)}{a_i(h)}}dh <\infty$ only for the unique $i$ for which $\Omega_i$ is unbounded).
Define a new metric graph with edges $E^G_i:=]0,L_i[\times\{i\}$ with the same adjacency rules as for $\widetilde M$. Denote this new metric graph by $G$. By abuse of notation, The set of vertices of $G$ will also be denoted by $V$.

For $v\in V$, let $I(v)$ be the set of all $i$ such that $E^G_i$ is an edge  adjacent to $v$. 
For $i\in I(v)$, set $O^i_v:= \{x \in E^G_i: d(x,v)<\frac{2L_i}{3}\}$ and $O_v:=\{v\}\cup\cup_{i\in I(v)}O^i_v$.
When $d(v)=1$ set $O^0_v=O_v\setminus\{v\}$ and when $d(v)\ge 2$, set $O^0_v=O_v$. Then $O^0_v$ is on open subset of $B_{m,1}$ with $m=d(v)=|I(v)|$. 

Let $G^0$ be the metric graph obtained out of $G$ by taking out of $G$ the vertices of degree $1$, and denote by $V^0$ the set of vertices of $G^0$.
Then it is easy to check that $(O^0_v)_{v\in V^0}$ is a covering of $G^0$ and that there are functions $\varphi_v$, $v\in V^0$, such that $\varphi_v: G^0\to [0,1]$, $\varphi_v=0$ on $G\setminus O^0_v$, $\varphi_v$ restricted to $O^0_v$ is $C^1$, with bounded derivatives, for each $i\in \{1,\dots, d(v)\}$ and such that $\sum_{v\in V^0}\varphi_v=1$.

For $f\in \widetilde{\mathcal{H}}$, let $g:G^0\to \mathbb{R}$ be defined by $g(v)=f(v)$ if $v\in V^0$ and such that $g(k_i(h),i)=f(h,i)$ if $(h,i)\in E_i$.

Let $i\ne j$ such that $d(E^G_i,E^G_j)=0$ and set $E^G_{i,j}:=E^G_i\cup E^G_j\cup\{x_{i,j}\}$, we then have that $g_{i,j}:=g_{|E^G_{i,j}}$ is weakly differentiable. This holds because $f_{i,j}$ is weakly differentiable and $g_{i,j}=f_{i,j}\circ h_{i,j}$, where $h_{i,j}:E^G_{i,j}\to E_{i,j}$ is the continuous function, differentiable on $E^G_i$ and on $E^G_j$, defined by
$h_{i,j}(k,i)=(k_i^{-1}(k),i)$ if $(k,i)\in E^G_i$, $h_{i,j}(k,j)=(k_j^{-1}(k),j)$ if $(k,j)\in E^G_j$ and $h_{i,j}(x_{i,j})=x_{i,j}$.

Then we have that $f\in \widetilde{\mathcal{H}}$ if and only if $g$ defined above belongs to $W(G^0,\omega)$ with $\omega$ a measurable function on $G^0$ such that if $(k,i)\in E^G_i$, $\omega(k,i)=\omega_i(k)=\sqrt{a_iT_i}\circ k^{-1}_i(k)$.
\begin{lemma}
For $\widetilde x\in G^0$ and $v\in V$. We have as $\widetilde x\to v$,
\begin{enumerate}
\item If $d(v)=1$, $\omega(\widetilde x)\sim c_v d(\widetilde x,v)$
\item If $d(v)\ge 2$, $\omega(\widetilde x)\sim c_v\sqrt{|\log(d(\widetilde x,v))|} \to \infty$
\end{enumerate}
where $c_v>0$.
\end{lemma}

\begin{proof}
The case $d(v)=1$ is straightforward. For the case $d(v)\geq 2$, we let $\widetilde x=(k,i)\in G^0$ and assume $v=(0,i)$, then $k=d(\widetilde x,v)$ and $h:=k^{-1}_i(k)$ with $k_i'(h)=\sqrt{\frac{T_i(h)}{a_i(h)}}\sim C_v\sqrt{|\log h|}$ as $h \to 0$. An integration by part yields that $k(h)\sim C h\sqrt{|\log h|}$ and thus $|\log(k)|\sim |\log h|$. Then, since $\omega(\widetilde x)=\sqrt{a_iT_i}(h)\sim C\sqrt{|\log h|}$, we get the result. The same result holds if $v=(L_i,i)$.
\end{proof}

To prove that $C_c(\widetilde M)\cap \widetilde{\mathcal{H}}$ is dense in $\widetilde{\mathcal{H}}$, we will now use Lemma \ref{lem:sob3}.

Fix $R>0$.
For $v\in V$ and $n\ge 1$, set $O_{v,R,n}=O_{v,R}$ if $d(v)\ge 2$ and set  $O_{v,R,n}:=\{x\in O_{v,R}: d(x,v)>n^{-1}\}$ if $d(v)=1$. Then \textit{(i)} of Lemma \ref{lem:sob3} is satisfied.

Let us now check that \textit{(ii)} is satisfied. This has only to be checked for $v\in V$ with $d(v)=1$. Let $E^G_i$ be the unique edge adjacent to $v$. Without loss of generality we will suppose that $v=(0,i)$. We then have that $O_{v,R}\setminus O_{v,R,n}= ]0,n^{-1}[\times \{i\}$.
To prove that $\mathrm{Cap}(O_{v,R}\setminus O_{v,R,n})\to 0$. In this case, fix $A>0$ such that $A<R$ and $A<\frac{2L_i}{3}$ and take (for $n>A^{-1}$), $g_n(k)=1$ if $k\le n^{-1}$ and $g_n(k)=0$ if $k>A$ and $g_n(k)=\frac{\log(A)-\log(k)}{\log(A)-\log(n^{-1})}$ if $n^{-1}\le k\le A$. Then 
$Cap(O_{v,R}\setminus O_{v,R,n})\le \int_{n^{-1}}^A \frac{\omega_i(k)}{k^2(\log(A)+\log(n))^2} dk + \int_{n^{-1}}^A \frac{(\log(A)-\log(k))^2\omega_i(k)}{(\log(A)+\log(n))^2} dk$ which converges to $0$ as $n\to\infty$. This proves \textit{(ii)}

To check \textit{(iii)}, we shall use Lemma \ref{lem:sob1}. for the open sets $O_{v,R,n}\subset B_{m,1}$, with $m=d(v)$.
For each $v,R,n$, we define an extension $\bar{\omega}$ of $\omega_{|O_{v,R,n}}$, defined on $B_{m,1}$. By a slight abuse of notation, we set identify $v$ as $0$ in the definition of $B_{m,1}$ and keep the labels of $I(v)$ to identify the branches of $B_{m,1}$.

If $m=d(v)=1$ then recall that $B_{1,1}=\mathbb{R}$. $O_{v,R,n}$ is an open interval $]l,r[$ and we let $\bar{\omega}(k)=\omega(l)\in ]0,\infty[$ if $k\le l$ and $\bar{\omega}(k)=\omega(r)\in ]0,\infty[$ if $k>r$. The $\bar{\omega}$ is continuous and positive, and it thus belongs to the class $A_2$.

If $m=d(v)\ge 2$, for $i\in I^+(v)$, $O_{v,R,n}\cap E^G_i$ is an open interval $]0,r[$ (with $r=(2L_i/3)\wedge R$) and for $k>r$, we set $\bar{\omega}(k,i)=\omega(r,i)\in ]0,\infty[$ and $\bar{\omega}(v)=\infty$. For $i\in I^-(v)$, $O_{v,R,n}\cap E^G_i$ is an open interval $]l,L_i[$ (with $l=L_i-(2L_i/3)\wedge R$) and we set
$$
\bar{\omega}(k,i)=
\begin{cases}
\omega(L_i-k,i)\in ]0,\infty[&\text{ for $0<k<L_i-l$}\\
\omega(g,i)\in ]0,\infty[&\text{ for $k>L_i-l$}
\end{cases}
$$
and $\bar{\omega}(v)=\infty$. This defines $\bar{\omega}$ on $B_{m,1}$.
It remains to check that for $i\ne j$ in $I(v)$, we have that $\bar{\omega}_{i,j}$ belongs to the class $A_2$. Note that $\lim_{x\to v}\bar{\omega}_{i,j}(x)=\infty$.  Since $\omega_{i,j}(x)\sim c \sqrt{\log(k)}$ as $k:=d(x,v)\to 0$, it is a simple exercise to show that $\bar{\omega}_{i,j}$ belongs to the class $A_2$.

Then Lemma \ref{lem:sob3} can be applied. This proves that  $W(G^0,\omega)=H^1_0(G^0,\omega)$.
Let $f\in \widetilde{\mathcal{H}}$ and $g\in W(G^0,\omega)$ defined as above by $g(k,i)=f(k_i^{-1}(k),i)$ if $(k,i)\in E^G_i$ and $g(v)=f(v)$ if $v\in G^0$.
Since $W(G^0,\omega)=H^1_0(G^0,\omega)$, for all $\epsilon>0$ there is $g_\epsilon\in C_c(G^0)\cap W(G^0,\omega)$ such that $\|g-g_\epsilon\|_{W(G^0,\omega)}<\epsilon$. Set $f_\epsilon(h,i)=g_\epsilon(k_i(h),i)$ if $(h,i)\in E_i$, $f_\epsilon(v)=g_\epsilon(v)$ if $v\in V^0$ and $f_\epsilon(v)=0$ if $v\in V\setminus V_0$. Then $f_\epsilon\in C_c(\widetilde M)\cap \widetilde {\mathcal{H}}$ and $\|f-f_\epsilon\|_{\widetilde{\mathcal{H}}}=\|g-g_\epsilon\|_{W(G^0,\omega)}<\epsilon$. 
\end{proof}

\subsection{Regularity for Section 6}
To prove the regularity we can apply directly Lemma \ref{lem:sob2} since our state space will already be an open subset of $B_{4,2}$. Set $I:=\{1,2,3,4\}$.
Recall that $\widetilde{M}=\cup_{i\in I} C_i$ is constituted of four cones glued along one edge. Set $\wt M_0:=\{y\in \wt M:\; y_1y_2 > 0\}$.

\subsubsection{The space $\wt \cH$.}
Recall that for $t\in ]0,1[$
\begin{align}
K(t)&=\int_0^{\pi/2}\frac{d\theta}{\sqrt{1-t^2\sin^2\theta}}, & E(t)&=\int_0^{\pi/2} \sqrt{1-t^2\sin^2\theta} d\theta\\
\alpha(t)&=1-\frac{E(t)}{K(t)} \qquad\quad\text{ and }&
\lambda(t)&=1-\frac{\alpha(t)(1+t^2)}{2t^2}.
\end{align}
For $y\in \cup_{i\in I}\ring{C}_i$, let $t(y)=\frac{|y_1|\wedge|y_2|}{|y_1|\vee|y_2|}$. Then $t(y)=\frac{y_2}{y_1}$ for $y\in \ring{C}_1\cup \ring{C}_3$ and $t(y)=\frac{y_1}{y_2}$ for $y\in \ring{C}_2\cup \ring{C}_4$. 

For $i\in I$, set $\widetilde{m}_i$ the measure on $C_i$ with density $h_i(y)e^{-W(\|y\|)}$ with respect to Lebesgue measure on $C_i$, where
\begin{align}
h_{1}(y)=h_3(y)&=4\sqrt{2} |y_2|K\left(t(y)\right); & h_2(y)=h_4(y)&=4\sqrt{2} |y_1|K\left(t(y)\right).
\end{align}
For $y\in \cup_{i\in I} \ring{C}_i$, set $a(y):=\nu_y(\nabla \pi \otimes \nabla \pi)$. We have that $\nu_y(x_3^2)=2y_1^2\alpha(t)$ for $y\in \ring{C}_1$, then,
\begin{equation}\label{eq:a.matrix}
\begin{array}{ll}
a(y)&=\begin{pmatrix}
1&0\\
0&1
\end{pmatrix}
+\frac{\alpha(t)}{2t^2}\begin{pmatrix}
-t^2&t\\
t&-1
\end{pmatrix} \text{ if } y\in \ring{C}_1\cup \ring{C}_3 \text{ and } t=\frac{y_2}{y_1};\\
a(y)&=\begin{pmatrix}
1&0\\
0&1
\end{pmatrix}
+\frac{\alpha(t)}{2t^2}\begin{pmatrix}
-1&t\\
t&-t^2
\end{pmatrix} \text{ if } y\in \ring{C}_2\cup \ring{C}_4 \text{ and } t=\frac{y_1}{y_2}.
\end{array}
\end{equation}
The eigenvalues of $a(y)$ are $1$ and $\lambda(t)$, with corresponding eigenvectors $(y_1,y_2)$ and $(-y_2,y_1)$. If $i\in I$ and $y\in \ring{C}_i$, set $a_i(y)=a(y)$.

Let us now collect some asymptotics.
\begin{lemma} \label{lem:asymp}
As $t\to 0^+$,
\begin{align}
&E(t)=\frac{\pi}{2}\left(1-\frac{t^2}{4}+O(t^4)\right);&
&K(t)=\frac{\pi}{2}\left(1+\frac{t^2}{4}+O(t^4)\right); \\
&\alpha(t)=\frac{t^2}{2}\left(1+O(t^2)\right);&
&\lambda(t)=\frac{3}{4}+O(t^2).
\end{align}
As $t\to 1^-$,
\begin{align}
&E(t)=1+o(1); &
&K(t)=-\frac{1}{2}\log(1-t) + O(1);\\
&\alpha(t)=1+\frac{2}{\log(1-t)}(1+o(1));&
&\lambda(t)=\frac{-2}{\log(1-t)}(1+o(1)).
\end{align}
\end{lemma}

\begin{proof}
Asymptotics for $E$  and $K$ are standard, the others are simple consequences.
\end{proof}

Let $f\in\widetilde{\mathcal{H}}$. For $i\in I$, the map $f_i:=f_{|C_i}$ is weakly differentiable on $C_i$ and, using the change of variable formula \eqref{eq:changevariable2}, we have that
\begin{equation}
\|f\|^2_{\widetilde{\mathcal{H}}}=\sum_{i=1}^4\int_{C_i}f^2 d\widetilde{m}_i+\sum_{i=1}^4\int_{C_i}a_i^{k\ell}\partial_k f_i\partial_\ell f_id\widetilde{m}_i.
\end{equation}

Set $\Omega:=]0,+\infty[\times B_{4,1}\subset B_{4,2}$, $\Omega_i:=\Omega\cap E_i$ for $i\in I$ and $\Omega_{i,j}:=\Omega\cap E_{i,j}$ for $i\ne j\in I$. 
Let us remark from Lemma \ref{lem:asymp} that $\displaystyle \mathcal{I}:=\int_0^1 \frac{du}{(1+u^2)\sqrt{\lambda(u)}} <+\infty$. Set $c=\frac{\pi}{2\mathcal{I}}$ and define $\varphi:[0,1]\to[0,\frac{\pi}2]$ by
\begin{equation}\label{eq:varphi}
\varphi(t)=c \int_t^1 \frac{du}{(1+u^2)\sqrt{\lambda(u)}}.
\end{equation}
Define $z:\cup_{i=1}^4\ring{C}_i\to ]0,+\infty[^2$ by
\begin{align}\label{eq:varz}
z(y)&=|y|\left(\cos\varphi (t(y)),\sin \varphi\left(t(y)\right)\right).
\end{align}
Note that for all $i\in I$, $z_i=z_{|\ring{C}_i}$ is one to one. 

For $g:\Omega\to\mathbb{R}$ define $\cF(g):\widetilde{M}_0\to\mathbb{R}$ by $\cF(g)(y)=g(z(y),i)$ if $y\in \mathring{C}_i$ and $\cF(g)(y)=g(y,0)$ if $y\in D^*$.
For $f:\widetilde{M}\to\mathbb{R}$ (or for $f:\widetilde{M}_0\to\mathbb{R}$),  define $\cG(f):\Omega\to\mathbb{R}$ such that if $(z,i)\in \Omega_i$, then  $\cG(f)(z,i)=f_i\circ z_i^{-1}(z)$ and if $z>0$, then $\cG(f)(z,0)=f(z)$.
Then, $\cG\circ \cF(g)=g$ and we have 
that $g\in C_c(\Omega)$ if and only if $\cF(g)\in C_c(\widetilde{M}_0)$.

Define $k:[0,\frac{\pi}{2}[\to \mathbb{R}^+\cup\{\infty\}$ by $k(0)=\infty$ and, for $\phi\in ]0,\frac{\pi}{2}[$ and $t=\varphi^{-1}(\phi)$,
\begin{equation}\label{eq:k}
k(\phi):=\frac{4\sqrt{2}tK(t)\sqrt{\lambda(t)}}{c\sqrt{1+t^2}}.
\end{equation}
For $z=r(\cos \phi,\sin \phi)$ with $(r,\phi)\in ]0,\infty[\times [0,\frac{\pi}2[$, set $\omega(z,i):=re^{-W(r)} k(\phi)$ 
 and
$$
b(z,i):=\begin{pmatrix}
\cos^2\phi+c^2\sin^2\phi & (1-c^2)\sin\phi\cos\phi\\
(1-c^2)\sin\phi\cos\phi & \sin^2\phi+c^2\cos^2\phi
\end{pmatrix}
=R_\phi
\begin{pmatrix}
1 & 0\\
0 & c^2
\end{pmatrix}
R_{-\phi}$$
where $R_\phi$ is the matrix associated to the rotation of angle $\phi$, $$R_\phi=
\begin{pmatrix}
\cos\phi&-\sin\phi\\
\sin\phi&\cos\phi
\end{pmatrix}.$$

Let $f\in \wt\cH$ and $g=\cG(f)$. The change of variable $z=z_i(y)$ on each $C_i$ yields 
\begin{align}\label{eq:cdv}
\|f\|^2_{\widetilde{\mathcal{H}}}=\sum_{i=1}^4\int_{\Omega_i}g^2(z,i)\omega(z,i)dz+\sum_{i=1}^4\int_{\Omega_i}b^{k\ell}(z,i)\partial_k g(z,i)\partial_\ell g(z,i)\omega(z,i)dz.
\end{align}
Recall that the norm on $W^1(\Omega,\omega)$ is given by
\begin{equation}\label{eq:normW}
\|g\|^2_{W^1(\Omega,\omega)}=\sum_{i=1}^4\int_{\Omega_i}(g^2(z,i)+|\nabla g(z,i)|^2)\omega(z,i)dz.
\end{equation}
We thus get that
\begin{equation}\label{eq:equiv}
(1\wedge c^2)\|g\|^2_{W^1(\Omega,\omega)}
\leq \|f\|^2_{\widetilde{\mathcal{H}}}
\leq(1\vee c^2)\|g\|^2_{W^1(\Omega,\omega)}.
\end{equation}

\begin{lemma}
 $f\in\widetilde{\mathcal{H}}$ if and only if $\cG(f)\in W^1(\Omega,\omega)$.
\end{lemma}

\begin{proof} In this proof, we fix $f:\wt M\to\mathbb{R}$ and we set $g:=\cG(f)$.
If $f \in \wt H$, then for each $i$, $f_i$ is weakly differentiable on $\ring{C}_i$ and $g_{|\Omega_i}$ is weakly differentiable on $\Omega_i$. And \eqref{eq:equiv} show that $g_i:=g_{|\Omega_i}\in W^1(\Omega_i,\omega_i)$, where $\omega_i=\omega_{|\Omega_i}$. 
Moreover, if $\mathcal{O}$ be a bounded open set in $\Omega_{i,j}$, then $g_{|\mathcal{O}_i}\in H^1(\mathcal{O}_i)$, $g_{|\mathcal{O}_j}\in H^1(\mathcal{O}_j)$ and coincide on $\mathcal{O}\cap E_0$. Since $\omega$ is bounded from below on $\mathcal{O}$, this implies that $g_{|\mathcal{O}}\in H^1(\mathcal{O})\supset W^1(\mathcal{O},\omega_{i,j})$. We thus have that $g_{|\Omega_{i,j}}\in W^1(\Omega_{i,j},\omega_{i,j})$, which proves that $g\in W^1(\Omega,\omega)$.

On the converse, if $g\in W^1(\Omega,\omega)$, then using that $f\circ \pi=\mathcal{F}(g)\circ\pi$ and \eqref{eq:equiv}, we have $f\circ \pi\in H^1(m)$ and so $f\in \wt H$.
\end{proof}

\begin{lemma}\label{lem:H10W}
$H^1_0(\Omega,\omega)=W^1(\Omega,\omega)$
\end{lemma}
We use Lemma \ref{lem:sob2} to prove this lemma.

This lemma shows the density of $C_c(\widetilde M)\cap\widetilde{\mathcal{H}}$ in $\widetilde{\mathcal{H}}$.
Indeed, $\mathcal{F}:\wt H\to W^1(\Omega,\omega)$ and $\mathcal{G}: W^1(\Omega,\omega)$ are continuous mappings and if $g\in C_c(\Omega) \cap W^1(\Omega,\omega)$ then $\mathcal{F}(g)\in C_c(\wt M)\cap \wt H$.
Therefore, if $f\in \wt H$ and $g=\mathcal{G}(f)$, there is a sequence $g_n\in C_c(\Omega) \cap W^1(\Omega,\omega)$ approximating $g$ in $W^1(\Omega,\omega)$ and so $f_n:=\mathcal{F}(g_n)$ is a sequence in $C_c(\wt M)\cap \wt H$ approximating $f$ in $\wt H$.

\begin{proof}[Proof of Lemma \ref{lem:H10W}]
For $R>0$, recall that $\Omega_R=\{z\in\Omega : \|z\|<R\}$ and define, for $n\geq 1$, $\Omega_{R,n}:=\{(r\cos \phi, r\sin \phi,i) :\; (r,\phi,i)\in ]0,R[\times [0,\frac\pi2-\frac1n[\times I\}$. 

In order to apply Lemma \ref{lem:sob2}, we have to check that \textit{(i)}, \textit{(ii)} and \textit{(iii)} are satisfied.
Item $(i)$ is clearly satisfied.

To prove $(ii)$, we construct a sequence of non negative functions $(g_n)$ on $\Omega$, such that $g_n(z)=1$ in $\Omega_{R,n}\setminus \Omega_R$ and $\lim_{n\to+\infty}\|g_n\|^2_{W^1(\Omega,\omega)}=0$.
Let us first obtain asymptotics for $\omega$.
\begin{lemma}\label{DL:k(0)}
We have, as $\phi\to \frac{\pi}{2}^-$,
$
k(\phi)\sim \frac{3\pi}{8c^2}(\frac{\pi}{2}-\phi).
$
\end{lemma}
\begin{proof}
Using Lemma \ref{lem:asymp} and Equation \eqref{eq:varphi}, we have that $\frac{\pi}{2}-\varphi(t)\sim \frac{2c}{\sqrt{3}}t$ as $t\to0^+$.
Thus, since $\phi=\varphi(t)$, we have $t\sim\frac{\sqrt{3}}{2c}(\frac{\pi}{2}-\phi)$ as $\phi\to\frac{\pi}{2}^-$.
From Lemma \ref{lem:asymp} and Equation \eqref{eq:k}, we obtain $k(\phi)\sim \frac{\pi\sqrt{3}}{4c}t\sim \frac{3\pi}{8c^2}(\frac{\pi}{2}-\phi)$ as $\phi\to \frac{\pi}{2}^-$. 
\end{proof}

Using this Lemma, we can fix $0<\phi_0<\frac\pi2$ and $c_0>0$ such that $0\leq k(\phi)\leq c_0(\frac{\pi}{2}-\phi)$ for all $\phi\in[\phi_0,\frac{\pi}{2}]$. Then we define $f_1:]0,+\infty[\to\mathbb{R}$, $f_{2,n}:[0,\frac{\pi}2[\to\mathbb{R}$ by
\begin{align}
f_1(r)&=
\begin{cases}
1&\text{ for $r<R$}\\
2-\frac{r}{R}&\text{ for $R\leq r<2R$}\\
0&\text{ for $ 2R\leq r$}
\end{cases},\\
f_{2,n}(\phi)&=
\begin{cases}
0&\text{ for $0\leq \phi\leq \phi_0$}\\
\frac{\ln(\frac{\pi}{2}-\phi)-\ln(\frac{\pi}{2}-\phi_0)}{\ln(\frac1n)-\ln(\frac{\pi}{2}-\phi_0)}&\text{ for $\phi_0\leq\phi<\frac{\pi}{2}-\frac1n$}\\
1&\text{ for $\frac{\pi}{2}-\frac1n\le \phi$}
\end{cases}
\end{align}

For $i\in I$ and $z=r(\cos(\phi),\sin(\phi))\in ]0,+\infty[^2$, set $g_n(z,i)=f_1(r)f_{2,n}(\phi)$. We then have
\begin{equation}
\|g_n\|^2_{W^1(\Omega,\omega)}=4\int_{\Omega_1}(g_n^2(z)+|\nabla g_n(z)|^2)\omega(z)dz.
\end{equation}
Since $|\nabla g_n(z)|^2=(f_1')^2(r)f^2_{2,n}(\phi)+\frac1{r^2}f^2_1(r)(f'_{2,n})^2(\phi)$, we have
\begin{align}
\int_{\Omega_1}g_n^2(z)\omega(z)dz
=&\int_0^{2R}f^2_1(r)r^2e^{-W(r)}dr\int_{\phi_0}^{\frac{\pi}{2}}f^2_{2,n}(\phi)k(\phi)d\phi,\\\nonumber
\int_{\Omega_1}|\nabla g_n(z)|^2\omega(z)dz
=&\int_0^{2R}(f'_1)^2(r)r^2e^{-W(r)}dr\int_{\phi_0}^{\frac{\pi}{2}}f^2_{2,n}(\phi)k(\phi)d\phi\\
&+\int_0^{2R}f^2_1(r)e^{-W(r)}dr\int_{\phi_0}^{\frac{\pi}{2}}(f'_{2,n}(\phi))^2k(\phi)d\phi.
\end{align}
The integrals involving $f_1$ are finite and does not depend on $n$. For $f_{2,n}$, we have
\begin{align}
\int_{\phi_0}^{\frac{\pi}{2}}f^2_{2,n}(\phi)k(\phi)d\phi
&\leq c_0\int_{\phi_0}^{\frac{\pi}{2}}f^2_{2,n}(\phi)\left(\frac{\pi}{2}-\phi\right) d\phi\\
&= c_0\left(\int_{\phi_0}^{\frac{\pi}{2}-\frac{1}{n}}f^2_{2,n}(\phi)\left(\frac{\pi}{2}-\phi\right) d\phi  + \int_{\frac{\pi}{2}-\frac{1}{n}}^{\frac{\pi}{2}}\left(\frac{\pi}{2}-\phi\right) d\phi\right)\\
&\leq c_0\left(\int_{\phi_0}^{\frac{\pi}{2}}\left(\frac{\ln(\frac{\pi}{2}-\phi)-\ln(\frac{\pi}{2}-\phi_0)}{\ln(\frac{1}{n})-\ln(\frac{\pi}{2}-\phi_0)}\right)^2 \left(\frac{\pi}{2}-\phi\right) d\phi + \frac1{2n^2}\right).
\end{align}
Thus, $\lim_{n\to+\infty}\int_{\phi_0}^{\frac{\pi}{2}}f^2_{2,n}(\phi)k(\phi)d\phi=0$. We also have
\begin{align}
\int_{\phi_0}^{\frac{\pi}{2}}(f'_{2,n}(\phi))^2k(\phi)d\phi
&\leq c_0\int_{\phi_0}^{\frac{\pi}{2}-\frac{1}{n}}(f'_{2,n})^2(\phi)\left(\frac{\pi}{2}-\phi\right) d\phi\\
&=\frac{c_0}{(\ln(n^{-1})-\ln(\frac{\pi}{2}-\phi_0))^2}\int_{\phi_0}^{\frac{\pi}{2}-\frac{1}{n}}\frac{d\phi}{\frac{\pi}{2}-\phi}\\
&\leq\frac{c_0}{|\ln(n^{-1})-\ln(\frac{\pi}{2}-\phi_0)|}.
\end{align}
Thus, $\lim_{n\to+\infty}
\int_{\phi_0}^{\frac{\pi}{2}}(f'_{2,n}(\phi))^2k(\phi)d\phi=0$ and then $\lim_{n\to+\infty}\|g_n\|^2_{W^1(\Omega,\omega)}=0$. Item $(ii)$ is proven.

To prove $(iii)$, we apply Lemma \ref{lem:sob1}. Let us first define an extension $\bar{\omega}_{n,R}$ to $B_{4,2}$ of $\omega$ which coincide with $\omega$ on $\Omega_{R,n}$. 
We define $k_n:[-\pi,\pi]\to \mathbb{R}^+\cup\{\infty\}$ by 
\begin{equation}
k_{n}(\phi)=
\begin{cases}
k(|\phi|)&\text{ for } \phi\in [-\frac\pi2+\frac1n,\frac\pi2-\frac1n]\\
k(\frac\pi2-\frac1n)&\text{ elsewhere,}
\end{cases}
\end{equation}
and define the weight $\bar{\omega}_{n,R}:B_{4,2}\to \mathbb{R}^+$ by $\bar{\omega}_{n,R}(z,i):=re^{-W(r\wedge R)}k_{n}(\phi)$ where $z=(r\cos\phi, r\sin\phi)$ with $(r,\phi,i)\in \mathbb{R}^+\times [-\pi,\pi]\times I$. 
Then, $\bar{\omega}_{n,R}$ coincide with $\omega$ on $\Omega_{R,n}$.
Note that for $i\neq j$, the restriction of $\bar{\omega}_{n,R}$ to $E_{i,j}$ (identified with $\mathbb{R}^2$) does not depend on $i$ and $j$. Denote this common measure $\omega_{n,R}$ and set $\omega_n:\mathbb{R}^2\to\mathbb{R}^+\cup\{+\infty\}$, defined by $\omega_{n}(z)=r k_n(\phi)$ where $z=r(\cos\phi,\sin\phi)\in\mathbb{R}^2$. Since, $0<\inf_{r>0} e^{-W(r\wedge R)} \le \sup_{r>0}e^{W(r\wedge R)}$, $\omega_{n,R}$ belongs to the class $A_2$ if and only if $\omega_{n}$ belongs to the class $A_2$. 

Define for $\rho>0$, $z\in \mathbb{R}^2$,
$$C(z,\rho)=\frac1{\pi^2 \rho^4}\int_{B(z,\rho)}\omega_{n}\int_{B(z,\rho)}\frac1{\omega_{n}}.$$
where $B(z,\rho)$ is the ball at center $z$ and radius $\rho$.

Let us first obtain asymptotics for $\omega_n$. 
\begin{lemma}\label{DL:k(phi)}
We have, as $\phi\to 0+$,
$
k(\phi)\sim \frac1{2c}\left|\log\left(\phi\right)\right|^{1/2}.
$
\end{lemma}
\begin{proof}
For $\phi>0$ close to $0$ and $t=\varphi^{-1}(\phi)$, set $\delta:=1-t$. We have 
\begin{equation}
\phi=c\int_{1-\delta}^{1}\frac{du}{(1+u^2)\sqrt{\lambda(u)}}=c\int_{0}^{\delta}\frac{du}{(1+(1-u)^2)\sqrt{\lambda(1-u)}}
\end{equation}
As $u\to 0^+$, using Lemma \ref{lem:asymp}, we get $\sqrt{\lambda(1-u)}\sim \frac{\sqrt{2}}{\sqrt{|\ln(u)|}}$, and thus as $\phi\to 0^+$, by integration by parts
\begin{align}
\phi\sim \frac{c}{2\sqrt{2}}\int_0^\delta\sqrt{|\log(u)|}du\sim\frac{c}{2\sqrt{2}}\delta\sqrt{|\log(\delta)|}.
\end{align}
This entails that $\log(\delta)\sim\log(\phi).$
It is straightforward, using Lemma \ref{lem:asymp} again, that $K(t)\sqrt{\lambda(t)}\sim\frac1{\sqrt{2}}\sqrt{|\log(\delta)|}$. Then we obtain that  $k(\phi)\sim\frac1{2c}\sqrt{|\log(\delta)|} \sim \frac1{2c}\sqrt{|\log(\phi)|}$.
\end{proof}

Let us now prove the condition A2 for $\omega_n$.
\begin{lemma}\label{lem:A2}
$\omega_n$ satisfies the $A_2$-Muckenhoupt condition: $\sup_{z,\rho}C(z,\rho)<+\infty$.
\end{lemma}

\begin{proof}[Proof of Lemma \ref{lem:A2}]
Note that $k_n$ is a continuous positive function on $[-\pi,\pi]\setminus\{0\}$. 
Lemma \ref{DL:k(phi)} shows that $k_{n}(\phi)\sim \frac1{2c}\sqrt{\left|\log|\phi|\right|}(1+o(1))$ as $\phi\to 0$. 
Thus the integrals $\int_{-\pi}^{\pi}k_{n}(\phi)d\phi$ and $\int_{-\pi}^{\pi}k_{n}(\phi)^{-1}d\phi$ are finite and $\omega_n$ and $\omega_n^{-1}$ are locally integrable on $\mathbb{R}^2$. Therefore $(z,\rho)\mapsto C(z,\rho)$ is a continuous function.
Note also that for $\lambda >0$ and $z\in\mathbb{R}^2$, $\omega_n(\lambda z)=\lambda\omega_n(z)$, then $C(\lambda z,\lambda \rho)=C(z,\rho)$. 
This leaves two cases to investigate: $z=0$, $\rho=1$ or $|z|=1$, $\rho>0$.

For $z=0$, $\rho=1$, we have 
\begin{align}
C(0,1)=\frac1{\pi^2}\int_{B(0,1)}\omega_{n}\int_{B(0,1)}\frac1{\omega_{n}}
&
=\frac1{3\pi^2}\int_{-\pi}^{\pi}k_{n}\int_{-\pi}^{\pi}\frac1{k_{n}}<+\infty.
\end{align}
Obviously, $C(0,\rho)=C(0,1)=C_n<+\infty$ does not depend on $\rho$.

For the second case, fix $\rho>0$ and $z=(\cos\phi,\sin\phi)$, with $\phi\in [-\pi,\pi]$.
If $\rho\geq\frac{1}{2}$, 
since $B(z,\rho)\subset B(0,\rho+1)$, we get $C(z,\rho)\leq\frac{(\rho+1)^4}{\rho^4}C_n\le 3^4 C_n$.
If $\rho<\frac12$, we set $\theta=\arcsin(\rho)\in ]0,\frac{\pi}{6}[$. 
We have $B(z,\rho)\subset T(z,\rho):=\{z=r(\cos\psi,\sin\psi); \;(r,\psi)\in [1-\rho,1+\rho]\times [\phi-\theta,\phi+\theta]\}$.
We obtain
\begin{align}
\int_{B(z,\rho)}\omega_{n}
\leq\int_{T(z,\rho)}\omega_{n}
&=\int_{1-\rho}^{1+\rho} r^2dr\int_{\phi-\theta}^{\phi+\theta}k_{n}
=\frac23(3\rho+\rho^3)\int_{\phi-\theta}^{\phi+\theta}k_{n},\\
\int_{B(z,\rho)}\frac1{\omega_{n}}
\leq\int_{T(z,\rho)}\frac1{\omega_{n}}
&=2\rho\int_{\phi-\theta}^{\phi+\theta}\frac1{k_{n}}.
\end{align}
Then we get (setting $c_0=\frac{13}{3\pi^2}$)
\begin{equation}\label{eq:C'}
C(z,\rho)\leq \frac{4(3+\rho^2)}{3\pi^2\rho^2}\int_{\phi-\theta}^{\phi+\theta}k_{n}\int_{\phi-\theta}^{\phi+\theta}\frac1{k_{n}}
\leq D(\phi,\theta):=\frac{c_0}{\sin(\theta)^2}\int_{\phi-\theta}^{\phi+\theta}k_{n}\int_{\phi-\theta}^{\phi+\theta}\frac1{k_{n}}.
\end{equation}

The function $(\phi,\theta)\mapsto D(\phi,\theta)$ is continuous on $[-\pi,\pi]\times]0,\frac\pi6]$. 
Let us now show that $D$ is bounded in a neighborhood of $\theta=0+$. 
Note that
\begin{align}
D(\phi,\theta)
\leq\frac{4c_0\theta^2}{\sin(\theta)^2}
\frac{\sup_{[\phi-\theta,\phi+\theta]}k_{n}}{\inf_{[\phi-\theta,\phi+\theta]}k_{n}}
\leq 8c_0\frac{\sup_{[\phi-\theta,\phi+\theta]}k_{n}}{\inf_{[\phi-\theta,\phi+\theta]}k_{n}}
\end{align}
since $\frac{\theta^2}{\sin(\theta)^2}<2$ for all $\theta\in[0,\frac{\pi}6]$.

Fix $\theta_0\in ]0,\frac{\pi}{6}[$. We have $\sup_{(\phi,\theta)\in [-\pi,\pi]\times [\theta_0,\frac\pi6]} D(\phi,\theta)<\infty$.
From Lemma \ref{DL:k(phi)}, we get that there are  $0<c_1<C_1$ and $0<c_2<C_2$ such that
\begin{align}\label{eq:kn2}
c_1\leq &k_n(\phi)\leq C_1, \text{ for }\phi\in[-\pi,\pi]\setminus [-\theta_0,\theta_0]\\
c_2\sqrt{|\log(|\phi[)|}\leq &k_n(\phi)=k(\phi)\leq C_2\sqrt{|\log(|\phi|)|}, \text{ for }\phi\in[-3\theta_0,3\theta_0].
\end{align}

Let $\theta<\theta_0$. Without loss of generality, we assume that  $\phi\in [0,\pi]$.
If $\phi-\theta\ge \theta_0$, we get $[\phi-\theta,\phi+\theta]\subset [\theta_0,\pi+\theta_0]$ and we get, using \eqref{eq:kn2},
$D(\phi,\theta)\leq \frac{4C_1}{c_1}$.

If  $\phi-\theta<\theta_0$, then $[\phi-\theta,\phi+\theta]\subset [-\theta_0,3\theta_0]$, and we consider again two subcases.
First subcase: $\phi\geq 2\theta$, let $u:=\phi+\theta\geq3\theta$,
\begin{align}
D(\phi,\theta) &\leq \frac{8c_0C_2}{c_2} \left(\frac{\log(u-2\theta)}{\log(u)}\right)^{\frac12}
\leq \frac{8c_0C_2}{c_2} \left(1+ \frac{\log(1-\frac{2\theta}{u})}{\log(u)}\right)^{\frac12}\\
&\leq \frac{8c_0C_2}{c_2} \left(1+ \frac{\log(1-\frac{2}{3})}{\log(u)}\right)^{\frac12}
\leq \frac{8c_0C_2}{c_2} \left(1+ \frac{\log(\frac1{3})}{\log(3\theta_0)}\right)^{\frac12}.
\end{align}

In the second subcase we have $\phi<2\theta$. We start back from \eqref{eq:C'}:
\begin{align}
D(\phi,\theta)\leq\frac{c_0}{\sin(\theta)^2}\int_0^{\phi+\theta}k_{n}\int_0^{\phi+\theta}\frac1{k_{n}}
\leq C_2(\theta):=\frac{c_0}{\sin(\theta)^2}\int_0^{3\theta}k_{n}\int_0^{3\theta}\frac1{k_{n}}.
\end{align}
$C_2$ is a continuous function on $]0,\theta_0]$ and as $\theta\to 0^+$, we get (using Lemma \ref{DL:k(phi)}
\begin{align}
\int_0^{3\theta}k_n(\phi)d\phi
\sim\frac{3\theta}{2c}\sqrt{-\log(3\theta)}\\
\int_0^{3\theta}k_n(\phi)^{-1}d\phi
\sim 6c\theta\frac1{\sqrt{-\log(3\theta)}}.
\end{align}
Thus $\lim_{\theta\to 0^+} C_2(\theta)=18.$ Finally, this gives us the result.
\end{proof}

With Lemma \ref{lem:A2}, conditions of Lemma \ref{lem:sob1} are satisfied and condition $(iii)$ of Lemma \ref{lem:sob2} holds. Therefore, $W^1(\Omega,\omega)=H^1_0(\Omega,\omega)$.
\end{proof}

\noindent\textbf{Acknowledgments} The authors are part of the LABEX MME-DII.
This research has been conducted within the FP2M federation (CNRS FR 2036).

\end{document}